\documentclass[11pt]{amsart}
\usepackage[left=1in,right=1in,bottom=1in,top=1in]{geometry}
\usepackage{amsmath}
\usepackage{amsfonts}
\usepackage{amssymb}
\usepackage{amsthm}
\usepackage{young}
\usepackage[vcentermath]{youngtab}
\usepackage{tikz}
\usetikzlibrary{shapes.misc}

\tikzset{cross/.style={cross out, draw=black, minimum size=2.5*(#1-\pgflinewidth), inner sep=2pt, outer sep=0.5pt},
	cross/.default={1pt}}
\usepackage{xcolor}
\usetikzlibrary{calc,arrows}

\usepackage{algorithm}
\usepackage[noend]{algpseudocode}
\usepackage{amsmath, amssymb, amsthm}
\usepackage{blindtext}
\usepackage[pdftex,bookmarks=true]{hyperref}
\usepackage{cleveref}
\usepackage{comment}
\usepackage{enumerate}
\usepackage{fullpage}
\usepackage{mathtools}
\usepackage{subcaption}

\usepackage{enumitem}
\setlist[enumerate]{itemsep=.3mm}
\setlist[itemize]{itemsep=.3mm}

\newcommand{\E}{\mathbb E}
\renewcommand{\P}{\mathbb P}
\newcommand{\Bin}{\textbf{Bin}}
\newcommand{\ep}{\varepsilon}
\renewcommand{\deg}{\text{deg}}
\newcommand{\one}{\textnormal{\textbf{1}}}
\newcommand{\zero}{\textnormal{\textbf{0}}}

\newcommand{\cp}{\textsf{\small{\textbf{CP}}}}
\newcommand{\tcp}{\widetilde{\textsf{\small{\textbf{CP}}}}}
\newcommand{\ddp}{\textsf{\small{\textbf{DP}}}}
\newcommand{\gw}{\textsf{\footnotesize{\textbf{GW}}}}
\newcommand{\gwc}{\textsf{\footnotesize{\textbf{GWC}}}}
\newcommand{\egw}{\textsf{\footnotesize{\textbf{EGW}}}}
\newcommand{\Exp}{\textnormal{Exp}}
\newcommand{\mX}{\mathcal X}

\newtheorem{theorem}{Theorem}[section]
\newtheorem{lemma}[theorem]{Lemma}

\newtheorem{proposition}[theorem]{Proposition}

\newtheorem{claim}[theorem]{Claim}

\newtheorem{thm}{Theorem}[section]

\newtheorem{conj}[thm]{Conjecture}
\newtheorem{cor}[thm]{Corollary}

\theoremstyle{definition}
\newtheorem{definition}[theorem]{Definition}
\newtheorem{remark}[theorem]{Remark}
\newtheorem*{remark-non}{Remark}

\begin{document}
	\title{\textbf{Survival and extinction of epidemics on random graphs with general degrees}}
	
	%\author{Shankar Bhamidi\footnote{Supported by  NSF grants DMS-1613072, DMS-1606839 and ARO grant W911NF-17-1-0010.}, Danny Nam\footnote{Supported by Samsung scholarship}, Oanh Nguyen and Allan Sly\footnote{Supported by NSF grant DMS-1352013, Simons Investigator grant and a MacArthur Fellowship.}
	%}
	
	\author{Shankar Bhamidi}
	\address{Department of Statistics and Operations Research, University of North Carolina, Chapel Hill, NC 27599}
	\email{bhamidi@email.unc.edu}
	\thanks{S. Bhamidi supported by  NSF grants DMS-1613072, DMS-1606839 and ARO grant W911NF-17-1-0010}
	
	\author{Danny Nam}
	\address{Department of Mathematics, Princeton University, Princeton, NJ 08544}
	\email{dhnam@princeton.edu}
	\thanks{D. Nam supported by Samsung scholarship}
	
	\author{Oanh Nguyen}
	\address{Department of Mathematics, Princeton University, Princeton, NJ 08544}
	\email{onguyen@princeton.edu}
	
	\author{Allan Sly}
	\address{Department of Mathematics, Princeton University, Princeton, NJ 08544}
	\email{asly@princeton.edu}
	\thanks{A. Sly supported by NSF grant DMS-1352013, Simons Investigator grant and a MacArthur Fellowship}
	\maketitle
	
	\begin{abstract}
		%The contact process is an epidemic model on a graph where each vertex is either infected or healthy, and an infected vertex gets healed at rate $1$ while it passes its disease to each of its neighbors at rate $\lambda$.
		In this paper, we establish the necessary and sufficient criterion for the contact process on Galton-Watson trees (resp. random graphs) to exhibit the phase of extinction (resp. short survival). We prove that the survival threshold $\lambda_1$ for a Galton-Watson tree is strictly positive if and only if its offspring distribution $\xi$ has an exponential tail, i.e., $\E e^{c \xi}<\infty $ for some $c>0$, settling a conjecture by Huang and Durrett \cite{hd18}. On the random graph with degree distribution $\mu$, we show that if $\mu$ has an exponential tail, then for small enough $\lambda$ the contact process with the all-infected initial condition survives for $n^{1+o(1)}$-time \textsf{whp} (short survival), while for large enough $\lambda$ it runs over $e^{\Theta(n)}$-time \textsf{whp} (long survival). When $\mu$ is subexponential, we prove that the contact process \textsf{whp} displays long survival for any fixed $\lambda>0$.
	\end{abstract}

\setcounter{tocdepth}{1}
\tableofcontents

\section{Introduction}

The contact process is a model of epidemics on networks introduced by Harris in 1974 \cite{harris74}.  Its transitions are given as follows:
\begin{itemize}
	\item Each vertex is either \textit{infected} or \textit{healthy}.
	
	\item Each infected vertex infects each of its neighbors independently at rate $\lambda$, and it is healed at rate $1$ independently of all the infections.
	
	\item Infection and recovery events in the process happen independently from vertex to vertex.
\end{itemize}

The phase diagrams of the contact processes on $\mathbb{Z}^d$ and on $\mathbb{T}_d$, the infinite $d$-ary tree, are well-understood. In particular,  the contact process on an infinite tree has drawn particular interest as it has two distinct phase transitions. In a series of beautiful works \cite{p92, l96, s96}, it was shown that the contact process on $\mathbb{T}_d$ for $d\geq 2$, with an initial infection at the root, has two different thresholds $0<\lambda_1<\lambda_2$ such that

\begin{itemize}
	\item (\textit{Extinction}) For $\lambda <\lambda_1$, the infection becomes extinct almost surely;
	
	\item (\textit{Weak survival}) For $\lambda\in (\lambda_1, \lambda_2)$, the infection survives with positive probability, but the root is infected finitely many times almost surely;
	
	\item (\textit{Strong survival}) For $\lambda >\lambda_2$, the infection survives and the root gets infected infinitely many times with positive probability.
\end{itemize}

A natural interest is then  to study the phase diagram of the contact process on Galton-Watson trees. In this paper, we establish the necessary and sufficient criterion for  $\lambda_1>0$. In particular, we provide the first known result for extinction in  Galton-Watson trees with unbounded offspring distribution.

\begin{thm}\label{thm:gw main}
	Consider the contact process on the Galton-Watson tree with offspring distribution $\xi$, and suppose that only the root of the tree is initially infected. If $\xi$ has an exponential tail, i.e., $\E e^{c\xi} <\infty$ for some $c>0$, then there exists $\lambda_0=\lambda_0(\xi)>0$ such that for all $\lambda<\lambda_0$, the process dies out almost surely.
\end{thm}

\noindent Recently, Huang and Durrett \cite{hd18} proved that on Galton-Watson trees, $\lambda_2=0$ if the offspring distribution $\xi$ is subexponential, i.e., $\E e^{c\xi}=\infty$ for all $c>0$. Combining Theorem \ref{thm:gw main} with their result, we have the complete characterization on the existence of extinction on Galton-Watson trees. Moreover, Theorem \ref{thm:gw main} establishes a stronger version of the following conjecture by Huang and Durrett:

\begin{conj}[\cite{hd18}]
	Suppose that $\P(\xi \geq k) = (1-p)^k$ for all $k$ larger than some constant $K$, and consider the contact process on the Galton-Watson tree with offspring distribution $\xi$. Then $\lambda_2$, the weak-strong survival threshold, is strictly positive.
\end{conj}

The challenge in understanding the infection time on trees with unbounded degree distributions is that the infection persists for a long time around high degree vertices as there are many neighbors from which it can be reinfected.  Indeed, it was shown in~\cite{berger2005spread} that the infection will last time $e^{c_\lambda d}$ in a neighborhood of a vertex of degree $d$ with positive probability for some $c_\lambda>0$.  Thus exponential tails on the degree distribution are needed for there to be few enough high degree vertices in the tree for extinction to be certain.

% Theorem \ref{thm:gw main} proves that $\lambda_1>0$ for all Galton-Watson trees with offspring distributions having exponential tail, and this implies that $\lambda_2>0$. Note that the geometric distribution proposed in the conjecture is a special case of what is covered in Theorem \ref{thm:gw main}.

The next object of interest is the contact process on random graphs. For the contact process on the  Erd\H{o}s-R\'enyi random graph $\mathcal{G}_{n, d/n}$,  no rigorous results were known regarding its phase diagram---whether it shows short or long survival, or both. In this work, we prove that on $\mathcal{G}_{n, d/n}$, the contact process exhibits two different phases depending on $\lambda$, as a consequence of an analogous criterion on more general random graphs.

We focus on studying the contact process on the random graph with degree distribution $\mu$, which we  denote by $G \sim \mathcal{G}(n,\mu)$ (definitions given in Section  \ref{subsec: rg basic}).
For the contact process on $G\sim \mathcal{G}(n,\mu)$, our main goal is to study how long  the process survives in terms of the size of the graph.  The second result of this paper establishes the necessary and sufficient criterion for the contact process on $G\sim \mathcal{G}(n,\mu)$ to display the short survival phase. We assume throughout that $\mu$ satisfies
\begin{equation}\label{eq:condition on mu}
\E_{D\sim\mu} D(D-2) >0 \quad \textnormal{and} \quad \E_{D\sim\mu}D^2 <\infty,
\end{equation}
in order to ensure the existence of the giant component and take advantage of the \textit{configuration model}. For details on (\ref{eq:condition on mu}), see section  \ref{subsec: rg basic}.

\begin{thm}\label{thm: exp tail main}
	Suppose that $\mu$ satisfies (\ref{eq:condition on mu}) and there exists some constant $c>0$ such that $\E_{D\sim \mu} e^{cD}<\infty$. Consider the contact process on $G\sim \mathcal{G}(n,\mu)$ where all vertices are initially infected. Then there exist constants $0<\underline{\lambda}(\mu) \le\overline{\lambda}(\mu) <\infty$ such that the following hold:
	\begin{enumerate}
		\item [(1)] For all $\lambda <\underline{\lambda}$, the survival time of the process is at most $n^{1+ o(1)}$-time \textsf{whp}.
		
		\item [(2)] For all $\lambda> \overline{\lambda}$, the survival time of the process is $e^{\Theta(n)}$-time \textsf{whp}.
	\end{enumerate}
\end{thm}

\begin{thm}\label{thm: subexp tail main}
	Suppose that $\mu$ satisfies (\ref{eq:condition on mu}) and $\E_{D\sim \mu} e^{cD}=\infty$ for all $c>0$.  Consider the contact process on $G\sim \mathcal{G}(n,\mu)$ where all vertices are initially infected. Then for any fixed $\lambda>0$, the survival time of the process is $e^{\Theta(n)}$-time \textsf{whp}.
\end{thm}

\begin{remark}
	In the statements of Theorems \ref{thm: exp tail main} and \ref{thm: subexp tail main} (and Corollary \ref{cor: ER main} below as well), the notion \textsf{whp} covers the randomness coming from both the choice of graph $G$ and the contact process. Therefore, they should be understood as
	\begin{center}
		There exists an event $\mathcal{A}$ over the choice of $G$ which occurs \textsf{whp}, \\such that the statement holds \textsf{whp} over the contact process on $G$ given $G\in \mathcal{A}$.
	\end{center}
\end{remark}

For the case of  Erd\H{o}s-R\'enyi random graphs, which are contiguous to $\mathcal{G}(n,\mu)$ with $\mu=\text{Pois}(d)$ (see Section \ref{subsec: rg basic} for details), we can show that the contact process on $\mathcal{G}_{n, d/n}$ exhibits two different phases, as a consequence of Theorem \ref{thm: exp tail main}.

\begin{cor}\label{cor: ER main}
	For any fixed $d>1$, consider the contact process on $G\sim \mathcal{G}_{n,d/n}$ where all vertices are initially infected. Then there exist constants $0<\underline{\lambda}(d) \le \overline{\lambda}(d) <\infty$ such that the following holds:
	\begin{enumerate}
		\item [(1)] For all $\lambda <\underline{\lambda}$, the survival time of the process is at most $n^{1+ o(1)}$-time \textsf{whp}.
		
		\item [(2)] For all $\lambda> \overline{\lambda}$, the survival time of the process is $e^{\Theta(n)}$-time \textsf{whp}.
	\end{enumerate}
\end{cor}

\begin{remark}\label{rmk:onevertex}
	One might be interested in studying the contact process on $G \sim \mathcal{G}(n,\mu)$, with the initial condition such that only a single vertex is infected. When  a uniformly random vertex in $G$ is infected initially while all the other ones are healthy,  we will later see that Theorems \ref{thm: exp tail main} and \ref{thm: subexp tail main} continue to hold, if we change ``\textsf{whp}" to ``with positive probability" at the end of the statements of Theorems \ref{thm: exp tail main}-(2) and \ref{thm: subexp tail main}. To be precise, by ``with positive probability", we mean \textsf{whp} over the choice of $G\sim \mathcal{G}(n,\mu)$, with positive probability over the choice of the initially infected vertex $v$ and with positive probability over the contact process.  Proofs are given in Remark \ref{rmk:onevertex:ex} for exponential distributions and Remark \ref{rmk:onevertex:subex} for subexponential distributions.
\end{remark}

To sum up, we establish a ``universality" criterion for the contact process on Galton-Watson trees (resp. random graphs with given degree distributions), on the existence of the phase of extinction (resp. short survival).  Our methods do not give sharp estimates on the critical value and it is an interesting open problem to determine the location of the phase transition. We also believe that the two critical values in Theorems \ref{thm:gw main} and \ref{thm: exp tail main} coincide. Precisely, we conjecture that  $\lambda_1(\gw(\mu') )  = \lambda_c (\mathcal{G}(n,\mu))$, where
\begin{itemize}
	\item $\lambda_1(\gw(\mu'))$ is the death-survival threshold of the Galton-Watson tree with offspring distribution $\mu'$, the size-biased distribution of $\mu$ (see Section \ref{subsec: rg basic} for details);
	
	\item $\lambda_c (\mathcal{G}(n,\mu))$ is the short-long survival threshold of $\mathcal{G}(n,\mu)$;
	
\end{itemize}

\subsection{Related works}

In \cite{harris74}, Harris first introduced the contact process on $\mathbb{Z}^d$ and showed that the death-survival threshold $\lambda_c(\mathbb{Z}^d)$ satisfies $0<\lambda_c( \mathbb{Z}^d) <\infty$ for any $d$. Building upon this work, the model on $\mathbb{Z}^d$ has been studied intensively and we refer to Liggett~\cite{liggett:sis} for a survey of results. Pemantle \cite{p92} studied the contact process on the infinite $d$-ary tree $\mathbb{T}_d$ and showed that it exhibits three different phases---extinction, weak survival and strong survival---for $d\geq 3$. This result was later generalized by Liggett \cite{l96} for the case $d=2$. Stacey~\cite{s96} gave a shorter proof that applies for any $d\geq 2$.

Less is known for the contact process on general Galton-Watson trees. Recently, Huang and Durrett \cite{hd18} proved that on Galton-Watson trees, $\lambda_2=0$ if the offspring distribution $\xi$ is subexponential. Along with Theorem \ref{thm:gw main}, we now have the complete characterization of the existence of extinction in the contact process on Galton-Watson trees.

There has been considerable work studying the phase transitions of survival times on large finite graphs. Stacey \cite{s01} and Cranston et. al. \cite{cmmv14} studied the contact process on the $d$-ary tree $\mathbb{T}_d^h$ of depth $h$ starting from the all-infected state, and their results show that the survival time $T_h$, as $h\rightarrow \infty$, satisfies (i) $T_h/h \rightarrow \gamma_1$ in probability if $\lambda < \lambda_2(\mathbb{T}_d)$; (ii) $|\mathbb{T}_d^h|^{-1}{\log \E T_h} \rightarrow \gamma_2$ in probability and $T_h/\E T_h \overset{d}{\rightarrow} \Exp(1)$ if $\lambda >\lambda_2(\mathbb{T}_d)$, where $\gamma_1,\gamma_2$ are constants depending on $d,\lambda$.  In \cite{dl88, ds88, m93}, similar results were established for the case of the lattice cube $\{1,\ldots,n \}^d$.

Recently in work of  Mourrat and Valesin~\cite{mv16} and Lalley and Su~\cite{ls17}, it was shown that for any $d\geq 3$, the contact process on the random $d$-regular graph, whose initial configuration is the all-infected state, exhibits the following phase transition:

\begin{itemize}
	\item (\textit{Short survival}) For $\lambda <\lambda_1(\mathbb{T}_d)$, it survives for $O(\log n)$-time \textsf{whp}.
	
	\item (\textit{Long survival}) For $\lambda >\lambda_1(\mathbb{T}_d)$, it survives for $e^{\Theta(n)}$-time \textsf{whp}.
\end{itemize}

\noindent Moreover in \cite{ls17}, a ``cutoff phenomenon" of the fraction of infected vertices was established. In \cite{mv16}, the same result as above is proven for   $G\sim \mathcal{G}(n,\mu)$ with bounded $\mu$ (Theorems 1.3 and 1.4).  For an  unbounded degree distribution $\mu$, Chatterjee and Durrett  \cite{cd09} proved that if $\mu$ obeys a power law, then the contact process always displays long survival for any $\lambda>0$, though their survival time was slightly weaker than exponential ($e^{n^{1-\delta}}$ for any $\delta>0$). This result was later generalized in \cite{mmvy16} to an exponential survival.  Our Theorems \ref{thm: exp tail main} and \ref{thm: subexp tail main} extend the aforementioned results to any general $\mu$. In \cite{sv2017}, a long survival on general graphs for $\lambda > \lambda_c(\mathbb{Z})$ was settled, with survival time at least $\exp(|G|/\{\log |G| \}^\kappa )$ for any $\kappa>1$. \cite{cd17} studied the contact process under similar settings as Theorem \ref{thm: exp tail main}-(2) and Corollary \ref{cor: ER main}-(2) with additional assumptions on the degree distribution and showed that the expected survival time is exponentially large in $n$.

On random graphs with power-law degree distributions, metastability properties on the size of infected vertices were studied in \cite{mvy13, cs15}. For other types of random graphs, recently in~\cite{ms16} it was shown that the contact process on random geometric graphs exhibits both short and long survival.  

\subsection{Main techniques}\label{subsec:maintec}
We  sketch the ideas in the paper before giving the full proofs. The analysis of the subcritical contact process (i.e., extinction and short survival) relies on three main ideas  which we now describe. Here, we assume that the offspring distribution of the Galton-Watson trees and the degree distribution of the random graphs have exponential tails.

\vspace{2mm}
\noindent $\blacktriangleright$ \textit{Modified process: preventing recoveries at the root.}  ~One main difficulty  in studying the contact process on Galton-Watson trees  comes from complicated dependencies inside the given tree. To overcome this obstacle, we consider the following modification of the process:
\begin{itemize}
	\item A vertex is added above the root that is always infected.  As such the chain no longer has an absorbing state.
	\item Recoveries at the root only occur when none of its descendants are infected at the time of recovery. All the other infections and recoveries are the same as the original process.
\end{itemize}
In the modified process, when the root is infected, the processes inside each subtree from a child of the root behave independently.  By relating the stationary probability of the root being uninfected to the extinction time we develop a recursive relationship over the tree height. As a result, we show that the expected survival time of the contact process with small enough $\lambda$ is bounded by a constant, for any finite-depth Galton-Watson trees.

\vspace{2mm}
\noindent $\blacktriangleright$ \textit{Exponential decay of infection depth: the delayed process}. ~To relate the finite Galton-Watson trees to the infinite tree, we prove that the probability that the infection goes deeper than depth $h$ decays exponentially in $h$. To this end, we introduce the \textit{delayed process}, which spends exponentially longer time at states containing deeper infections. Based on a similar argument introduced above, we show that the expected survival time of the delayed process on the Galton-Watson tree is bounded by a constant if $\lambda$ is small enough. This will imply that in the original process, the infection can go deeper than $h$ at most with an exponentially small probability in $h$.
Thus, the contact process on the infinite Galton-Watson tree can be regarded as that on a large-depth finite tree, and hence we establish Theorem \ref{thm:gw main}.

\vspace{2mm}
\noindent $\blacktriangleright$ \textit{Coupling the local neighborhoods of $\mathcal{G}(n,\mu)$.} ~To study the contact process on $G\sim \mathcal{G}(n,\mu)$ exhibiting short survival, we attempt to dominate the local neighborhoods of the graph by Galton-Watson trees, in terms of isomorphic inclusions of graphs. However,  some of the local neighborhoods $N(v,r)$ will contain a cycle, and hence we introduce modified Galton-Watson type processes that contain a cycle and behave similarly as the Galton-Watson trees. After dominating the local neighborhoods of $G$ by the new branching processes, we study the contact process on the latter graphs and bound its survival time based on the aforementioned ideas.

\vspace{2mm}
On the other hand, when studying the long survival for $G\sim \mathcal{G}(n,\mu)$, we rely on the existence of what we call \textit{embedded expanders} inside the graph. Roughly speaking, we call a subset $W$ of vertices in $G$ an \textit{embedded expander}, if (i) all vertices in $W$ have high degree, say, at least $M$,  and (ii) distance $R$-neighborhood of every subset $W' \subset W$ of at most a certain size intersects with $W$ at more than $2|W'|$ vertices (See a precise definition in Lemma \ref{lm:structure}).  As noted above, we expect an infection at a degree $M$ vertex to last for at least time exponential in $M$.

Intuitively, if a subset $W'$ of an embedded expander $W$ is infected,  then the infections inside $W'$ would happen repeatedly for a reasonably long time due to its large degrees, and hence it will not die out \textsf{whp} before infecting its neighbors within distance $R$. Thus, if $W$ is an embedded expander, then the infection is likely to spread over $2|W'|$ vertices after some time. In Sections~ \ref{sec:proof:thm:ex} and \ref{sec:proof:thm:subex} we make this intuition rigorous and prove the existence of an embedded expander inside $G$. For the latter argument, we partially rely on the \textit{Cut-off line algorithm} (Definition \ref{def:cut-off line}) which was introduced in~\cite{k06}  to find the cores of random graphs.

For Theorem~\ref{thm: subexp tail main}, we show that if $\mu$ is subexponential, then we can find an embedded expander in $G$ such that $R$ is arbitrarily smaller than $M$. Therefore, even if $\lambda$ is very small, it will be possible for infections in the embedded expander to travel the distance of $R$ before dying out.

\subsection{Organization}

The rest of the paper is organized as follows. After we set up notations and review some preliminary facts in \S \ref{sec:prelim}, we prove Theorem \ref{thm:gw main},  Theorem \ref{thm: exp tail main}-(1), Theorem \ref{thm: exp tail main}-(2) in \S \ref{sec:sub gw}, \S \ref{sec:sub rg}, and \S \ref{sec:proof:thm:ex} respectively.  In \S \ref{sec:proof:thm:subex}, we prove Theorems \ref{thm: subexp tail main}. In \S \ref{sec:structure}, we prove a structural lemma on the embedded expanders mentioned above which plays a crucial role in establishing Theorems \ref{thm: exp tail main}-(2) and \ref{thm: subexp tail main}.  

\subsection{Notations} For two positive sequences $(a_n)$ and $(b_n)$, we say that $a_n = O(b_n)$ or $b_n =\Omega(a_n)$ if there exists a constant $C$ independent of $n$ such that $a_n\le C b_n$ for all $n$. If $a_n = O(b_n)$ and $b_n = O(a_n)$, we write $a_n = \Theta(b_n)$.

\section{Preliminaries}\label{sec:prelim}

In this section, we set up notation and briefly describe some basic properties of the contact process and random graphs which will be used throughout the paper.

For a graph $G=(V,E)$ (finite or infinite),  the contact process on $G$ with infection rate $\lambda$ is the continuous-time Markov chain on the state space $\{0,1\}^V$, where $0$ (resp. $1$) corresponds to the healthy (resp. infected) state. If the initial state is $\one_A$, i.e., the vertices in $A\subset V$ are infected, we denote the process by $$(X_t) \sim \cp^\lambda(G;\one_A).$$
We will frequently use the notation $\zero$ for the all-healthy state $\zero = \one_{\emptyset}$, and write $\one_v = \one_{\{v\}}$ if the state has a single infected vertex $v$. The transition rule of the process can be described as follows:

\begin{itemize}
	\item $X_t$ becomes $X_t - \one_v$ with rate $1$ for each $v$ such that $X_t(v)=1$.
	
	\item $X_t$ becomes $X_t + \one_u$ with rate $\lambda N_t(u)$ for each $u$ with $X_t(u)=0$, where $N_t(u)$ is the number of neighbors $v$ of $u$ with $X_t(v)=1$.
\end{itemize}
We sometimes write $\cp^\lambda (G)$ when the initial condition is unnecessary. For convenience, we usually denote the state space by $\{0,1\}^G$.

\subsection{Graphical representation of contact processes}

We briefly discuss a coupling method of the contact processes using a  \textit{graphical representation} based on Chapter 3, section 6 of \cite{liggett:ips}. The idea is to record the infections and recoveries in
$\cp^\lambda(G;\one_A)$ on the space-time domain $G \times \mathbb{R}_+$. Define i.i.d. Poisson processes $\{N_v(t)\}_{v\in V}$ with rate $1$ and i.i.d. Poisson processes $\{N_{\vec{uv}}(t) \}_{\vec{uv}\in \overrightarrow{E}}$ with rate $\lambda$, where  $\overrightarrow{E}=\{\vec{uv},\;\vec{vu}:(uv)\in E \}$ is the set of directed edges. Further, we let $\{N_v(t)\}_{v\in V}$ and $\{N_{\vec{uv}}(t) \}_{\vec{uv}\in \overrightarrow{E}}$ to be mutually independent. Then the graphical representation is defined as follows:
\begin{enumerate}
	\item Initially, we have the \textit{empty} domain $V \times \mathbb{R}_+$.
	
	\item For each $v\in V$, mark  $\times$ at the point $(v,t)$, at each event time $t$ of $N_v(\cdot)$.
	
	\item For each $\vec{uv}\in \overrightarrow{E}$, add an arrow from $(u,t)$ to $(v,t)$, at each event time $t$ of $N_{\vec{uv}}(\cdot)$.
\end{enumerate}

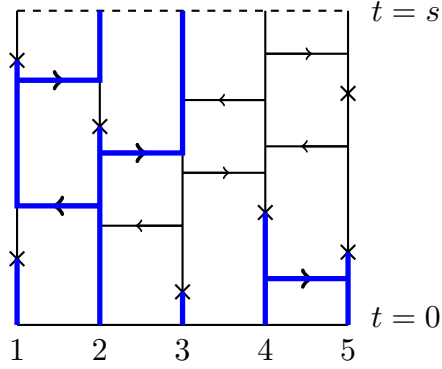
\begin{figure}
	\centering
	\begin{tikzpicture}[thick,scale=1.1, every node/.style={transform shape}]
	\foreach \x in {0,...,4}{
		\draw[black] (\x,0) -- (\x,3.8);
	}
	\draw[black] (0,0)--(4,0);
	\draw[dashed] (0,3.8)--(4,3.8);
	\draw (0,.8) node[cross=2.2pt,black]{};
	\draw (0,3.2) node[cross=2.2pt,black]{};
	\draw (1,2.4) node[cross=2.2pt,black]{};
	\draw (2,.4) node[cross=2.2pt,black]{};
	\draw (3,1.36) node[cross=2.2pt,black]{};
	\draw (4,.88) node[cross=2.2pt,black]{};
	\draw (4,2.8) node[cross=2.2pt,black]{};
	
	\draw[->,line width=.6mm,black] (0,2.96)--(.57,2.96);
	\draw[black] (0,2.96)--(1,2.96);
	\draw[->,line width=.6mm,black] (1,1.44)--(.43,1.44);
	\draw[black] (1,1.44)--(0,1.44);
	\draw[->,line width=.6mm,black] (1,2.08)--(1.57,2.08);
	\draw[black] (1,2.08)--(2,2.08);
	\draw[->,black] (2,1.2)--(1.43,1.2);
	\draw[black] (1,1.2)--(2,1.2);		
	\draw[->,black] (2,1.84)--(2.57,1.84);
	\draw[black] (2,1.84)--(3,1.84);
	\draw[->,black] (3,2.72)--(2.43,2.72);
	\draw[black] (2,2.72)--(3,2.72);		
	\draw[->,line width=.6mm,black] (3,.56)--(3.57,.56);
	\draw[black] (3,.56)--(4,.56);
	\draw[->,black] (3,3.28)--(3.57,3.28);
	\draw[black] (3,3.28)--(4,3.28);
	\draw[->,black] (4,2.16)--(3.43,2.16);
	\draw[black] (3,2.16)--(4,2.16);			
	\node at (4.7,0.1) {$t=0$};
	\node at (4.7,3.8) {$t=s$};
	\node at (0,-.3) {$1$};
	\node at (1,-.3) {$2$};
	\node at (2,-.3) {$3$};
	\node at (3,-.3) {$4$};
	\node at (4,-.3) {$5$};
	\draw [line width=.07cm, blue] (0,0) -- (0,.8);
	\draw [line width=.07cm, blue] (1,0) -- (1,2.4);	
	\draw [line width=.07cm, blue] (1,1.44)--(0,1.44) -- (0,3.2);
	\draw [line width=.07cm, blue] (0,2.96) -- (1,2.96)--(1,3.8);
	\draw [line width=.07cm, blue] (1,2.08) -- (2,2.08)--(2,3.8);	
	\draw [line width=.07cm, blue] (2,0) -- (2,.4);
	\draw [line width=.07cm, blue] (3,1.36) -- (3,0);
	\draw [line width=.07cm, blue] (4,0) -- (4,.88);
	\draw [line width=.07cm, blue] (3,.56) -- (4,.56);			
	
	\end{tikzpicture}
	\caption{A realization of the contact process on the interval $V=\{1,\ldots,5\}$, with initial condition $X_0=\one_V$. The blue lines describe the spread of infection. We see that $X_s = \one_{\{2,3\}}$} \label{fig1}
\end{figure}

\noindent This gives a geometric picture of $\cp^\lambda(G;\one_A)$, and further provides a coupling of the processes over all possible initial states. Figure \ref{fig1} tells us how to interpret the infections at time $t$ based on this graphical representation.  We point out two lemmas which are easy consequences of the above construction. For proofs, see, e.g., \cite{liggett:ips}.

\begin{lemma}\label{lem:graph rep1}
	Suppose that we have the aforementioned coupling among the contact processes on a graph $G$. Let $T_v$ and $T_G$ be the first time when $\cp^\lambda(G;\one_v)$ and $\cp^\lambda(G;\one_G)$ reach the all-healthy state $\zero$, respectively. Then we have $T_G = \max \{T_v:v\in G \}$.
\end{lemma}

\begin{lemma}\label{lem:graph rep2}
	For a given graph $G=(V,E)$ and any $A\subset V$, let $(X_t)\sim \cp^\lambda (G;\one_A)$. Consider any (random) subset $\mathcal{I}$ of $ \mathbb{R}_+$, and  define $(X_t')$ to be the coupled process of $(X_t)$ that has the same initial state, infections and recoveries, except that the recoveries at a fixed vertex $v$ are ignored at times $t\in \mathcal{I}$. Then for any $t\geq 0$, we have $X_t \leq X_t'$, i.e., $X_t(v) \leq  X_t'(v)$ for all $v$.
\end{lemma}

\subsection{Random graphs}\label{subsec: rg basic}

Let $\mu$ be a probability distribution on $\mathbb{N}$, and $n$ be any integer. The random graph $\mathcal{G}(n,\mu)$ with degree distribution $\mu$  is defined by the following procedure:
\begin{itemize}
	\item Let $d_1,\ldots , d_n$ be $n$ i.i.d. samples from $\mu$ conditioned on $\{\sum_{i=1}^n d_i \textnormal{ is even} \}$.
	
	\item Sample $G$ by taking a simple graph on $n$ vertices with degrees $\{ d_i\}_{i=1}^n$, uniformly at random among all possible choices.
\end{itemize}

\noindent Further, we consider a variant of $\mathcal{G}(n,\mu)$ which is  constructed as follows:
\begin{itemize}
	\item Sample $d_1,\ldots , d_n$ as above. Here, $d_i$ denotes the number of \textit{half-edges} attached to vertex $i$.
	
	\item Pair all the half-edges uniformly at random.
\end{itemize}
The resulting graph $G$ is called the configuration model, which is denoted by $\mathcal{G}_{\textsf{cf}}(n,\mu)$. The difference here is that $G$ is not necessarily a simple graph. However, if the second moment of $\mu$ is finite, we have the following contiguity between the two models. For details, see, e.g., \cite{vanderhofstad17}, Chapter 7.

\begin{lemma}[\cite{j09, vanderhofstad17}]
	Suppose that $\E_{D\sim \mu} D^2<\infty$. Then, uniformly in $n$, we have $$\P_{G\sim \mathcal{G}_{\textsf{cf}}(n,\mu)} (G \textnormal{ is simple}) \in (0,1).$$
	In particular, for any subset $A_n$ of graphs with $n$ vertices,
	$$\P_{G\sim \mathcal{G}_{\textsf{cf}}(n,\mu)} (G\in A_n) \rightarrow 0 \quad\textnormal{implies}\quad \P_{G\sim \mathcal{G}(n,\mu)} (G\in A_n) \rightarrow 0.$$
\end{lemma}

\noindent Throughout the paper, we study the configuration model $\mathcal{G}_{\textsf{cf}}(n,\mu)$ instead of $\mathcal{G}(n,\mu)$, under the assumption $\E_{D\sim \mu} D^2<\infty$. 
Further, it is well known that $G\sim \mathcal{G}_{\textsf{cf}}(n,\mu)$ (and hence, $G\sim \mathcal{G}(n,\mu)$) \textsf{whp} contains the unique connected component of size linear in $n$, if and only if $\E_{D\sim \mu} D(D-2) > 0$ (for details, see \cite{mr95}). Hence, we always assume $\E_{D\sim \mu} D(D-2)> 0,$ which is the most interesting case for us. Otherwise, the graph decomposes into many small components and the contact process would not exhibit $e^{\Theta(n)}$-survival time for any $\lambda>0$.

\vspace{3mm}
\subsection{Local weak convergence} Given a sequence of random graphs $G_n$, let $N(v, r)$ be an induced subgraph of $G_n$ consisting of vertices of distance at most $r$ from $v$. Let $\P_n$ be the distribution of the neighborhood $N(v, r)$ where $v$ is a uniformly chosen vertex of $G_n$. We say that a random rooted tree $\mathcal T$ is the local weak limit of $G_n$ if for any finite $r$ and any rooted tree $T$ of depth at most $r$,
$$\lim_{n\to\infty} \P_n(N(v, r)=T) = \P(\mathcal T_r = T)$$
where $\mathcal T_r$ is the subtree of the first $r$ generations of $\mathcal T$.

We shall use the following known convergence of $G(n, \mu)$ to its corresponding Galton-Watson tree. Define the size-biased distribution
$\mu'$ to be
$$\mu'(k-1) = \frac{k\mu(k)}{\sum_{i=1}^{\infty} i\mu(i)}, \quad k=1, 2, \dots$$
Note that if $\mu = \text{Pois}(d)$ then $\mu' = \mu$.
Let $\mathcal{T}(\mu) \sim \gw(\mu, \mu')$ be the \textit{size-biased} Galton-Watson tree in which the number of children of the root has distribution $\mu$ and the number of children of an $i$-th generation vertex ($i\ge 1$) has distribution $\mu'$. We also stress that the Galton-Watson tree is supercritical if and only if $\E_{D'\sim{\mu'}}{D'}>1$,  equivalent to $\E_{D\sim \mu} D(D-2)> 0$ which we saw above.
\begin{lemma} \cite[Section 2.1]{DemboMontanari2010}
	Assume that $\mu$ has finite mean. Then the size-biased Galton-Watson tree $\mathcal T(\mu)$ is the local weak limit of $\mathcal G(n, \mu)$. The Galton-Watson tree with degree distribution $\text{Pois}(d)$ is the local weak limit of the Erd\H{o}s-R\'enyi random graph $\mathcal G_{n, d/n}$.
\end{lemma}

%\vspace{1mm}
%\textit{Galton-Watson trees.} Consider the Galton-Watson tree $\mathcal{T}\sim \gw(\xi)$ with offspring distribution $\xi$ satisfying $b:=\E\xi >1$. Let $Y_l$ be the number of descendants at depth $l$ of $\mathcal{T}$. Then it is well-known that $\{Y_l/b^l\}_l$ is a martingale, whose convergence can be described by the following lemma.

%Lemma...

\section{Extinction in Galton-Watson trees}\label{sec:sub gw}
Let $\xi$ be a random variable on $\mathbb{N}$ having an exponential tail, namely,  $\E \exp(c \xi)= M <\infty$ for some constants $c, M >0$. Throughout this section we assume $\E \xi >1$, which makes $\mathcal{T}\sim \gw(\xi)$, the Galton-Watson tree with offspring distribution $\xi$, survive forever with positive probability.  We also denote the depth-$L$ Galton-Watson tree by $\mathcal{T}_L \sim \gw(\xi)_L$, and   its root denoted by $\rho$.

The goal of this section is to establish Theorem \ref{thm:gw main}. To this end, we prove the following in the next two sections:
\begin{itemize}
	\item We first show that for small enough $\lambda$, the expected survival time of $\cp^\lambda (\mathcal{T}_L; \one_\rho)$ is bounded by a constant uniform in $L$.
	
	\item Then, we prove that for small enough $\lambda$, the probability that the infection in $\cp^\lambda (\mathcal{T}_L;\one_\rho)$ goes deeper than $h$ decays exponentially in $h$.
\end{itemize}
At the end, we will combine the two to see that the death-survival threshold $\lambda_1$ of the infinite Galton-Watson tree is strictly positive. Moreover, both properties will be essential in Section \ref{sec:sub rg}.

\subsection{Expected survival time in finite trees}\label{sec:cp on gw}

In this section, we prove the following theorem:

\begin{theorem}\label{thm:cp on gw}
	Let $L$ be an arbitrary integer and $\xi, \mathcal{T}_L$ be defined as above. Let $R_L$ be the first time when $\cp^\lambda(\mathcal{T}_L;\one_\rho)$ reaches state $\zero$. Then there exist constants $C, \lambda_0>0$ depending only on $\xi$ such that for any $\lambda \leq \lambda_0$ and $L$, we have $\E R_L \leq C$.
\end{theorem}

Let $D\sim\xi$ denote the degree of the root $\rho$, and $v_1,\dots,v_D$ be the children of $\rho$. Further, let $\mathcal{T}_{v_i}$ denote the subtree of $\mathcal{T}$ rooted at $v_i$. To establish Theorem \ref{thm:cp on gw}, our attempt is to study the effect of joining the subtrees $\mathcal{T}_{v_i}$ together at $\rho$, and hence expressing $R_L$ in terms of $R_{L-1}$. The main difficulty of this approach comes from the fact that the contact process on $\mathcal{T}_L$ does not behave independently on each subtree $\mathcal{T}_{v_i}$. To overcome this obstacle, we study the contact process in a slightly different setting, by adding a parent $\rho^+$ above the root $\rho$ which is  infected
permanently.

\begin{definition}[Root-added contact process]\label{def:modified cp}
	Let $T$ be a finite tree rooted at $\rho$. Let $T^+$ be the tree that has a parent vertex $\rho^+$ of $\rho$ which is connected only with $\rho$.  The \textit{root-added contact process} on $T$ is the continuous-time Markov chain on the state space $\{0,1\}^T$, defined as the contact process on $T^+$ with  $\rho^+$ set to be infected permanently (hence we exclude $\rho^+$ from the state space). That is, $\rho^+$ is infected initially, and it does not have a recovery clock attached to itself. Let $\cp^\lambda_{\rho^+}(T^+;x_0)$ denote the root-added contact process on $T$ with initial condition $x_0 \in \{0,1 \}^T $.
\end{definition}

By adding a permanently infected parent, we can take advantage of independence between different subtrees as well as the stationary distribution of the process, as briefly discussed in Section  \ref{subsec:maintec}. In the following lemma, we formally introduce the ``modified process" explained in Section  \ref{subsec:maintec} and construct a quantitative recursion argument in terms of the tree depth.

% in the following ways. Let $(X_t) \sim \cp^\lambda(\mathcal{T}_L;\one_\rho)$ and consider another contact process $(X_t')$ coupled with $(X_t)$ such that

%\begin{itemize}
%	\item $(X_t')$ starts at $\one_\rho$ and shares the same locations of infections and recoveries as $(X_t)$, except for the recoveries at $\rho$.

%	\item Recoveries at $\rho$ are neglected if and only if $X'_t \neq\one_\rho$ at the time it happens.
%\end{itemize}

%During the time when $X_t'\neq \one_{\rho}$, $X_t'$ can be viewed as $D$ independent copies of $\cp^\lambda_{v_i^+}(\mathcal{T}_{v_i}^+; \zero)$, since the infection at $\rho$ is never healed until the process returns to $\one_\rho$. Thus, $\rho$ plays the role of permanently infected parent of $v_i$'s. This feature of independence plays a crucial role in establishing the following lemma.

%Moreover, the root-added contact process $\cp^\lambda_{\rho^+}(T^+;\one_{\rho})$ no longer has an absorbing state, and hence has a nontrivial stationary distribution. We will later relate the running time $S_L$ (defined in Lemma \ref{lem:cp on gw}) with the stationary distributions of root-added contact processes, which will be important in formulating the inductive tree recursion argument.

\begin{lemma}\label{lem:cp on gw}
	Let $L$ be an arbitrary integer and $\xi, \mathcal{T}_L$ be defined as above. Define $S_L$ to be the first time when $\cp^\lambda_{\rho^+}(\mathcal{T}_L^+;\one_\rho)$ reaches state $\zero$. Then there exists a constant $\lambda_0 >0$ depending only on $\xi$ such that for any $\lambda \leq \lambda_0$ and $L$, $\E S_L \leq e$.
\end{lemma}

\begin{proof}
	We build an inductive argument in terms of $L$, by considering the  modified contact process $(\widetilde{X}_t) \sim\tcp^\lambda_{\rho^+;\rho} (\mathcal{T}_L^+ ; \one_\rho)$ defined as follows:
	\begin{itemize}
		\item $(\widetilde{X}_t)$ is coupled with $(X_t) \sim \cp^\lambda_{\rho^+} (\mathcal{T}_L^+; \one_\rho)$ in the sense that they share the same locations of recovery and infection clocks. In particular, $\rho^+$ is permanently infected in $(\widetilde{X}_t)$.
		
		\item In $(\widetilde{X}_t)$, the recovery at $\rho$ at time $s$ is valid if and only if $\widetilde{X}_s =  \one_\rho$. Otherwise, we ignore the recovery at $\rho$. In other words, when there exists an infected vertex other than $\rho$ and $\rho^+$, the recovery at $\rho$ is neglected.
	\end{itemize}
	Let $\widetilde{S}_L$ be the first time when $\widetilde{X}_t$ reaches the all-healthy state $\zero$. Then Lemma \ref{lem:graph rep2} tells us that $S_L \leq \widetilde{S}_L$.
	Assume that we started running $(\widetilde{X}_t)$ from $t=0$. Then there are two possibilities for the transition to the second state from the initial state $\one_\rho$:
	
	\begin{enumerate}
		\item [\textbf{A}.]  $\rho$ is healed;
		
		\item [\textbf{B}.] $\rho$ infects one of its children, say, $v_i$.
	\end{enumerate}
	When \textbf{A} happens, then $\widetilde{S}_L$ is just the time elapsed until encountering \textbf{A}. If we let $D$ denote the number of children of $\rho$, then probability of the event $\textbf{A}$ is $\frac{1}{1+\lambda D}$ and also
	\begin{equation*}
	\E [\widetilde{S}_L | {\textnormal{\textbf{A}}} ] = \frac{1}{1+\lambda D}.
	\end{equation*}
	On the other hand, when \textbf{B} happens, then the recoveries at $\rho$ are neglected until all of its descendants are healthy. Therefore, the infection and recovery occurring inside the subtrees $\mathcal{T}_{v_i} \cup \{\rho\}$ become independent of each other until all of them become completely healthy at the same time. Hence, after the occurrence of \textbf{B}, where we have $\one_{\{\rho,v_i \}}$ as its new initial state, $(\widetilde{X}_t)$ can be viewed as the product chain $\left(X^\otimes_t\right)$ of root-added contact processes defined as follows.
	\begin{equation*}
	\left(X^\otimes_t \right) ~\sim~ \cp^\otimes_\rho (\mathcal{T}_L ; \one_{v_i} )
	:= \left(\otimes_{\substack{ j=1\\j\neq i}}^D \cp^\lambda_{\rho} (\mathcal{T}^+_{v_j}; \zero) \right) \otimes \cp^\lambda_{\rho}(\mathcal{T}^+_{v_i}; \one_{v_i}).
	\end{equation*}
	(Here for each $\mathcal{T}_{v_j}$, we view $\rho$ as its permanently infected parent of the root $v_i$.) Note that this perspective is valid until $\widetilde{X}_t$ returns back to $\one_\rho$.
	
	Let $\widetilde{S}^\otimes_i$ denote the time that the above product chain on $\cup_{j=1}^D \mathcal{T}_{v_j} $ started from the state $\one_{v_i}$ reaches the all-healthy state $\zero$.
	At time $s=\widetilde{S}^\otimes_i$, $\widetilde{X}_s$ is again in the state $\one_\rho$, hence it again meets with either \textbf{A} or \textbf{B} in the next step. Note that in this situation the expected waiting time to encounter either event is $\frac{1}{1+\lambda D}$. Also, define $\widetilde{S}^\otimes$ to be the average of $\widetilde{S}^\otimes_i$ over all $i$, recalling that when event \textbf{B} occurs, each child $v_i$ is infected with equal probability.  Then, if we continue this procedure until $\widetilde{X}_t$ reaches $\zero$, we get
	\begin{equation}\label{eq:geometric trial cp}
	\begin{split}
	\E [\widetilde{S}_L \;  |\; \mathcal{T}_L]
	=
	\sum_{k=0}^\infty \left(\frac{\lambda D}{1+\lambda D} \right)^k \frac{1}{1+\lambda D} \times\left[(k+1)\frac{1}{1+\lambda D}  + k\E \left[\left.\widetilde{S}^\otimes\; \right|\; \{ \mathcal{T}_{v_i} : i\in [D] \}\right] \right].
	\end{split}
	\end{equation}
	Simplifying the sum then gives 
	$$\E [\widetilde{S}_L \;  |\; \mathcal{T}_L]= 1+\lambda D \E \left[\left.\widetilde{S}^\otimes\; \right|\; \{ \mathcal{T}_{v_i} : i\in [D] \}\right], $$
	which implies
	\begin{equation}\label{eq:cp tree}
	\E [\widetilde{S}_L \,  |\, D] = 1+\lambda D\, \E[\widetilde{S}^\otimes|D] .
	\end{equation}
	The next step of the proof is to estimate $\E[\widetilde{S}^\otimes|D]$ by relating it to the stationary distributions of the root-added contact processes. Let $\pi^{(D)}$ be the stationary distribution of the product chain $\cp^\otimes_\rho(\mathcal{T}_L)$ (when defining $\pi^{(D)}$, note that the initial state of the process is irrelevant). We also let $\pi_i$ be the stationary distribution of $\cp^\lambda_{\rho} (\mathcal{T}_i^+)$. Then we have $$\pi^{(D)} = \otimes_{i=1}^D \pi_i.$$
	For any state $x$ on $\mathcal{T}_L \setminus\{\rho\}$, $\pi^{(D)}(x)$ is proportional to the expected time that the chain $(X^\otimes_t) \sim \cp^\otimes_\rho(\mathcal{T}_L)$ stays at state $x$. Moreover, the expected time for the chain to stay at $\zero$ is $(\lambda D)^{-1}$, and after escaping from $\zero$, it spends time $\E[\widetilde{S}^\otimes| \mathcal{T}_L]$ in expectation before returning back to $\zero$. Therefore,
	\begin{equation}\label{eq:stationary and running time 1}
	\pi^{(D)} (\zero) = \frac{(\lambda D)^{-1}}{(\lambda D)^{-1} + \E [\widetilde{S}^\otimes \,| \, \mathcal{T}_L] }
	= \frac{1}{1+ \lambda D \, \E [\widetilde{S}^\otimes \,| \, \mathcal{T}_L] }.
	\end{equation}
	Similarly, we have
	\begin{equation}\label{eq:stationary and running time 2}
	\pi_i (\zero) = \frac{1}{1+ \lambda \E [{S}_{L-1} \, | \, \mathcal{T}_{v_i}]},
	\end{equation}
	where ${S}_{L-1}$ is the first time when $(X_t^i) \sim \cp^\lambda_{\rho} (\mathcal{T}_{v_i}^+; \one_{v_i})$ reaches state $\zero$. Here, note that $S_{L-1}$ matches with the definition from the statement of this lemma since $\mathcal{T}_{v_i} \sim \gw(\xi)_{L-1}$. Therefore, we obtain that
	\begin{equation}\label{eq:cp tree rec}
	1+\lambda D \,  \E [\widetilde{S}^\otimes \,| \, \mathcal{T}_L] =
	\prod_{i=1}^D (1+\lambda \E[{S}_{L-1} | \mathcal{T}_{v_i}]).
	\end{equation}
	Since  $\{\mathcal{T}_{v_i}\}_{i\geq1}$ are i.i.d. $ \gw(\xi)_{L-1}$ for all $i$, integrating (\ref{eq:cp tree rec}) over the randomness of $\{\mathcal{T}_{v_i}: i\in[D] \}$ tells us that
	\begin{equation*}
	1+\lambda D \E[ \widetilde{S}^\otimes |D] = (1+\lambda \E [{S}_{L-1}])^D \leq \exp\{(\lambda \E [{S}_{L-1}]) D\}.
	\end{equation*}
	Combining this with (\ref{eq:cp tree}), we get
	\begin{equation}\label{eq:cp tree fin}
	\E [ \widetilde {S}_L|D] \leq \exp\{(\lambda \E [{S}_{L-1}]) D\}.
	\end{equation}
	
	In the last step of the proof, we complete the inductive argument using the fact that $D \sim \xi$ has an exponential tail. Let us set $c,M>0$ to be the constants satisfying $\E \exp({c D}) = M$. When $L=0$, we trivially have that $\E {S}_0= 1$. Define $K$ and $\lambda_0$ as
	\begin{equation*}
	K = e\cdot \max\{\log M,\, 1 \}, \quad \lambda_0 = \frac{c}{K}.
	\end{equation*}
	Suppose that $\E [S_{L-1}] \leq e$. Then for any $\lambda\leq\lambda_0$, we have
	\begin{equation*}
	\begin{split}
	\E S_L \leq \E \widetilde{S}_L \leq \E_{D\sim \xi} \left[\exp( \lambda \E[{S}_{L-1}] D) \right] &=
	\E_{D\sim \xi}\left[\exp\left(\frac{ \lambda \E[{S}_{L-1}]}{c}\cdot cD\right) \right]
	\\ &\leq \exp \left\{\log M \frac{\lambda \E[{S}_{L-1}]}{c} \right\} \leq e,
	\end{split}
	\end{equation*}
	where we used Jensen's inequality to deduce the first inequality in the second line. Finally, an elementary induction argument implies the desired result.
\end{proof}

\begin{proof}[Proof of Theorem \ref{thm:cp on gw}] For $R_L, S_L$ defined as in the statement of Theorem \ref{thm:cp on gw} and Lemma \ref{lem:cp on gw}, respectively, we have $\E R_L \leq \E S_L$ due to Lemma \ref{lem:graph rep2}. Therefore, setting $\lambda_0$ as in the proof of Lemma \ref{lem:cp on gw} and $C=e$, we obtain $\E R_L\leq C$ for all $\lambda\leq \lambda_0$ and $L$. 
\end{proof}

\subsection{Exponential decay of the infection depth}\label{sec:decay of depth}

In this section, we show that the maximal depth that the infection can reach before dying out decays exponentially.

For any integer $L$, let $\mathcal{T}_L \sim \gw(\xi)_L$ and $\mathcal{T}_L^+$ be the graph obtained by adding a new parent root $\rho^+ $ above $\rho$ in $\mathcal{T}_L$ as before. For each state $x \in \{0,1\}^{\mathcal{T}_L}$, define the depth of $x$ in $\mathcal{T}_L^+$ to be
\begin{equation*}
r(x) =r(x;\mathcal{T}_L^+) = \max \{d(\rho^+,v): x(v)=1  \}.
\end{equation*}
For $x=\zero$, we set $r(\zero)=0$. Consider the root-added process $(X_t) \sim \cp^\lambda_{\rho^+} (\mathcal{T}_L^+; \one_{\rho } )$ (Definition \ref{def:modified cp}), and let $S_L$ be the first time then the process reaches the state $\zero$. Let $H=\max \{r(X_t) : t\in [0, S_L] \}$ be the maximal depth that the process reaches during an excursion from $\zero$. Our goal in this section is to establish the following theorem and conclude the proof of  Theorem \ref{thm:gw main}.

\begin{theorem}\label{thm:decay of depth}
	Let $L>0$ be any integer and let $\mathcal{T}_L$, $S_L$ and $H$  be as above.  There exist constants $K, \lambda_0>0$ depending only on $\xi$ such that for all $\lambda \leq \lambda_0$, $h>0$ and $m>0$, we have
	\begin{equation*}
	\P (H>h \,|\, \mathcal{T}_L) \leq 2m (K\lambda)^h,
	\end{equation*}
	with probability at least $1-m^{-1}$ over the choice of $\mathcal{T}_L$.
\end{theorem}

In order to control the deepest depth of infection, we introduce the \textit{delayed} contact process.

\begin{definition}[Delayed contact process]\label{def:dcp}
	Let $\mathcal{S}^+$ be a graph rooted at $\rho^+$ and $\mathcal{S}=\mathcal{S}^+\setminus \{\rho^+\}$. For any two states $x,y \in \{0,1\}^{\mathcal{S}}$, let $Q_{xy}$ be the rate of transition from $x$ to $y$ in the contact process $\cp^\lambda_{\rho^+} (\mathcal{S}^+)$. For a fixed constant $\theta \in (0,1)$, the \textit{delayed contact process}, denoted by $\ddp^{\lambda, \theta}_{\rho^+} (\mathcal{S}^+; x_0 )$, is the continuous-time Markov chain on $\{0,1 \}^{\mathcal{S}}$ with initial state $x_0$ and transition rate
	\begin{equation*}
	Q^{(\theta)}_{xy} = \theta^{r(x;\mathcal{S}^+)}Q_{xy}=\theta^{r(x)}Q_{xy}.
	\end{equation*}
\end{definition}

\noindent According to the definition, in the delayed contact process, we spend exponentially longer time in the states with deeper depths. Let $\pi_{\mathcal{S}}, \nu^\theta_{\mathcal{S}}$ denote the stationary distributions of $\cp^\lambda_{\rho^+} (\mathcal{S}^+)$ and $\ddp^{\lambda,\theta}_{\rho^+} (\mathcal{S}^+)$, respectively. Then,
\begin{equation}\label{eq:stationary distr basic}
\nu^\theta_{\mathcal{S}} (x) = \frac{\theta^{-r(x)}\pi_{\mathcal{S}}(x) }{\sum_y \theta^{-r(y)}\pi_{\mathcal{S}}(y)  },
\end{equation}
where the summation is over all possible states $y\in \{0,1 \}^{\mathcal{S}}$.

Suppose we have a lower bound on $\nu^\theta_{\mathcal{S}}(\zero) $. Then this implies an upper bound on $\pi_{\mathcal{S}}(x)$, by
\begin{equation*}
\pi_{\mathcal{S}}(x) =\frac{\theta^{r(x)}\nu^\theta_{\mathcal{S}}(x)}{\sum_y \theta^{r(y)} \nu^\theta_{\mathcal{S}}(y)} \leq
\frac{\theta^{r(x)} \nu^\theta_{\mathcal{S}}(x)}{\nu^\theta_{\mathcal{S}}(\zero)},
\end{equation*}
which intuitively infers that it is (exponentially) unlikely to see states of having very deep infections until the process comes back to $\zero$. Based on this intuition, we establish the following proposition.

\begin{proposition}\label{prop:dcp on gw}
	Let $L>0$ be any integer and $\mathcal{T}_L \sim \gw(\xi)_L$.	Set $\nu^\theta_{\mathcal{T}_L}$ to denote the stationary distribution of $\ddp^{\lambda,\theta}_{\rho^+}(\mathcal{T}_L^+)$ on the space $\{0,1\}^{\mathcal{T}_L}$. Then there exist constants $K, \lambda_0>0$ depending only on $\xi$ such that for all $\lambda\leq \lambda_0$ and $L$,
	\begin{equation*}
	\E\left[\nu^\theta_{\mathcal{T}_L}(\zero)^{-1} \right] \leq 2,
	\end{equation*}
	where $\theta$ is given by $\theta=K\lambda$.
\end{proposition}

\begin{proof}
	Let $S_L^\theta$ be the first time when $(X_t)\sim \ddp^{\lambda,\theta}_{\rho^+}(\mathcal{T}_L^+;\one_\rho) $ reaches state $\zero$. We first derive an analog of (\ref{eq:cp tree}) based on the methods from Lemma \ref{lem:cp on gw}. To this end, define $(\widetilde{X}_t)\sim \widetilde{\ddp}^{\lambda,\theta}_{\rho^+;\rho}(\mathcal{T}_L^+; \one_{\rho})$ to be the modification of $(X_t)$ in such a way that
	\begin{enumerate}
		\item [1.] $(\widetilde{X}_t)$ shares the same infection and recovery clocks as $(X_t)$.
		
		\item [2.] In $(\widetilde{X}_t)$, healing attempt at $\rho$ is ignored if there exists an infected vertex other than $\rho^+$ and $\rho$ at that moment.
	\end{enumerate}
	
	Let $D\sim \xi$ denote the number of children of $\rho$, and let $\{\mathcal{T}_{u_i} :i=1,\ldots,D \}$ be the subtrees from the children $u_1,\ldots,u_D$ of $\rho$. Then, when there is an infected vertex other than $\rho^+$ and $\rho$ in $\widetilde{X}_t$, it can be regarded as the slowed-down version of process $\ddp^{\lambda,\theta}_\rho (\mathcal{T}_L)$, where it spends longer time  by the factor of $\theta^{-1}$ at each state, since the tree $\mathcal{T}_L$ is one depth lower than $\mathcal{T}_L^+$. Note that in $\ddp^{\lambda,\theta}_\rho(\mathcal{T}_L)$, $\rho  $ is the permanently infected parent that has $D$ children.
	
	Let $\widetilde{S}_L^\theta$ be the first time when $\widetilde{X}_t$ becomes $\zero$. Also, let $\widetilde{S}_i^\theta$ be the first time when $\ddp^{\lambda,\theta}_\rho (\mathcal{T}_L; \one_{u_i})$ is $\zero$, and set $\widetilde{S}^\theta$ to be the average of $\widetilde{S}_i^\theta$ over $i=1,\ldots, D$.
	Then, we can apply the same argument as Lemma \ref{lem:cp on gw} to this setting and deduce that
	\begin{equation}\label{eq:gw delayed recursion}
	\begin{split}
	\E\left[\left.\widetilde{S}_L^\theta\,\right|\,\mathcal{T}_L  \right]
	&=
	\sum_{k=0}^\infty \left(\frac{\lambda D}{1+\lambda D} \right)^k \frac{1}{1+\lambda D} \left[\frac{k+1}{\theta(1+\lambda D)} + \frac{k}{\theta} \E\left[\widetilde{S}^\theta\,|\,\mathcal{T}_L\right] \right]\\
	&=
	\frac{1}{\theta} \left(1+\lambda D \E\left[\widetilde{S}^\theta\,|\,\mathcal{T}_L\right]\right).
	\end{split}
	\end{equation}
	
	Now we relate these equations with the stationary distributions. Let $\widetilde{\nu}_{\mathcal{T}_L}^\theta, \nu_{\mathcal{T}_{u_i}}^\theta$ be the stationary distributions of $\ddp^{\lambda,\theta}_\rho(\mathcal{T}_L)$, $\ddp^{\lambda,\theta}_\rho (\mathcal{T}_{u_i}^+)$, respectively. Further, define $\nu^\otimes_{\mathcal{T}_L} = \otimes_{i=1}^D \nu_{\mathcal{T}_{u_i}}^\theta$. In contrast to what we had in Lemma \ref{lem:cp on gw}, we do not necessarily have $\widetilde{\nu}_{\mathcal{T}_L}^\theta = \nu^\otimes_{\mathcal{T}_L}$.
	
	For each state $x\in \Omega_L := \{0,1 \}^{\cup_{i=1}^D \mathcal{T}_{u_i}}$ of $\ddp^{\lambda,\theta}_\rho(\mathcal{T}_L)$, we decompose it into $x=(x_i)_{i=1}^D$, where $x_i \in \Omega_i := \{0,1\}^{\mathcal{T}_{u_i}}$. Setting $\pi_{\mathcal{T}_{u_i}}$ to be the stationary distribution of $\cp^{\lambda}_\rho(\mathcal{T}_{u_i}^+)$ and $\pi^\otimes_{\mathcal{T}_L}:= \otimes_{i=1}^D \pi_{\mathcal{T}_{u_i}}$, the equation (\ref{eq:stationary distr basic}) implies that
	
	\begin{equation}\label{eq:stationary of delayed gw}
	\begin{split}
	\widetilde{\nu}_{\mathcal{T}_L}^\theta (x)
	&=
	\frac{\theta^{-r(x;\mathcal{T}_L)}\pi_{\mathcal{T}_L}^\otimes(x) }{\sum_{y\in \Omega_L} \theta^{-r(y;\mathcal{T}_L)}\pi_{\mathcal{T}_L}^\otimes(y)  }
	=
	\frac{\theta^{-r(x;\mathcal{T}_L)} \prod_{i=1}^D \pi_{\mathcal{T}_{u_i}}(x_i) }{\sum_{y\in \Omega_L} \theta^{-r(y;\mathcal{T}_L)}\prod_{i=1}^D \pi_{\mathcal{T}_{u_i}}(y_i)  };\\
	{\nu}_{\mathcal{T}_L}^\otimes (x)
	&=
	\prod_{i=1}^D \left[\frac{\theta^{-r(x_i;\mathcal{T}_{u_i})}\pi_{\mathcal{T}_{u_i}}(x_i) }{\sum_{y_i\in \Omega_i} \theta^{-r(y_i;\mathcal{T}_{u_i})}\pi_{\mathcal{T}_{u_i}(y_i)  } }\right] =
	\frac{\theta^{-\sum_{i=1}^D r(x_i;\mathcal{T}_{u_i})} \prod_{i=1}^D \pi_{\mathcal{T}_{u_i}}(x_i) }{\sum_{y\in \Omega_L} \theta^{-\sum_{i=1}^D r(y_i;\mathcal{T}_{u_i})}\prod_{i=1}^D \pi_{\mathcal{T}_{u_i}}(y_i)  }.
	\end{split}
	\end{equation}
	
	\noindent Notice that $$r(x;\mathcal{T}_L) =
	\max \{r(x_i;\mathcal{T}_{u_i}) : i=1,\ldots,D \}
	\leq \sum_{i=1}^D r(x_i;\mathcal{T}_{u_i}).$$
	Therefore, deeper states tend to have larger weight in $\nu_{\mathcal{T}_L}^\otimes$ than in $\widetilde{\nu}_{\mathcal{T}_L}^\theta$, which implies that $$\nu_{\mathcal{T}_L}^\otimes (\zero) \leq \widetilde{\nu}_{\mathcal{T}_L}^\theta(\zero).$$
	Moreover, we have the following equations as an analog of (\ref{eq:stationary and running time 1}), (\ref{eq:stationary and running time 2}):
	\begin{equation}\label{eq:stationary and delayed running time}
	\widetilde{\nu}_{\mathcal{T}_L}^\theta(\zero)
	=\frac{1}{1+\lambda D \E[\widetilde{S}^\theta | \mathcal{T}_L ]}, \quad 
	\nu_{\mathcal{T}_L}^\otimes (\zero)
	= \prod_{i=1}^D \left[\frac{1}{1+\lambda \E \left[S^\theta_{L-1}|\mathcal{T}_{u_i} \right]} \right].
	\end{equation}
	We combine our discussion with (\ref{eq:gw delayed recursion}) to deduce that
	\begin{equation*}
	\E\left[\left. S_L^\theta \right| \mathcal{T}_L \right]
	\leq
	\E \left[\left. \widetilde{S}_L^\theta \right| \mathcal{T}_L \right]
	\leq
	\frac{1}{\theta} \prod_{i=1}^D \left(1+ \lambda \E \left[\left.S_{L-1}^\theta \right| \mathcal{T}_{u_i} \right] \right),
	\end{equation*}
	and hence
	\begin{equation}\label{eq:gw delayed rec ineq}
	\E\left[S_L^\theta\right] \leq \frac{1}{\theta}
	\E_{D\sim \xi} \left[\exp \left(\lambda \E \left[S_{L-1}^\theta \right] D \right) \right].
	\end{equation}
	
	The final step is to adjust the constants and deduce the conclusion. Let $c, M>0$ be constants satisfying $\E_{D\sim \xi} \exp(c D) =M$. We set $K,\lambda_0>0$ and $\theta$ as
	\begin{equation*}
	K=\max \left\{ \frac{2\log M}{c \log 2}, 2 \right\}, \quad \lambda_0 = \frac{1}{2K}, \quad\theta = K\lambda,
	\end{equation*}
	where $\lambda\in (0,\lambda_0]$ is arbitrary. For $L=0$, we have $\E S_0^\theta = \theta^{-1}$. Suppose that $\E S_{L-1}^\theta \leq 2/\theta$. Then, the right hand side of \eqref{eq:gw delayed rec ineq} can be bounded by
	\begin{equation*}
	\begin{split}
	&\frac{1}{\theta} \E_{D\sim \xi}\left[\exp \left(\lambda \E\left[S_{L-1}^\theta \right]D \right) \right]
	=
	\frac{1}{\theta} \E_{D\sim \xi}\left[\exp \left(\frac{\lambda \E\left[S_{L-1}^\theta \right]}{c} \cdot cD \right) \right]\\
	&\quad \leq \frac{1}{\theta}
	\exp\left\{\log M \left(\frac{\lambda}{c}\E\left[S_{L-1}^\theta \right] \right)  \right\}
	\leq \frac{1}{\theta} M^{\frac{2}{cK}}\leq\frac{2}{\theta},
	\end{split}
	\end{equation*}
	where the first inequality is due to Jensen's inequality. Therefore, for $K,\lambda_0$ as above,
	we have $\E S_L^\theta \leq 2/\theta$ for all $\lambda\leq\lambda_0$ with $\theta=K\lambda$. Finally, note that  $\nu_{\mathcal{T}_L}^\theta (\zero)$ is given by
	\begin{equation*}
	\nu_{\mathcal{T}_L}^\theta (\zero)  
	=\frac{1}{1+\lambda  \E[S_L^\theta | \mathcal{T}_L ]}.
	\end{equation*}
	Thus, we obtain the desired conclusion by taking expectation over its reciprocal and plugging in the estimate $\E S_L^\theta \leq 2/\theta$.
\end{proof}

\begin{proof}[Proof of Theorem \ref{thm:decay of depth}]
	Let $L>0$ be an arbitrary integer and let $\mathcal{T}_L \sim \gw(\xi)_L$. Also, let $K, \lambda_0$ be the constants given by Lemma \ref{prop:dcp on gw}, and let $\pi_{\mathcal{T}_L}$, $\nu_{\mathcal{T}_L}^\theta$ be the stationary distributions of $\cp^\lambda_{\rho^+}(\mathcal{T}_L^+)$ and $\ddp^{\lambda,\theta}_{\rho^+} (\mathcal{T}_L^+)$, respectively, with $\theta = K\lambda$.
	
	Set $\Omega = \{0,1\}^{\mathcal{T}_L}$, and define $$A:= \{x\in \Omega: r(x;\mathcal{T}_L^+) \geq h \}.$$
	We first observe that
	\begin{equation*}
	\frac{\pi_{\mathcal{T}_L}(A)}{\pi_{\mathcal{T}_L}(\zero)} = \frac{\sum_{x\in A} \theta^{r(x;\mathcal{T}_L^+)} \widetilde{\nu}_{\mathcal{T}_L}^\theta(x) }{\widetilde{\nu}_{\mathcal{T}_L}^\theta(\zero) }
	\leq
	\frac{\widetilde{\nu}_{\mathcal{T}_L}^\theta(A)}{\widetilde{\nu}_{\mathcal{T}_L}^\theta(\zero)} \theta^h.
	\end{equation*}
	Proposition \ref{prop:dcp on gw} and Markov's inequality imply that with probability $1-m^{-1}$ over the choice of $\mathcal{T}_L$, we have $\widetilde{\nu}_{\mathcal{T}_L}^\theta(\zero)^{-1}\leq 2m $, and hence for such choices
	\begin{equation}\label{eq:decay of stationary msr}
	\frac{\pi_{\mathcal{T}_L}(A)}{\pi_{\mathcal{T}_L}(\zero)} \leq 2m (K\lambda)^h.
	\end{equation}
	Moreover, if $(X_t)\sim \cp^\lambda_{\rho^+}(\mathcal{T}_L^+)$ hits $A$, then the expected time needed for $X_t$ to escape from $A$ is at least $1$. Indeed, it takes a unit expected time just to heal one infected site of depth at least $h$. In other words, if we set $S_L,H $ as in the statement and define $\gamma_L(h) := |\{t\in [0,S_L] : X_t \in A \}|$ where $|\cdot|$ denotes the Lebesgue measure, then
	\begin{equation*}
	\E [\gamma_L(h) \,|\,H\geq h,\, \mathcal{T}_L ] \geq 1.
	\end{equation*}
	Combining this with (\ref{eq:decay of stationary msr}) tells us that
	\begin{eqnarray*}
		\P(H\geq h \,|\,\mathcal{T}_L) &\leq\;&
		\E[\gamma_L(h)\,|\, H\geq h,\mathcal{T}_L]\, \P(H\geq h \,|\,\mathcal{T}_L) \leq\;
		\E [\gamma_L(h)\,|\,\mathcal{T}_L]\nonumber\\
		& \leq\;& \frac{\pi_{\mathcal{T}_L}(A)}{\pi_{\mathcal{T}_L}(\zero)}
		\;\leq 2m (K\lambda)^h,
	\end{eqnarray*}
	with probability $1-m^{-1}$ over the choice of $\mathcal{T}_L$. 
\end{proof}

We conclude this section by completing the proof of Theorem \ref{thm:gw main}.

\begin{proof}[Proof of Theorem \ref{thm:gw main}]  Let $K,\lambda_0$ be given as Theorem \ref{thm:decay of depth}, and set $\lambda \leq \lambda_0$ to be a constant such that $K\lambda<1$. Let $\delta>0$ be any given small number, and set $h$ to be the constant satisfying $(K\lambda)^h=\frac{\delta^2}{8}$. Further, let $\mathcal{T}\sim \gw(\xi)$ and $\rho$ be its root.
	
	Define $E(h)$ to be the event that the infection inside $\cp^\lambda(\mathcal{T};\one_{\rho})$ does not go deeper than depth $h$ until dying out. Then, Theorem \ref{thm:decay of depth} implies that $$\mathbb{P}(E(h)) \geq 1-\delta, $$
	which can be seen by setting $m=\frac{2}{\delta}$.
	
	Let $\mathcal{T}_h$ be the truncated tree of $\mathcal{T}$ at depth $h$, and couple the processes $\cp^\lambda(\mathcal{T};\one_\rho)$ and $\cp^\lambda(\mathcal{T}_h;\one_\rho)$ by identifying the recoveries and infections inside $\mathcal{T}_h$. Then, on $E(h)$, $\cp^\lambda(\mathcal{T};\one_\rho)$ can be regarded as $\cp^\lambda(\mathcal{T}_h;\one_\rho)$.  Let $R$ and $R_h$ be the times when $\cp^\lambda(\mathcal{T};\one_\rho)$ and $\cp^\lambda(\mathcal{T}_h;\one_\rho)$ reaches $\zero$.
	Then, Theorem \ref{thm:cp on gw} tells us that
	\begin{equation*}
	\E[R|E(h)] = \E[R_h | E(h)] \leq \frac{\E R_h}{\P (E(h))} <\infty.
	\end{equation*}
	Thus, for $(X_t)\sim \cp^\lambda(\mathcal{T};\one_\rho)$, we have
	$$\P (X_t \neq \zero ~\textnormal{for all } t\geq 0) \leq \delta.$$ Since this holds true for all $\delta>0$, we conclude that $\lambda_1(\gw(\xi)) \geq \lambda >0$.
\end{proof}

\section{Short survival in random graphs}\label{sec:sub rg}

We turn our attention to the contact process on random graphs $G \sim \mathcal{G}(n,\mu)$. Throughout the rest of the paper, $\mu$ is a probability distribution on $\mathbb{N}$ that satisfies for $D\sim \mu$,
$$\sigma^2:= \E D^2 <\infty \quad \textnormal{and}\quad b :=\frac{\E D(D-1)}{\E D}>1, $$ as discussed in Section  \ref{subsec: rg basic}. In this section, in particular, we assume that $\mu$ has an exponential tail, i.e., $\E\exp(cD)<\infty$ for some $c>0$. Our goal is to  establish Theorem \ref{thm: exp tail main}-(1), by proving the following:

\begin{theorem}\label{thm:sub}
	Let $\mu$ be as above and $G \sim \mathcal{G}(n,\mu)$. For a vertex $v\in G$,  let $T_v$ denote the time when $\cp^\lambda (G; \one_v )$ reaches  the state $\zero$. Then there exist  events $\mathcal{E}_1$, $\mathcal{E}_2(G)$, and constants $B, \lambda_0>0$ depending on $\mu$ such that the following hold:
	\begin{itemize}
		\item $\mathcal{E}_1$ is an event over the random graphs such that $\P(G\in \mathcal{E}_1) = 1-o(1)$.
		
		\item $\mathcal{E}_2(G)$ is an event over the contact process $\cp^\lambda(G)$ such that $$\P(\mathcal{E}_2 \,|\, G\in \mathcal{E}_1 )= 1-o(1) .$$
		
		\item For all $\lambda \in (0, \lambda_0)$ we have
		\begin{equation*}
		\frac{1}{n}\sum_{v\in G} \E [T_v | G\in \mathcal{E} ] \leq B,
		\end{equation*}
		for all large enough $n$.
	\end{itemize}
\end{theorem}

\noindent Then, our main theorem follows simply by applying Markov's inequality.

\begin{proof}[Proof of Theorem \ref{thm: exp tail main}-(1)] Let $T$ be the time when $\cp^\lambda(G;\one_G)$ reaches the state $\zero$.  On the event $\mathcal{E}_1$ and $\mathcal{E}_2=\mathcal{E}_2(G)$ given in Theorem \ref{thm:sub}, for any constant $C>0$ we have
	\begin{equation*}
	\begin{split}
	\P& \left( \left. T > Cn \,\right|\, \mathcal{E}_1,\, \mathcal{E}_2  \right)
	=
	\P \left( \left. \max_{v \in G} T_v > Cn \,\right|\,\mathcal{E}_1,\, \mathcal{E}_2  \right)\leq
	\sum_{v\in G} \P \left( T_v >Cn \,|\, \mathcal{E}_1,\, \mathcal{E}_2 \right)
	%\\
	%&\leq
	%n\P \left(T_v > Cn \,|\, \mathcal{E} \right) 
	\leq \frac{B}{C},
	\end{split}
	\end{equation*}
	where  the first equality is due to Lemma \ref{lem:graph rep1}, and the second is from Markov's inequality. Since the events $\mathcal{E}_1$ and $\mathcal{E}_2$ given $G\in\mathcal{E}_1$ both hold \textsf{whp}, we obtain the conclusion.  \end{proof}

In the rest of the section we focus on proving Theorem \ref{thm:sub}. Our proof relies much on the fact that the local neighborhood $N(v,L):=\{u\in G: \textnormal{dist}(u,v)\leq L \}$ of a fixed vertex $v$ roughly looks like a Galton-Watson branching process. Hence the results from Section  \ref{sec:sub gw} will play a huge role in this section as well.

However, since $\mathcal{G}(n,\mu)$ contains  cycles with nontrivial probability, we introduce a variant of Galton-Watson trees that can cover the effect of cycles in $\mathcal{G}(n,\mu)$, and develop a delicate coupling argument with the local neighborhood $N(v,L)$. This new branching process will stochastically dominate  $N(v,L)$ in terms of isomorphic embeddings of graphs, and hence the contact process will survive for a longer time. The result will then follow by showing Theorem \ref{thm:sub} for this new graph.

\subsection{Coupling the local neighborhood}

%Since the local neighborhood of a given vertex $v$ in $G\sim \mathcal{G}(n,\mu)$ is closely related with  Galton-Watson trees, understanding the behavior of contact process on $\mathcal{T}_L$ forms the basis of our work.

Let $G\sim \mathcal{G}(n,\mu)$, where $\mu$ has an exponential tail, and let $\mu'$ denote the size-biased distribution of $\mu$. As discussed in Section  \ref{subsec: rg basic}, it is well known that the local neighborhood $N(v,L)$ around $v\in G$ behaves roughly as the Galton-Watson process $\gw(\mu,\mu')_L$. However, the standard coupling between the two object produces an error at least $\Theta(n^{-1})$. Therefore, we consider augmented versions of $\mu, \mu'$ to stochastically dominate $N(v,L)$ by a larger geometry.

\begin{definition}[Augmented distribution]\label{def:aug}
	Let $\mu$ be a probability distribution on $\mathbb{N}$ with an exponential tail. Let $k_0 = \max \{k:\sum_{j\geq k}\sqrt{p_j} \geq 1/2 \}$, and $k_{\textnormal{max}} :=\max\{k: p_k >0 \} $, with $k_{\textnormal{max}}=+\infty$ if the maximum does not exist. When $k_0< k_{\textnormal{max}}$, we define the augmented distribution $\mu^\sharp$ of $\mu$ by
	\begin{equation*}
	\mu ^\sharp(j) = \frac{1}{Z}
	\begin{cases}
	p_j/2 & \textnormal{if} ~j \leq k_0;\\
	\sqrt{p_j} & \textnormal{if} ~j > k_0,
	\end{cases}
	\end{equation*}
	where $Z= \sum_{j\leq k_0} p_j/2 + \sum_{j > k_0} \sqrt{p_j}.$ If $k_0=k_{\textnormal{max}}$, then we let
	\begin{equation*}
	\mu ^\sharp(j) = \frac{1}{Z}
	\begin{cases}
	p_j/2 & \textnormal{if} ~j < k_0;\\
	\sqrt{p_j} & \textnormal{if} ~j = k_0,
	\end{cases}
	\end{equation*}
	where $Z=  \sum_{j<k_0} p_{j}/2 +  \sqrt{p_{k_0}}.$
	
	%\item[(2)] Suppose that $2\leq |\supp (\mu)| <\infty$. Let $k_0 = \max \{k: \sum_{j\leq k} p_j < 1/2 \}.$. Then we define
	%\begin{equation}
	%\P(D_+ =j) = \frac{1}{Z}
	%p_j/2 & \textnormal{if} ~j\leq k_0;\\
	%p_j & \textnormal{if} ~j>k_0,
	%\end{cases}
	%\end{equation}
	%where $Z= \sum_{j\leq k_0} p_j/2 + \sum_{j\leq k_0}p_j.$
	
\end{definition}

We observe some of the basic properties of augmented distributions in the following lemma. The proof is based on  elementary applications of estimating large deviation events, and is postponed to Appendix (Section \ref{app:aug}) since it is a bit technical and less related with the main theme of the work.

\begin{lemma}\label{lem:aug}
	Let $\mu$ be a probability distribution on $\mathbb{N}$.
	\begin{enumerate}
		\item[(1)] If $\mu$ has an exponential tail, then so does $\mu^\sharp$.
		
		\item[(2)] 	Let $D_1,\ldots,D_n$ be $n$ independent samples of $\mu$. For a subset $\Delta \subset [n]$, let $\{p_k^\Delta \}_k$ denote the empirical distribution of $\{D_i \}_{i\in [n]\setminus \Delta}$. With high probability over the choice of $D_i$'s, $\{p_k^\Delta \}_k $ is stochastically dominated by $\mu^\sharp,$ for any $\Delta \in [n]$ with $|\Delta| \leq n/3$
	\end{enumerate}
\end{lemma}

\begin{remark}
	The i.i.d $D_i $ in the second condition of Lemma \ref{lem:aug} can be viewed as a degree sequence of $G \sim \mathcal{G}(n,\mu)$. Consider the exploration procedure starting from a single fixed vertex $v$, which, at each step, reveals a vertex adjacent to the current explored neighborhood and the half-edges incident to the new vertex. Then the second statement says that when the exploration process revealed $N \leq n/3$ vertices inside the local neighborhood of $v$, the empirical degree distribution of the $n-N$ unexplored vertices is stochastically dominated by $\mu^\sharp$, with high probability.
\end{remark}

Using the above properties of augmented distributions, we develop a coupling argument to dominate $N(v,L) \subset G$ by a Galton-Watson type branching process. To this end, we first take account of the effect of emerging cycles in $N(v,L)$.

For a constant $\gamma>0$, let $\mathcal{A}(\gamma)$ be the event that $N(v,\gamma \log n)$ in $G$ contains at most one cycle   for all $ v\in G $. The following lemma shows that we typically have $\mathcal{A}(\gamma)$ for some constant $\gamma$.

\begin{lemma}\label{lem:1cyc}
	There exists $\gamma=\gamma(\mu)>0$ such that for $G\sim \mathcal{G}(n,\mu)$, $\P(G\in \mathcal{A(\gamma)})  = 1-o(1)$.
\end{lemma}

This is a well-known property that holds true in general for various types of random graphs. Our proof of this lemma will be very similar to that of Lemma 2.1 in \cite{ls10}. However, it is more technical due to generality of the model and hence we postpone the proof to Section \ref{app:treeexcess}.

Fix a constant $\gamma_1>0$ satisfying the condition in Lemma \ref{lem:1cyc}, and let $\mathcal{A}=\mathcal{A}(\gamma_1)$ for convenience. In the following, we define two Galton-Watson type branching processes, which are used to stochastically dominate $N(v,\gamma_1\log n)$.

\begin{definition}[Galton-Watson-on-cycle process] Let $s, L$ be positive integers with $s\geq 2$, and let $\xi$ be a probability distribution on $\mathbb{N}$. We define the \textit{Galton-Watson-on-cycle process} (in short, \textit{GWC-process}), denoted by $\gwc(\xi; s)_L$, as follows:
	\begin{enumerate}
		\item Let $C$ be a cycle of length $s$, and distinguish one vertex as the root $\rho$.
		
		\item On $C$, we add $(s-1)$ independent $\gw(\xi)_L$ trees, each rooted at a vertex of $C$ except for $\rho$.
	\end{enumerate}
	The vertex $\rho$ is called the root of $\gwc(\xi; s)_L$.
\end{definition}

\begin{definition}[Edge-added Galton-Watson process]
	Let $l,s,L$ be positive integers with $s\geq 2$ and $l\leq L$, and let $\xi$ be a probability distribution on $\mathbb{N}$. We define  $\egw(\xi;l,s)_L$, the edge-added Galton-Watson process (in short, EGW-process) as follows:
	\begin{enumerate}
		\item Generate a $\gw(\xi)_L$ tree, conditioned on survival until depth $l$.
		
		\item At each vertex $v$ at depth $l$, add an independent $\gwc(\xi;s)_{L-l}$ process rooted at $v$. Here we preserve the existing subtrees from $v$.
	\end{enumerate}
	Let $\xi'$ be another probability measure on $\mathbb{N}$. Then $\egw(\xi,\xi';l,s)_L$ denotes the EGW-process where the root has degree distribution $\xi$, and all the descendants have $\xi'$. Here we also add $\gwc(\xi';s)_{L-l}$ in the second step of the definition.
\end{definition}

We now develop an argument showing that the local neighborhood $N(v,L)$ is dominated by a combined law of EGW-processes. In what follows, for two probability measures $\nu_1$ and $\nu_2$ on graphs, we say $\nu_1$ \textit{stochastically dominates} $\nu_2$ and write $\nu_1 \geq_{st} \nu_2$ if there exists a coupling between $S_1 \sim \nu_1$ and $S_2 \sim \nu_2$ such that $S_2 \subset S_1 $ in terms of isomorphic embeddings of graphs, i.e., there exists an injective graph homomorphism from $S_2$ into $S_1$.

Fix a vertex $v \in G$, and consider its local neighborhood $N(v,L_n)$ where $L_n= \gamma_1 \log n$ with $\gamma_1$  as in Lemma \ref{lem:1cyc}. For each $l,s$ with $s\geq 2$, we define the event $\mathcal{B}_{l,s}(v)$ to be the subevent of $ \mathcal{A}$ such  that in addition to $\mathcal{A}$, $N(v, L_n)$ forms a cycle of length $s$ at distance $l$ from $v$.

For the given degree distribution $\mu$, let $\mu'$ be its size-biased distribution, and $\widetilde{\mu} := \mu'_{[1, \infty)}$ denote the distribution $\mu'$ conditioned on being in the interval $[1, \infty)$. Let $\mu^\sharp$ and $\widetilde{\mu}^\sharp$ be the augmented distributions of $\mu$ and $\widetilde{\mu}$, respectively. Further, let $\eta$, $\eta_{l,s}$ and  $\eta_0$ denote the probability measures on rooted graphs describing the laws of $N(v, L_n)$, $\egw(\mu^\sharp, \widetilde{\mu}^\sharp;l,s)_{L_n}$ and $\gw(\mu^\sharp, \widetilde{\mu}^\sharp)_{L_n}$, respectively.

\begin{lemma}\label{lem:nbdcoupling}
	Under the above setting, for a fixed vertex $v\in G$ we have the following stochastic domination:
	\begin{equation*}
	\eta \one_{\mathcal{A}} \leq_{st} \sum_{l,s:s\geq 2} b_{s,l} \eta_{s,l} + b_0 \eta_0,
	\end{equation*}
	where $b_{l,s}=\P(\mathcal{B}_{l,s}(v))$, $b_0 = 1-\sum_{l,s} b_{s,l}$.
\end{lemma}

\begin{proof}
	% Let $v$ be an arbitrary given vertex and let $s\geq 2, l$ be any integers. Further, let $v_0=v, v_1, \ldots, v_{l-1}$ and $u_1, \ldots, u_s$ be any distinct vertices in $G$, and let $U_{s,l} $ denote the collection of all these vertices. Define $\mathcal{B}(U_{s,l}) $ to be the subevent of $\mathcal{A}$ that the path $v_0v_1\ldots v_{l-1}u_1$ and the cycle $u_1\ldots u_s u_1$ are included in $N(v,L_n)$. (the case $s=2$ implies a multi-edge between $u_1$ and $u_2$.) Then it suffices to show that $\eta$ conditioned on $\mathcal{B}(U_{s,l})$ is stochastically dominated by $\eta_{s,l}$. This will follow from Lemma \ref{lem:aug}, based on a rather standard way of coupling between the two objects.
	
	We study $N(v,L_n)$ from an exploration procedure point of view, in terms of the breadth-first search algorithm. Initially before exploring anything, we have $n$ vertices with each of them having i.i.d. $\mu$ half-edges. The term ``explore" means that we match a pair of half-edges and form an edge between their endpoint vertices. For convenience, we initially impose an arbitrary ordering on all half-edges before exploring anything. We consider the following exploration procedure:
	%Note that $\mathcal{B}(U_{s,l})$ implies that we have already explored the path $vv_1\ldots v_{l-1}u_1$ and the cycle $u_1\ldots u_s$.
	\begin{itemize}
		\item We start from the single vertex $v$ and the half-edges adjacent to it.
		
		\item Suppose that we explored up to depth-$t$ neighborhood of $v$. Let $\partial N(v,t)$ denote the unmatched half-edges on the boundary of $N(v,t)$, and we explore the half-edges in $\partial N(v,t)$ one by one, respecting the aforementioned ordering. During the $(t,i)$-th exploration step for $1\leq i \leq |\partial N(v,t)|$,  the $i$-th half-edge in $\partial N(v,t)$ is paired with a uniformly random unexplored half-edge.

	\end{itemize}
	
	Let $N(v,t;i)$ denote the explored neighborhood until $(t,i)$-th exploration step.  Also, let $H_t = |\partial N(v,t)|$. During the $(t,i)$-th exploration step, the $i$-th half-edge of $\partial N(v,t)$ seeks for its uniformly random pair from the unexplored half-edges. Therefore, if we have yet explored fewer than $\frac{n}{3}$ vertices, then after pairing a half-edge, the number of newly added half-edges to $N(v,t;i)$ from $N(v,t;i-1)$ is stochastically dominated by $\widetilde{\mu}^\sharp$, due to Lemma \ref{lem:aug}. This implies that conditioned on the event that $N(v,t;i)$ does not contain any cycles, $N(v,t;i)$ is stochastically dominated by $\mathcal{T}_{t,i} $, where $\mathcal{T}_{t,i}$ is generated by adding new offsprings according to $\widetilde{\mu}^\sharp$ to $i$ vertices of depth $t$, to the Galton-Watson tree $\mathcal{T}_t \sim \gw(\mu^\sharp, \widetilde{\mu}^\sharp)_t$.
	
	Define $(T,I)$ to be the index of the exploration step when a cycle is formed. In other words, the $I$-th half-edge in $\partial N(v,T)$ is either paired to a $j$-th half-edge in $\partial N(v,T)$ for some $j>I$ or to one of the newly explored half-edges during the $(T,k)$-th exploration step for some $k<I$.
	Note that on the event $\mathcal{A}$, either the  unique $(T,I)$ exists or it does not exist up to exploring $N(v,L_n)$.
	
	Suppose that there exists unique valid $(T,I)$. Let $C$ be the cycle formed at this step and $v(C)$ be the vertex in $C$ that is closest to $v$. Up to the $(T,I-1)$-th exploration step, we can stochastically dominate $N(v,T;I-1)$ by $\mathcal{T}_{T,I-1}$ as mentioned above. Let $w(C)$ be the vertex in $\mathcal{T}_{T,I-1}$ corresponding to $v(C)$ via an isomorphic embedding of $N(v,T;I-1)$ into $\mathcal{T}_{T,I-1}$. At $(T,I)$-th exploration step, we add $\mathcal{S} \sim \gwc(\widetilde{\mu}^\sharp;|C|)_{L_n}$ at $w(C)$. Note that this GWC-process $\mathcal{S}$ can be coupled with $C$ and its descendants in $N(v,L_n)$ in the sense that each Galton-Watson subtree hanging to the cycle of $\mathcal{S}$ stochastically dominates the corresponding subtree in $N(v,L_n)$ hanging to $C$.
	
	Let $l$ denote the distance from $v$ to $v(C)$. Completing the rest of the exploration as discussed above, $N(v,L_n)$ is stochastically dominated by $\egw ( \mu^\sharp, \widetilde{\mu}^\sharp; l, |C|)_{L_n}$, given that there exists the unique valid $(T,I)$. This implies that on the event $\mathcal{A}$, $N(v,L_n)$ is stochastically dominated by a combined law of EGW-processes, and in this  combination, the probability mass of appearance of $\egw(\mu^\sharp, \widetilde{\mu}^\sharp;l,s)_{L_n}$ should be $b_{l,s}=\P(\mathcal{B}_{l,s})$. This concludes the proof of the claimed result.
\end{proof}

\subsection{Estimating the survival time}
Thanks to Lemma \ref{lem:nbdcoupling}, we now study the contact process on edge-added Galton-Watson processes. On such graphs, we first show that the expected survival time of the contact process is bounded by a constant when the infection rate is small enough, as an analog of Theorem \ref{thm:cp on gw}.

\begin{proposition}\label{prop:cp on egw}
	Let $l,s,L$ be any integers such that $s\geq 2$ and $L\geq l$. Let $R_{l,s,L}$ denote the first time when $\cp^\lambda (\mathcal{S};\one_{\rho})$ reaches at state $\zero$, where $\mathcal{S} \sim \egw(\mu^\sharp, \widetilde{\mu}^\sharp;l,s)_L$ rooted at $\rho$. Then there exist constants $C, \lambda_0>0$ depending only on $\mu$ such that for all $\lambda \leq \lambda_0$, $s, l$ and $ L$, we have	$\E R_{l,s,L} \leq C$.
	
\end{proposition}

\begin{remark}
	Since the coupling given in Lemma \ref{lem:nbdcoupling} only works until depth $L_n=\gamma_1 \log n$, we need to show that the contact process on edge-added Galton-Watson process does not go deeper than $\gamma_1 \log n $ with probability $1-o(n^{-1})$. Note that the $o(n^{-1})$ error is needed when applying a union bound over all vertices in order to translate our results to $G\sim \mathcal{G}(n,\mu)$. This will be done in the next section based on Theorem \ref{thm:decay of depth}.
\end{remark}

To establish Proposition \ref{prop:cp on egw},  we develop a recursive argument on both $s$ and $L$ to deduce an analog of Lemma \ref{lem:cp on gw} for GWC- and EGW-processes.  The idea will be similar to that of Lemma \ref{lem:cp on gw}, which is to utilize the notion of \textit{root-added contact process} (Definition \ref{def:modified cp}). We first extend the result of Lemma \ref{lem:cp on gw} to the case of GWC-processes: in the following lemma, we estimate the time that the \textit{root-added} contact process $\cp^\lambda_\rho (\mathcal{S}) $ reaches $\zero$, where $\mathcal{S} \sim \gwc (\widetilde{\mu}^\sharp;s)_L$ (note that the state space is now $\{0,1 \}^{\mathcal{S}\setminus \{\rho\}}$). Here, we fix the root $\rho$ of $\mathcal{S}$ to be the permanently infected parent. There is a slight difference from the previous root-added contact processes considered in Lemma \ref{lem:cp on gw}, since now the permanently infected parent has two children rather than one. We pick a child $v$ of $\rho$ and study $\cp^\lambda_\rho (\mathcal{S};\one_v)$.

\begin{lemma}\label{lem:cp on gwc}
	Let $s,L$ be any integers with $s\geq 2$ and let $\mathcal{S} \sim \gwc(\widetilde{\mu}^\sharp;s)_L$ be a GWC-process rooted at $\rho$, with $\widetilde{\mu}^\sharp$ as before. Let $v$ be any neighbor of $\rho$, and let $S_{s,L}$ denote the first time when $\cp^\lambda_{\rho}(\mathcal{S}; \one_v)$ reaches at state $\zero$. Then there exists a constant $ \lambda_0 >0$ depending only on $\mu$ such that for any $\lambda \leq \lambda_0$ and $s, L$, $\E S_{s, L }\leq 2e$.
\end{lemma}

\begin{proof}
	Let us first study the case of $s\geq 3$. We will develop an inductive argument on $s$, similarly as in Lemma \ref{lem:cp on gw}. Let $(X_t) \sim \cp^\lambda_\rho (\mathcal{S};\one_v)$, and consider the modified version $(\widetilde{X}_t) \sim \widetilde{\cp}^\lambda_{\rho;v} (\mathcal{S};\one_v)$ of $(X_t)$, defined as
	\begin{enumerate}
		\item $(\widetilde{X}_t)$ is coupled with $(X_t)$ in the sense that they share the same locations of recoveries and infections. In particular, $\rho$ is infected in $(\widetilde{X}_t)$.
		
		\item In $(\widetilde{X}_t)$, the recovery at $v$ at time $s$ is valid if and only if $\widetilde{X}_s =  \one_v$. Otherwise, we ignore the recovery at $v$. In other words, when there exists an infected vertex other than $\rho$ and $v$, the recovery at $v$ is neglected.
	\end{enumerate}
	The modified process $\widetilde{\cp}^\lambda_{\rho;v} (\mathcal{S};\one_v)$ plays the same role as the $\widetilde{\cp}$-process introduced in the proof of Lemma \ref{lem:cp on gw}, which we now detail. Let $D\sim \widetilde{\mu}^\sharp$ to satisfy $\deg(v)=D+2$, and let $u_1 ,\ldots, u_{D}$ be the neighbors of $v$ which are not on the cycle of $\mathcal{S}$. Let $\mathcal{T}_{u_i}$ denote the subtrees branching from $u_i$, which has the law of i.i.d $\gw(\widetilde{\mu}^\sharp)_{L-1}$, and regard $v\in \mathcal{T}_{u_i}^+$ as the permanently infected parent of $u_i$. Further, call $ \mathcal{S}' = \mathcal{S}\setminus \cup_{i=1}^D \mathcal{T}_{u_i}$, and define $\cp^\lambda_{\rho, v}(\mathcal{S}')$ to be the contact process on $\mathcal{S}'$ in which $\rho$ and $v$ are set to be  infected permanently.
	As we run the process $(\widetilde{X}_t)$ from $t=0$,
	
	\begin{enumerate}
		\item [\textbf{A}.]  The second state of $\widetilde{X}_t$ is $\zero$ with probability $\frac{1}{1+ (D+2)\lambda}$. Here, $D+2$ comes from $D+1$ possible new infections from $v$, and one possible infection from $\rho$ to its child other than $v$. When this happens,  the expected waiting time until the transition to $\zero$ is $\frac{1}{1+(D+2)\lambda}$.
		
		\item [\textbf{B}.] Otherwise, $\{\rho,v \}$ infects a uniformly random neighbor $U$ before $v$ is healed, and $(\widetilde{X}_t)$ then can be regarded as a product chain of $\cp^\lambda_{\rho, v}(\mathcal{S}')$ and $\{\cp^\lambda_v (\mathcal{T}_{u_i}^+): i\in [D] \}$ with initial state $\one_U $, until $\widetilde{X}_t$ returns back to $\one_v$. Denote this product chain by $\cp^\otimes_{\rho;v}(\mathcal{S})$ (whose state space is $\{0,1 \}^{\mathcal{S}\setminus \{\rho,v \}}$). Here, $U$ can be thought of the first infected vertex besides $\{\rho,v\}$ in $\cp^\otimes_{\rho;v}(\mathcal{S};\zero) $.
	\end{enumerate}
	Let $\widetilde{S}_{s,L}$ be the first time that $\widetilde{\cp}^\lambda_{\rho;v} (\mathcal{S};\one_v)$ becomes $\zero$, and let $S^\otimes$ denote the first time that $\cp^\otimes_{\rho; v}(\mathcal{S};\one_U)$ reaches all-healthy state except $\rho, v$. Then, similarly as in Lemma \ref{lem:cp on gw}, the above reasoning implies that
	\begin{equation}\label{eq:geometric trial dcp}
	\begin{split}
	\E \left[\left. \widetilde{S}_{s,L} \,\right|\, \mathcal{S} \right]
	&=
	\sum_{k=0}^\infty
	\left(\frac{(D+2)\lambda}{1+(D+2)\lambda} \right)^k
	\frac{1}{1+(D+2)\lambda} \times	\left[\frac{k+1}{1+(D+2)\lambda}
	+ k\,\E \left[\left. S^\otimes \, \right| \, \mathcal{S} \right] \right]\\
	&=1+ (D+2)\lambda \E  \left[\left. S^\otimes \, \right| \, \mathcal{S} \right].
	\end{split}
	\end{equation}
	Therefore, we have
	\begin{equation}\label{eq:gwc recursion}
	\E \left[\left. \widetilde{S}_{s,L} \,\right|\, D \right]
	=1+(D+2)\lambda \E \left[\left. S^\otimes \right| D \right].
	\end{equation}
	
	Now we take account of the stationary distributions of the above processes to obtain the conclusion. Let $\pi^\otimes, \,\pi'$ and $\pi_i$ be the stationary distributions of $\cp^\otimes_{\rho;v}(\mathcal{S}), \, \cp^\lambda_{\rho, v}(\mathcal{S}')$ and $\cp^\lambda_v (\mathcal{T}_{u_i}^+)$, respectively. Then clearly, $\pi^\otimes= (\otimes_{i=1}^D \pi_i)\otimes \pi' $. We can relate these objects with the running times similarly as (\ref{eq:stationary and running time 1}, \ref{eq:stationary and running time 2}), by
	\begin{equation}\label{eq:stationary and gwc}
	\begin{split}
	\pi^\otimes (\zero) &= \frac{1}{1+(D+2)\lambda \E[S^\otimes | \mathcal{S}] };\\
	\pi_i (\zero)&= \frac{1}{1+\lambda \E [S_{L-1}|\mathcal{T}_{u_i} ]} ;\\
	\pi'(\zero) &= \frac{1}{1+ 2\lambda \E[S_{s-1,L}| \mathcal{S}']},
	\end{split}
	\end{equation}
	where $S_{L-1}$  denotes the first time when $\cp^\lambda_v(\mathcal{T}_{u_i}^+; \one_{u_i})$ becomes $\zero$, and $S_{s-1,L}$ is the time it takes for $\cp^\lambda_{\rho,v}(\mathcal{S}'; \zero)$ to return to $\zero$ after the first infection besides $\{\rho,v\}$ occurs. Note that the existence of the infection other than $\{\rho,v\}$ in $\cp^\lambda_{\rho,v}(\mathcal{S}'; \zero)$ is guaranteed by the condition $s\geq 3$. Also, notice that $\mathcal{S}'$ can be regarded as  $\mathcal{S}''\sim \gwc(\widetilde{\mu}^\sharp;s-1)_L$, since the processes $\cp^\lambda_{\rho,v} (\mathcal{S}'; \one_w)$ (with $w$ being a neighbor of $\{\rho,v\}$ in $\mathcal{S}'$) and $\cp^\lambda_{\rho'}(\mathcal{S}'';\one_{v'}) $ for the root $\rho'$ and one of its neighbor $v'$ of $\mathcal{S}''$ share the same law. This implies that the notation $S_{s-1,L}$ in (\ref{eq:stationary and gwc}) matches with the definition of it given in the statement of the lemma.
	
	Therefore, combining (\ref{eq:gwc recursion}) and (\ref{eq:stationary and gwc}) gives that
	\begin{equation}\label{eq:gwc recur ineq}
	\E \left[S_{s,L} | D \right] \leq \E \left[\left. \widetilde{S}_{s,L} \right| D \right]
	\leq
	\left(1+ 2\lambda \E [S_{s-1,L}] \right) \left(1+\lambda \E [S_{L-1}] \right)^D.
	\end{equation}
	Since we already have a bound for $\E[S_{L-1}]$ due to Lemma \ref{lem:cp on gw}, we deduce the desired result by manipulating (\ref{eq:gwc recur ineq}),  as in the final step of the proof of Lemma \ref{lem:cp on gw}. Namely, for any $\lambda\leq \min\{ \lambda_0, (4e)^{-1} \}$ with $\lambda_0$ given as in Lemma \ref{lem:cp on gw}, (\ref{eq:gwc recur ineq}) gives us that
	\begin{equation*}
	\E[S_{s,L}] \leq e(1+2\lambda \E[S_{s-1,L}]),
	\end{equation*}
	and hence $\E[S_{s-1,L}] \leq 2e$ implies $\E[S_{s,L}]\leq 2e$.
	
	The case $s=2$ is simpler, since $\gwc(\widetilde{\mu}^\sharp;2)_L$ is the same as the law of $\mathcal{T}_L^+$ for $\mathcal{T}_L \sim \gw(\widetilde{\mu}^\sharp)_L$, except that the parent $\rho^+$ of $\rho$ in $\mathcal{T}_L^+$ is now connected with $\rho$ by two edges. From the contact process point of view, this means that the intensity of infection from $\rho^+$ to $\rho$ is $2\lambda$, and everything else is identical to the case of $\gw(\widetilde{\mu}^\sharp)_L$. Hence, the same proof as Lemma \ref{lem:cp on gw} can be replicated, and we obtain that there exists a constant $\lambda_0>0$ such that $\E[S_{2,L}]\leq 2e$ for all $\lambda\leq \lambda_0$. We leave the details of the proof to the reader.
\end{proof}

\begin{proof}[Proof of Proposition \ref{prop:cp on egw}] It can be proven by the same way as Lemma \ref{lem:cp on gwc}. For completeness, we present the proof in Appendix, Section \ref{app:cp on egw}.
\end{proof}

\subsection{Proof of Theorem \ref{thm:sub}}

Let $l,s,L$ be arbitrary integers with $s\geq 2$ and $L\geq l$, and consider an edge-added Galton-Watson process $\mathcal{S} \sim \egw(\mu^\sharp, \widetilde{\mu}^\sharp;l,s)_L $.

We can extend the result of Theorem \ref{thm:decay of depth} and Proposition \ref{prop:dcp on gw} to the case of edge-added Galton-Watson processes. The method will be the same as Proposition \ref{prop:cp on egw} and Lemma \ref{lem:cp on gwc}, appropriately adjusted to the current setting of \textit{delayed} contact process (Definition \ref{def:dcp}). We state the result in the following lemma, whose proof is deferred to Appendix (Section \ref{app:dcp on egw}), since it is similar to the previous proofs but more technical.

\begin{lemma}\label{lem:dcp on egw}
	Let $\mathcal{S} \sim \egw(\mu^\sharp, \widetilde{\mu}^\sharp;l,s)_L$, and $\nu^\theta_{\mathcal{S}}$ be the stationary distribution of $\ddp^{\lambda,\theta}_{\rho^+} (\mathcal {S}^+)$ on the space $\{0,1\}^\mathcal{S}$. Then there exist constants $C, \lambda_0>0$ depending only on $\mu$ such that for all $\lambda\leq \lambda_0$, we have $\E [\nu^\theta_{\mathcal{S}}(\zero)^{-1}] \leq 2$ for $\theta = C\lambda$.
\end{lemma}

Based on Proposition \ref{prop:dcp on gw} and Lemma \ref{lem:dcp on egw}, we have an analog of Theorem \ref{thm:decay of depth} for EGW-processes.  Thus, we can complete the proof of Theorem \ref{thm: exp tail main}-(1), by combining the previous results to build up a coupling between the contact processes on local neighborhood $N(v,L)$ and on EGW-processes.

\begin{proof}[Proof of Theorem \ref{thm:sub}]
	Let $G\sim\mathcal{G}(n,\mu)$ for any large enough $n$ and let $\gamma_1>0$ be the constant satisfying Lemma \ref{lem:1cyc}. For each $v\in G$,   let  $\mathcal{A}_v$ be the event that  $N_v:=N(v,\gamma_1 \log n)$ in $G$ contains at most one cycle. Then, the proof of Lemma \ref{lem:1cyc} tells us $\P(\mathcal{A}_v) \geq 1-o(n^{-1})$. 
	%Fixing such $\gamma_1$, let $\sum_{s,l}b_{s,l}\eta_{s,l} +b_0 \eta_0$ be the combined law of EGW-processes defined in Lemma \ref{lem:nbdcoupling}. Here, we can couple $N(v,\gamma_1\log n)$ and $\mathcal{S}_v\sim \sum_{s,l}b_{s,l}\eta_{s,l} +b_0 \eta_0$ such that $N(v,\gamma_1 \log n) \subset \mathcal{S}_v$.
	
	Let $\lambda_0'$ be the minimum between the $\lambda_0$'s given by Proposition \ref{prop:cp on egw} and Lemma \ref{lem:dcp on egw}, $C$ be as in \ref{lem:dcp on egw}, and $\theta = C\lambda$ for $\lambda \leq \lambda_0'$. Further, let $(X_t)\sim \cp^\lambda(N_v;\one_v)$ and $H_v:=\max\{r(X_t): t\geq 0 \}$, where $r(X_t)$ denotes the maximal depth among the infected sites in $X_t$. Note that $r(X_t)$ stays $0$ after $X_t$ becomes $\zero$.
	
	Define the event $\mathcal{B}_v$ as
	$$\mathcal{B}_v := \{\P(H_v\geq h \,|\, N_v) \leq 2n^2 (C\lambda)^h,~\textnormal{for all }h \} .$$
	Following the same proof as Theorem \ref{thm:decay of depth} based on Lemma \ref{lem:dcp on egw}, we have $\P(\mathcal{B}_v) \geq 1-n^{-2}$, by dominating $N_v$ by the EGW-processes as Lemma \ref{lem:nbdcoupling}.  Assume that $(C\lambda_0')^{\gamma_1\log n} \leq n^{-4}$ by making $\lambda_0'$ smaller if needed. Then, the event $\mathcal{C}_v$ given by
	$$\mathcal{C}_v= \mathcal{C}_v(G) := \{H_v<\gamma_1 \log n \},$$
	satisfies $\P(\mathcal{C}_v | \mathcal{A}_v \cap \mathcal{B}_v ) = O(n^{-2})$. Note that $\mathcal{A}_v \cap \mathcal{B}_v$ is an event over the random graph $G$, while $\mathcal{C}_v$ is an event over the contact process $(X_t)$ given the graph $G$. By the aforementioned coupling of $N_v$ and the EGW-process, Proposition \ref{prop:cp on egw} gives that
	\begin{equation}\label{eq:TvonABC}
	\E [T_v \,|\, G\in\mathcal{A}_v \cap \mathcal{B}_v, \,  \mathcal{C}_v] \leq B,
	\end{equation}
	for some constant $B=B(\mu)>0$.
	
	Define the events 
	$$\mathcal{E}_1 :=\cap_{v\in G} (\mathcal{A}_v \cap \mathcal{B}_v ), \quad \textnormal{and} \quad \mathcal{E}_2(G) := \cap_{v\in G} \mathcal{C}_v. $$
	Then, the above discussion shows that $\P(G\in \mathcal{E}_1)=1-o(1)$ and $\P((X_t) \in \mathcal{E}_2(G) | G\in \mathcal{E}_1 )  = 1-o(1)$,  and hence \eqref{eq:TvonABC} holds the same given $\mathcal{E}_1$ and $\mathcal{E}_2(G)$, namely,
	\begin{equation*}
	\E [T_v \,|\, G\in \mathcal{E}_1,\, \mathcal{E}_2(G)] \leq B,
	\end{equation*}
	under a possible modification of $B$ if needed. Therefore, by linearity of expectation, summing the above over all $v\in G$ gives the conclusion.
\end{proof}

\begin{proof}[Proof of Corollary \ref{cor: ER main}-(1)]  The statement follows immediately from the contiguity of $\mathcal{G}_{\textsf{cf}}(n,\textnormal{Pois}(d))$ and $\mathcal{G}_{n,d/n}$ (\cite{k06}, Theorem 1.1). To be precise,  for any subset $A_n$ of graphs with $n$ vertices,
	$$\P_{G\sim \mathcal{G}_{\textsf{cf}}(n,\mu)} (G\in A_n) \rightarrow 0 \quad \textnormal{implies}\quad \P_{G\sim \mathcal{G}_{n,d/n}} (G\in A_n) \rightarrow 0,$$
	where $\mu = \textnormal{Pois}(d)$. Since the statement is true \textsf{whp} for $G\sim \mathcal{G}_{\textsf{cf}}(n,\mu)$, the configuration model,  it is also true \textsf{whp} for  $G\sim\mathcal{G}_{n,d/n}$. 
\end{proof}

\section{Long survival in random graphs: Proof of Theorem \ref{thm: exp tail main}, part 2}\label{sec:proof:thm:ex}
\subsection{A structural lemma}
Our main tool to prove long survival time is the following structural lemma whose proof is deferred to Section \ref{sec:structure}. As mentioned in Section \ref{subsec:maintec}, we show that the random graph $G_n$ contains a large \textit{$(\alpha, R)$-embedded expander}. Once some subset of this expander is infected, it is likely to spread the infection over its $R$-neighborhood whose size more than doubles the original subset. We define an embedded expander as follows.

\begin{definition}[Embedded expander]
	For two positive numbers $\alpha$ and $R$, we say that a subset of vertices $W_0$ is an \textit{$(\alpha, R)$-embedded expander} of $G_n$ if  for every subset $A\subset W_0$ with $|A|\le \alpha |W_0|$, we have
	\begin{equation}\label{eq:expander}
	|N(A, R)\cap W_0|\ge  2|A|.
	\end{equation}
	where $N(A, R)$ is the collection of all vertices in $G_n$ of distance at most $R$ from $A$.
\end{definition}
The following lemma concerns the existence of such an \textit{$(\alpha, R)$-embedded expander} in the random graph $G_n$.
\begin{lemma} \label{lm:structure}
	Suppose that $\mu$ satisfies (\ref{eq:condition on mu}) and there exists some constant $c>0$ such that $\E_{D\sim \mu} e^{cD}<\infty$. Let $G\sim \mathcal{G}(n,\mu)$. There exist positive constants $\alpha, \beta, R, j$ such that the following holds \textsf{whp}. There exist a subgraph $\bar G_n$ of $G_n$ whose maximal degree is at most $2j$ and an $(\alpha, R)$-embedded expander $W_0$ of $\bar G_n$ with $|W_0|\ge \beta n $.  
\end{lemma}

\subsection{Proof  of Theorem \ref{thm: exp tail main}, part 2}
We first make a simple observation.

\begin{lemma}\label{lm:pathinfection:extail} Let $R$ be a positive integer constant. Consider the contact process $Y_t$ with infection rate $\lambda$ on a path of length at most $R$ connecting two vertices $v$ and $u$. Then there exist positive constants $C$ and $\lambda_0$ depending only on $R$ such that for all $\lambda\ge\lambda_0$, we have
	\begin{equation*}
	\P\left (u\in Y_{t+C}  \bigg| v\in Y_{t}\right ) \ge \frac{3}{4}.\nonumber
	\end{equation*}
\end{lemma}
\begin{proof}
	Let $C$ be a sufficiently large constant compared to $R$. Let $\mathcal A$ be the event that the infection on $N(v, 1)$ survives in the entire time interval $\left [t, t+C\right ]$.
	
	By \cite[Lemma 1.1]{cd09}, for sufficiently large $\lambda$ compared to $C$,
	\begin{equation*}
	\P\left (\mathcal A\bigg| v\in Y_{t} \right ) \ge \frac{99}{100}.\nonumber
	\end{equation*}
	Since there is a path of length at most $R$ from $v$ to $u$, by \cite[Lemma 2.4]{cd09},
	\begin{equation*}
	\P\left (u\in Y_{t''}\text{ for some $t''\in [t', t'+R]$}\bigg| N(v, 1)\cap Y_{t'} \neq \emptyset\right ) \ge \frac{1}{e^{6R}},\nonumber
	\end{equation*}
	and so as $C$ is large compared to $R$, we have
	\begin{eqnarray}\label{key}
	\P\left (u\in Y_{t'}\text{ for some $t'\in [t, t+C+R]$}\bigg| v\in Y_{t}\right )   \ge\P\left (\Bin \left (\frac{C}{R}, \frac{1}{e^{6R}}\right )\ge 1\right )-\frac{1}{100}\ge \frac{98}{100} \nonumber.
	\end{eqnarray}
	%	Since $R\ll C$, we can replace the interval $[t, t+C+R]$ in the above inequality by the bigger interval $[t, t+2C]$.
	Assume that $u\in Y_{t'}$ for some $t'\in [t, t+C+R]$. Fix a neighbor $u'$ of $u$. By  \cite[Lemma 1.1]{cd09} again, the contact process on the edge $(u, u')$ survives in the entire interval $[t', t' + 3C]$ with probability at least $\frac{99}{100}$. 
	Since $t+C+R\in [t', t' + 2C]$, there is at least one clock ring in $[t+C+R, t'+3C]$ and the last clock ring before time $t+3C$ is an infection clock from $u'$ to $u$ rather than the recovery clock at $u$,  with probability at least $\frac{99}{100}$. If $u'$ is already infected at that time, $u$ will be infected. Otherwise, $u$ has already been infected and remains infected. In either case, $u$ is infected at time $t+3C$ with probability at least $\frac{96}{100}\ge \frac{3}{4}$. By replacing $C$ by $C/3$, we complete the proof.
\end{proof}

Let $\alpha, \beta, j, R$, $\bar G_n$ and $W_0$ as in Lemma \ref{lm:structure}.  
It suffices to show that the contact process $(X_t)$ on $\bar G_n$ with all vertices infected initially survives for $e^{\Theta(n)}$-time with probability at least $1-e^{-\Omega(n)}$ over the contact process. For the rest of this proof, all the vertices, edges, paths, and balls are of $\bar G_n$ unless otherwise noted.

Let $X^{0}_{t}= X_{t}\cap W_0$ be the collection of infected vertices of $W_0$ at time $t$. We show that, thanks to the expander property of $W_0$, with very high probability, after some time $C$, the number of infected vertices in $W_0$ increases.

\begin{lemma}\label{lm:azuma:1}	Let $C$ and $\lambda_0$ be the constants in Lemma \ref{lm:pathinfection:extail}.	There exists a positive constant $C'$ depending only on $j$ and $R$ such that for all $\lambda\ge \lambda_0$ and for every integer $a\in (0,  \alpha \beta n]$,
	\begin{equation}\label{eq:path:expectation}
	\P\left (\left |X^{0}_{t+C}\right |\le \frac{5}{4}a\bigg| |X^{0}_{t}|= a\right )\le 2\exp\left (-\frac{a}{C'}\right ).
	\end{equation}
\end{lemma}
\begin{proof} 	  We will use Azuma's inequality. Let $G_{n,t}$ be the induced subgraph of $\bar G_n$ on the set $\bigcup_{v\in X^{0}_t} N(v, R)$. Let $(\hat X_{t'})_{t'\in [t, t+C]}$ be the contact process on $G_{n,t}$ with $\hat X_t := X^{0}_t$ (so $(\hat X_{t'})$ only uses the infection and recovery clocks of vertices and edges inside $G_{n, t}$). Let $\hat X_{t'}^{0} = \hat X_{t'}\cap W_0$. Let $\mathcal X$ be the infected vertices $u$ of $\hat X_{t+C}^{0}$ such that there exists $v\in X^0_t$ and a directed path of infection on the graphical representation of $(\hat X_{t'})$ from $(v, t)$ to $(u, t+C)$ and the vertices of the path lie entirely in $B (v, R)$. We have
	$$\mX\subset \hat X_{t+C}^{0}\subset X^{0}_{t+C}.$$
	
	Since the maximal degree in $\bar G_n$ is at most $2j$, the number of vertices, denoted by $a'$, in $G_{n, t}$ is at most $a(2j)^{R+1}$.
	Enumerate the vertices in $G_{n, t}$ by $v_1, \dots, v_{a'}$. For each $i=0, 1, 2, \dots, a'$, let $\mathcal F_i$ be the $\sigma$-algebra generated by the randomness of the recovery clocks and infection clocks during time $(t, t+C]$ on the vertices $v_1, \dots, v_i$ and edges connecting them. Let
	$$\mX_i := \E\left (\left |\mX \right |\bigg| \mathcal F_i, X^{0}_{t}, |X^{0}_t|=a\right ).$$
	We have $\mX_{a'} = \mX$. By Azuma's inequality, we have for every $s>0$,
	\begin{equation}\label{eq:azuma:subex}
	\P\left (\left |\mX_{a'}-\mX_0\right |\ge s\right )\le 2\exp\left (-\frac{s^{2}}{2a'K^{2}}\right )
	\end{equation}
	where
	\begin{equation}\label{key}
	K := \max_{j} ||\mX_j-\mX_{j-1}||_{\infty} \le \left |B\left (v_{i+1}, R\right )\right |\le (2j)^{R+1} \nonumber.
	\end{equation}
	From Lemma \ref{lm:pathinfection:extail} and the expander properties of $W_0$ as in Lemma \ref{lm:structure}, we obtain
	\begin{equation}\label{eq:path:expectation:2}
	\mX_0 = \E\left (\left |\mX\right |\bigg| |X^{0}_{t}|= a \right ) \ge \frac{3}{4}\left |N(X^{0}_{t}, R)\cap W_0\right |\ge   \frac{3a}{2}.
	\end{equation}
	Thus, by \eqref{eq:azuma:subex} for $s = \frac{a}{4}$ and the fact that $\mX_{a'}\le \left |X^{0}_{t+C}\right |$,  we obtain \eqref{eq:path:expectation}.
\end{proof}

\begin{proof}[Proof of Theorem \ref{thm: exp tail main}-(2)]
	For the lower bound on survival time, initially, all vertices in $W_0$ are infected so $|X^0_0|\ge \beta n$. Let $t_1$ be the first time that $|X^{0}_{t_1}| = \alpha \beta n$. At time $t_2 = t_1+C$, we have $|X^{0}_{t_2}| \ge \frac{5}{4} \alpha \beta n$ with probability at least $1 - \frac{2}{m^{2}}$ where $m:= \exp\left (\frac{\alpha \beta n}{2C'}\right )$ by Lemma \ref{lm:azuma:1}. Let $t_3$ be the first time after $t_2$ that $|X^{0}_{t_3}| = \alpha \beta n$ again. Repeating this process $m$ times, we get that the contact process survives until time $m C$ with probability at least $1 - 2/m$ by the union bound, proving the lower bound for Theorem \ref{thm: exp tail main}-(2).
	
	As for the upper bound, observe that for any time $t$, the probability that the contact process dies out during the time interval $[t, t+1]$ is at least the probability that for each vertex $v$ in $G_n$, at least one of the infection clocks from a neighbor $u$ of $v$ to $v$ or the recovery clock at $v$ rings in $[t, t+1]$ and the last clock rings before time $t+1$ is the recovery clock at $v$. Thus, the probability that the process dies out during $[t, t+1]$ is at least $\prod_{v} \frac{c}{\deg (v)}$. By Cauchy-Schwartz inequality and the fact that \textsf{whp}, the total degrees in $G_n$ is $O(n)$, we have $\prod_{v} \frac{c}{\deg (v)}\ge e^{-c'n}$ for some small constants $c, c'$. Therefore, \textsf{whp}, the contact process dies out before time $e^{2c'n}$. 
\end{proof}

\begin{remark}\label{rmk:onevertex:ex} To prove the corresponding result (Remark \ref{rmk:onevertex}) for $\mu$ having an exponential tail, when initially, there is only one uniformly chosen vertex $v$ infected in $G_n$, observe that with positive probability over the choice of $v$, $v$ belongs to $W_0$. Thus, it suffices to condition on this event and show that with positive probability over the contact process, the process survives until time $e^{cn}$ for some constant $c$. Let $\lambda_{0, R}$, $C_R$ and $C'_R$ be the constants $\lambda_0$, $C$ and $C'$ corresponding to $R$ in Lemmas \ref{lm:pathinfection:extail} and \ref{lm:azuma:1}, respectively. By \eqref{eq:expander}, for any bounded number $k$, we have
	$$|N(v, kR)\cap W_0| \ge 2^{k}$$
	for sufficiently large $n$.
	
	We now show that for sufficiently large $\lambda$, at some time, there will be a lot of infected vertices in $W_0$. This will then allow to take the union bound of the tail probability occurring in Lemma \ref{lm:azuma:1}. Let $k$ be a sufficiently large constant. Since the number of vertices in $N(v, kR)$ is at most $(2j)^{kR+1}=O_{j, k, R}(1)$, there are $O_{j, k, R}(1)$ edges in $N(v, kR)$. Thus, there exist constants $\lambda_{j, k, R}, t_{j, k, R}$ such that for all $\lambda\ge \lambda_{j, k, R}$, the probability that each vertex in $N(v, kR)$ is infected before time $t_{j, k, R}$ and that there are no recovery clocks ring before time $t_0$ is at least $3/4$. Hence, with probability at least $3/4$, there exists $t_1\le t_{j, k, R}$ at which all vertices in $N(v, kR)$ are infected. This implies $|X^{0}_{t_1}|\ge 2^{k}\ge (5/4)^{k}$. 
	
	Conditioning on this event and applying Lemma \ref{lm:azuma:1}, we get that with probability at least 
	\begin{equation*}
	1-\sum_{i=k}^{\infty} 2\exp\left (-\frac{5^{i}}{C'_{R}4^{i}}\right ),
	\end{equation*}
	there exists a time $t\ge t_1$ at which $|X^{0}_t|\ge \alpha \beta n$. Since $k$ is a sufficiently large constant, this probability is at least $1/2$. Finally, conditioned on this event, the same argument as in the proof of Theorem \ref{thm: exp tail main}-(2) shows that starting from this $t$, the contact process survives until time $e^{-\Omega(n)}$ \textsf{whp}. Altogether, the contact process $(X_t)$ starting from $v$ survives until time $e^{-\Omega(n)}$ with probability at least $1/4-o(1)$ for all $\lambda\ge \max\{\lambda_{0, R}, \lambda_{j, k, R}\}$ as desired.
\end{remark}

\begin{remark}
	Corollary \ref{cor: ER main}-(2) follows from Theorem \ref{thm: exp tail main}-(2) in the exact same way as we deduced Corollary \ref{cor: ER main}-(1)  from Theorem \ref{thm: exp tail main}-(1).
\end{remark}

\section{Long survival in random graphs: Proof of Theorem \ref{thm: subexp tail main}} \label{sec:proof:thm:subex}

In this section, we prove Theorem \ref{thm: subexp tail main} following the same strategy as in the proof of Theorem \ref{thm: exp tail main}-(2). The following structural lemma is an analog of Lemma \ref{lm:structure} for subexponential distributions. 
\begin{lemma} \label{lm:structure:subex}
	Suppose that $\mu$ satisfies (\ref{eq:condition on mu}) and $\E_{D\sim \mu} e^{cD}=\infty$ for all $c>0$. Let $G\sim \mathcal{G}(n,\mu)$. For any $\delta>0$,  there exist $\alpha, \beta, j, R>0$ with $R\le \delta j$ such that the following holds \textsf{whp}. There exist a subgraph $\bar G_n$ of $G_n$ whose maximal degree is at most $2j$ and an $(\alpha, R)$-embedded expander $W_0$ of $\bar G_n$ with $|W_0|\ge \beta n $ and $\deg_{\bar G_n} w\ge j/2$ for all $w\in W_0$. 
\end{lemma}

Fix $\lambda>0$. Let $\alpha, \beta, j, R$, $\bar G_n$, $W_0$ be as in Lemma \ref{lm:structure:subex}  where $\delta>0$ is a sufficiently small constant depending on $\lambda$. It suffices to show that the contact process $(X_t)$ on $\bar G_n$ with all vertices infected initially survives for $e^{\Omega(n)}$-time with probability at least $1-e^{-\Omega(n)}$ over the contact process. For the rest of the proof, all the vertices, edges, paths, and balls are of $\bar G_n$.

Let $X^{0}_{t}$ be the collection of infected vertices of $W_0$ at time $t$. We show the following analog of Lemma \ref{lm:pathinfection:extail}.

\begin{lemma}\label{lm:pathinfection:subex} There exists a constant $c$ depending only on $\lambda$ such that for sufficiently large $n$ and for every $u\in N(v, R)\cap W_0$, we have
	\begin{equation}\label{eq:pathinfection}
	\P\left (u\in X_{t+2e^{cj}}^{0} \bigg| v\in X^{0}_{t}\right ) \ge \frac{3}{4}.
	\end{equation}
\end{lemma}

Assuming Lemma \ref{lm:pathinfection:subex}, we prove the following analog of Lemma \ref{lm:azuma:1}.
\begin{lemma}\label{lm:azuma:2}	Let $c$ be the constant in Lemma \ref{lm:pathinfection:subex}.	There exists a positive constant $C'$ depending only on $j$ and $R$ such that for every integer $a\in (0,  \alpha \beta n]$,
	\begin{equation}\label{eq:path:expectation2}
	\P\left (\left |X^{0}_{t+e^{c j}}\right |\le \frac{5}{4}a\bigg| |X^{0}_{t}|= a\right )\le 2\exp\left (-\frac{a}{C'}\right ).
	\end{equation}
\end{lemma}
The proof of this lemma is identical to the proof of Lemma \ref{lm:azuma:1}. Using this lemma, the proof of Theorem \ref{thm: subexp tail main} is identical to that of Theorem \ref{thm: exp tail main}-(2). It remains to prove Lemma \ref{lm:pathinfection:subex}.  

\begin{proof}[Proof of Lemma \ref{lm:pathinfection:subex}]  Let $\Gamma$ be a path of length at most $R$ connecting $v$ and $u$. Since $v\in W_0$, $N_{v, u}:= \{v\}\cup\left (N(v, 1)\setminus \Gamma\right )$ contains a star with $j/4$ leaves. Let $\mathcal A$ be the event that the infection on $N_{v, u}$ survives in the entire time interval $\left [t, t+e^{c j}\right ]$ for some constant $c$ depending only on $\lambda$.
	
	By \cite[Lemma 5.3]{berger2005spread}, by choosing $\delta$ sufficiently small in Lemma \ref{lm:structure:subex} and using the inequality $j\ge R/\delta\ge 1/\delta$, we have
	\begin{equation}\label{key}
	\P\left (\mathcal A\bigg| v\in X^{0}_{t} \right ) \ge \frac{99}{100}.\nonumber
	\end{equation}
	Since there is a path of length at most $R$ from $v$ to $u$, by \cite[Lemma 2.4]{cd09}, there exists a constant $c'>0$ depending only on $\lambda$ such that for any time $t'$,
	\begin{equation}\label{key}
	\P\left (u\in X^{0}_{t''}\text{ for some $t''\in [t', t'+R+1]$}\bigg| N_{v, u}\cap X^{0}_{t'} \neq \emptyset\right ) \ge c'^{R+1}. \nonumber
	\end{equation}
	Thus,
	\begin{eqnarray}\label{key}
	&&\P\left (u\in X^{0}_{t'}\text{ for some $t'\in [t, t+e^{cj}+R+1]$}\bigg| v\in X^{0}_{t}\right ) \nonumber\\
	&&\qquad  \ge\P\left (\Bin \left (\frac{e^{cj}}{R}, c'^{R+1}\right )\ge 1\right )-\frac{1}{100}  \nonumber\\
	&&\qquad \ge\P\left (\Bin \left (\frac{e^{c R/\delta}}{R}, c'^{R+1}\right )\ge 1\right )-\frac{1}{100}\ge \frac{98}{100} \nonumber
	\end{eqnarray}
	by choosing $\delta$ sufficiently small (for example, $\delta = \frac{c}{100\log c'^{-1}}$) in Lemma \ref{lm:structure:subex}. Since $R\le \delta j$, we can replace the interval $[t, t+e^{cj}+R+1]$ in the above inequality by the bigger interval $[t, t+2e^{cj}]$.
	
	Assume that $u\in X^{0}_{t'}$ for some $t'\in [t, t+2e^{cj}]$. By the third inequality in \cite{cd09}, Lemma 2.3 and Markov's inequality, with probability at least $\frac{99}{100}$, there exists a time $t''\in [t', t'+e^{cj}]$ such that there are at least $\lambda j/4$ neighbors of  $u$ infected at time $t''$. By  \cite[Lemma 2.2]{cd09}, with probability at least $\frac{99}{100}$, there are at least $\lambda j/10$ neighbors of $u$ infected at any time in the time interval $[t'', t'' + 3e^{cj}]$. Since $t+2e^{cj}\in [t'', t'' + 3e^{cj}]$, the probability that $u$ is infected at time $t+2e^{cj}$ is at least the probability that the last clock rings before time $t+2e^{cj}$ is an infection clock rather than the recovery clock at $u$. Since there are at least $\lambda j/10$ neighbors of $u$ infected at any time in the interval $[t'', t+2e^{cj}]$, the probability of the above event is
	$$\frac{\lambda^{2}j/10}{\lambda^{2}j/10 +1 } \ge \frac{99}{100}.$$
	That completes the proof of Lemma \ref{lm:pathinfection:subex}.
\end{proof}

\begin{remark}\label{rmk:onevertex:subex} To prove the corresponding result (Remark \ref{rmk:onevertex}) for subexponential $\mu$ when initially, there is only one uniformly chosen vertex $v$ infected in $G_n$, observe that with positive probability over the choice of $v$, $v$ belongs to $W_0$. Thus, it suffices to condition on this event and show that with positive probability over the contact process, the process survives until time $e^{\Omega(n)}$. 
	
	Since $W_0$ is an $(\alpha, R)$-embedded expander, observe that for every $A\subset W_0$ with $|A|\le \alpha |W_0|/4$, we have
	$$|N(A, 2R)|\ge 4|A|.$$
	
	Thus, using the same proof as for Lemmas and \ref{lm:pathinfection:subex} and \ref{lm:azuma:2}, one can see that there exist constants $c'$ and $C''$ such that for all integer
	$a\in (0,  \alpha \beta n/4]$,
	\begin{equation}\label{eq:path:expectation2:2}
	\P\left (\left |X^{0}_{t+e^{c' j}}\right |\le 2a\bigg| |X^{0}_{t}|= a\right )\le 2\exp\left (-\frac{a}{C''}\right ).
	\end{equation}

	Let $k$ be the largest number such that $2^{k}\le \alpha \beta n/4$.  Let $\mathcal A_0 $ be the event that $X^{0}_{0} = \{v\}\subset W_0$. For each $i=1, \dots, k$, let $\mathcal A_i$ be the event that 
	$$\left |X^{0}_{ie^{c'j}}\right |\ge 2^{i}.$$ 
	Let $\mathcal A^*$ be the event that the contact process survives up to time $e^{\Omega(n)}$. We want to show that
	\begin{equation}\label{eq:onevertex:subex1}
	\P\left (\mathcal A^*\bigg|X^{0}_{0} = \{v\}\right ) = \Omega(1).
	\end{equation}
	In fact, 
	\begin{equation*}
		\P\left (\mathcal A^*\bigg|X^{0}_{0} = \{v\}\right ) \ge \P\left (\bigcap_{i=1}^{k}\mathcal A_i \cap \mathcal A^*\bigg|X^{0}_{0} = \{v\}\right )
		\ge  \prod_{i=1}^{k}\P\left (\mathcal A_i  \bigg|\mathcal A_{i-1}\right) \P\left (\mathcal A^*\bigg|\mathcal A_k\right )\nonumber.
	\end{equation*}
	By \eqref{eq:path:expectation2:2}, for each $i=1, \dots, k$, $\P\left (\mathcal A_i  \bigg|\mathcal A_{i-1}\right) \ge 1 - 2\exp\left (-\frac{2^{i-1}}{ C''}\right )$. Finally, by the same argument as in the proof of Theorem \ref{thm: exp tail main}-(2), once there are about $\Theta(n)$ vertices in $W_0$ infected, the contact process survives for an exponentially long time \textsf{whp}. In other words, 
	$\P\left (\mathcal A^*\bigg|\mathcal A_k\right ) = 1-o(1)$. 
	Let $k'$ be the smallest number such that $2^{k'-1}\ge 10C''$. Combining all of these inequalities, we obtain
	\begin{eqnarray}\label{}
	\P\left (\mathcal A^*\bigg|X^{0}_{0} = \{v\}\right ) &\ge& \frac{98}{100}\prod_{i=2}^{k'}\left (1 - 2\exp\left (-\frac{2^{i-1}}{ C''}\right )\right ) \times \prod_{i=k'+1}^{k}\left (1 - 2\exp\left (-\frac{2^{i-1}}{ C''}\right )\right )\nonumber\\
	&\ge& \frac{98}{100}\left (1 - 4e^{-10}\right )\prod_{i=2}^{k'}\left (1 - 2\exp\left (-\frac{2^{i-1}}{ C''}\right )\right )=\Omega(1)\nonumber
	\end{eqnarray}
	as desired. That completes the proof of Remark \ref{rmk:onevertex}.
\end{remark}

\section{Proof of the structural lemma}\label{sec:structure}

In this section, we prove the structural Lemmas \ref{lm:structure} and \ref{lm:structure:subex}. We start by proving Lemma \ref{lm:structure:subex} for subexponential $\mu$ in Section \ref{sec:structural:subex}. The proof of Lemma \ref{lm:structure} is very similar and is presented in Section \ref{sec:structural:ex}.
\subsection{Proof of Lemma \ref{lm:structure:subex}} \label{sec:structural:subex} \quad 

\textit{Step 1. Preprocessing.} ~In this step, we eliminate high-degree vertices in $G_n$ so that the degrees become bounded.  This will allow us to control the size of the neighborhoods that we explore in the next steps. We prove in Lemma \ref{lm:degree:reduction} that the elimination does not significantly affect relevant parameters of $G_n$.

Let $b = \frac{\sum_{l=1}^{\infty}l(l-1)\mu(l)}{\sum_{l=1}^{\infty} l\mu(l)}>1$ be the branching rate of $\mu$.  Let $d$ be the mean $d = \E_{D\sim \mu} D>1$.
For a constant $j$, consider the graph $\bar G_n$ obtained from $G_n$ by deleting all vertices with degree at least $2j+1$ together with their half-edges and their matches. Let $\bar n$ be the number of vertices of $\bar G_n$ and $0\le d_1\le \dots\le d_{\bar n}\le 2j$ be the degree sequence of vertices in $\bar G_n$.

The branching rate $\bar b$ of a (deterministic) degree sequence $(d_i)$ is defined to be the branching rate of the empirical measure generated by $(d_i)$, namely,
$$\bar b = \frac{\sum_{l} l(l-1)\#\{i: d_i = l\}}{\sum_{l} l\#\{i: d_i = l\}}.$$
Throughout the proof, $\ep$ can be any small constant (for example, $\ep = 1/2$).
\begin{lemma}[Eliminating high-degree vertices]\label{lm:degree:reduction} Let $j_0, \ep$ be any positive constants with $\ep<1$. There exists a positive constant $j \ge j_0$ such that the following hold \textsf{whp}.
	\begin{enumerate}
		\item\label{item:config} Conditioned on the degree sequence $(d_1, \dots, d_{\bar n})$, the edges of $\bar G_n$ form a uniformly chosen perfect matching of its half-edges.
		%		\item \label{item:muj} $\mu[j, 2j]\ge e^{-\ep j}$.
		\item \label{item:sumdi}The number of vertices and the total degree in $\bar G_n$ (which is twice the number of edges of $\bar G_{n}$) satisfy
		\begin{equation*} 
		\bar n\ge (1-\ep)n \quad \text{and}\quad 
		d_1+\dots +d_{\bar n} \in \left (1- \ep, 1+\ep\right ) nd.\nonumber
		\end{equation*}
		\item \label{item:branchingrate} The branching rate $\bar b$ of the degree sequence of $\bar G_n$ satisfies $\bar b\in (1-\ep, 1+\ep)b$.
		\item \label{item:dj} For all $i$, $0\le d_i\le 2j$. The number of vertices with large degree is as expected
		\begin{equation}\label{key}
		\#\left \{i\in \{1, \dots, \bar n\}: d_i \in \left [\frac{j}{2}, 2j\right ]\right \} \ge \ep n \mu[j, 2j].\nonumber
		\end{equation}
	\end{enumerate}

\end{lemma}
To simplify the notation, for the rest of this section \ref{sec:structure}, we define
$$\mathfrak u_j := \mu [j, 2j].$$
\begin{proof} 
	Since the proof of Items \ref{item:config}-\ref{item:branchingrate} is rather standard, we defer it to the Appendix, Section \ref{app:preprocessing}. Here, we only prove Item \ref{item:dj}. 
	Choose $j$ large enough such that 
	$$ \E_{D\sim \mu} D\one_{D\ge 2j+1}\le \ep/4.$$
	For each vertex $v\in G_n$, consider the random variable 
	$$X_v := \deg_{G_n}(v)\one_{\deg_{G_n}(v)\ge 2j+1}.$$ 
	These random variables are independent with mean at most $\ep/4$ and variance bounded by the second moment of $\mu$. By Chebyshev's inequality, \textsf{whp}
	$$\sum_{v\in G_n} X_v\le \ep n/2.$$
	Thus, \textsf{whp}, the total number of removed half-edges  from vertices of degree in $[0, 2j]$ in $G_n$ is at most $\sum_{v\in G_n} X_v\le \ep n/2$.

	By Chernoff inequality, \textsf{whp}, the number of half-edges of $G_n$ of vertices of degree in $[0, j)$ and $[j, 2j]$ are $(1-\ep, 1+\ep)nd$ and $(1-\ep, 2+2\ep)jn\mathfrak u_j $, respectively. Since $j$ is sufficiently large, $j\mathfrak u_j$ is very small compared to $d$. The first deleted half-edge has probability roughly $\frac{j\mathfrak u_j }{d}$ to be from vertices of degree in $[j, 2j]$. Ideally, one expects to delete at most $\frac{j\mathfrak u_j }{d}\ep nd $ half-edges from these vertices. We will show that it is the case, namely,
	\begin{claim}\label{claim:edge:j}
		The number of half-edges deleted from vertices of degree in $[j, 2j]$ is at most
		$$10\ep jn\mathfrak u_j\quad\text{\textsf{whp}}.$$
	\end{claim}
	Assuming the claim, the number of vertices originally with degree in $G_n$ in $[j, 2j]$ and with degree less than $j/2$ in $\bar G_n$ is at most
	$$\frac{10\ep jn\mathfrak u_j }{j/2} = 20\ep n\mathfrak u_j\quad\text{\textsf{whp}}.$$
	By Chernoff inequality, in $G_n$, the number of vertices of degree in $[j/2, 2j]$ is in $(1-\ep, 2+2\ep) n\mathfrak u_j$ (where we choose $j$ so that $\mu[j/2, j)\le (1+\ep)\mathfrak u_j$). Hence, \textsf{whp}, the number of vertices in $\bar G_n$ with degree in $[j/2, 2j]$ is in $(1-21\ep, 2+2\ep) n\mathfrak u_j$, completing the proof of Item \ref{item:dj}.
\end{proof}

To prove Claim \ref{claim:edge:j}, we will use the following cut-off line algorithm to find the random matches of the deleted high-degree vertices in $G_n$.
\begin{definition}[Cut-off line algorithm]\label{def:cut-off line}
	Given a graph $G_n$ in which each vertex $v$ has degree $d_{G_n}(v)$. A perfect matching of the half-edges of $G_n$ is obtained through the following algorithm.
	\begin{itemize}
		\item Each half-edge of a vertex $v$ is assigned a height uniformly chosen in $[0, 1]$ and is placed on the line of vertex $v$.
		\item Set the cut-off line at height 1.
		\item Pick an unmatched half-edge independent of the heights of all unmatched half-edges and match it to the highest unmatched half-edge. Move the cut-off line to the height of the latter half-edge.
	\end{itemize}
\end{definition}
Figure \ref{fig2} illustrates the algorithm.
\begin{figure}%[H]%H: precisely the location in the LaTeX code
	\centering
	\begin{tikzpicture}[thick,scale=1.1, every node/.style={transform shape}]
	\foreach \x in {0,1, 2, 4, 6, 8}{
		\draw[black] (\x,0) -- (\x,4.8);
	}
	
	\draw[black] (0,0)--(8,0);%draw a line connecting two points
	\node at (0,-0.4) {$v_1$};
	\node at (1,-0.4) {$v_2$};
	\node at (2,-0.4) {$v_3$};
	\node at (4,-0.4) {$\dots$};
	\node at (6,-0.4) {$\dots$};
	\node at (8,-0.4) {$v_n$};
	
	\draw[red] (-0.3,3.2)--(8.3,3.2);
	\node at (9.5,3.2) {cut-off line}; %name a node
	
	\draw[red, dashed] (-0.3,2.8)--(8.3,2.8);
	\node at (9.5,2.8) {new cut-off line}; 
	\draw (0, 3.2) node[black]{o};
	\draw (2, 0.5) node[black]{o};
	\draw (4, 4) node[black]{o};
	\draw (6, 3.5) node[black]{o};
	\draw (8, 3.7) node[black]{o};
	\draw (8, 3) node[black]{o};
	
	\draw (0, 0.5) node[cross=2.2pt,black]{}; %create a node
	\draw (0,2.2) node[cross=2.2pt,black]{};
	\draw (1,2.4) node[cross=2.2pt,black]{};
	\draw (2,1.6) node[cross=2.2pt,black]{};
	\draw (4,1.36) node[cross=2.2pt,black]{};
	\draw (4,2.8) node[cross=2.2pt,red]{};
	
	\draw (6,0.3) node[cross=2.2pt,black]{};
	\draw (6,1.5) node[cross=2.2pt,blue]{};
	\draw (6,2.5) node[cross=2.2pt,black]{};
	\draw (8,.88) node[cross=2.2pt,black]{};
	\draw[->, blue] (5.8,1.6) -- (4.2,2.7);

	\end{tikzpicture}
	\caption{The circles `o' represent matched half-edges and the crosses `$\times$' represent unmatched half-edges. The blue half-edge is chosen and matched to the red half-edge which is the highest unmatched half-edge. Then the cut-off line is moved to the new cut-off line (dashed).} \label{fig2}
\end{figure}
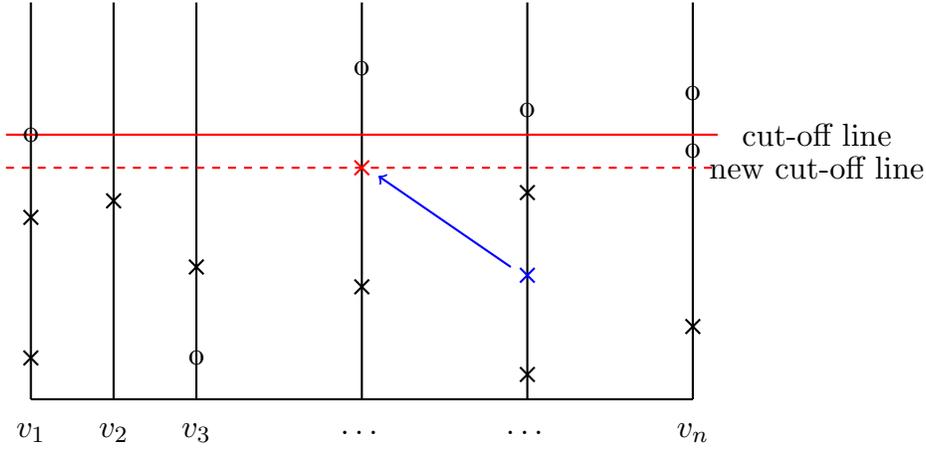

\begin{proof}[Proof of Claim \ref{claim:edge:j}]
	For each half-edge of a vertex in $G_n$ with degree at least $2j+1$, we choose their matches according to the above algorithm. Assume that at the end of this process of deleting half-edges of vertices with degree at least $2j+1$, the cut-off line is at height $h\in (0, 1)$. Note that a half-edge of a vertex whose degree in $G_n$ lies in $[j, 2j]$ is deleted if and only if it is above the cut-off line. We show that \textsf{whp}, $h\ge 1-4\ep$. Indeed, Item \ref{item:sumdi} implies that the number of half-edges of $G_n$ is at least $(1-\ep)nd/2$ and in choosing the heights of these half-edges, the number of half-edges with heights above $1-4\ep$ is at least $2\ep(1-2\ep)nd$ \textsf{whp}, by Chernoff inequality, which contradicts the event that we only delete at most $\ep nd$ half-edges altogether. Hence the cut-off line is above $1-4\ep$ \textsf{whp}. Since the number of half-edges with degree in $G_n$ belonging to $[j, 2j]$ is at most $(1+\ep) 2jn\mathfrak u_j$ \textsf{whp}, the number of half-edges above $1-4\ep$ is at most $8\ep (1+\ep)j n\mathfrak u_j\le 10 \ep j n \mathfrak u_j$ as claimed. That proves Claim \ref{claim:edge:j}.
\end{proof}

\textit{Step 2. Exploration.}  Let $W$ be the set of vertices in $\bar G_n$ whose degrees belong to $[j/2, 2j]$. We shall find the desired $(\alpha, R)$-embedded expander $W_0$ inside $W$. In this step, we explore the $R$-neighborhoods of these high-degree vertices in $W$.

After having preprocessed the graph $G_n$ to obtain $\bar G_n$, the remaining randomness is the perfect matching of the half-edges in $\bar G_n$. In this step, we condition on the preprocessing step and write the probability in terms of the randomness of the perfect matching in $\bar G_n$.
We run the following exploration process to perform some matchings of the half-edges of $\bar G_n$.  Let $R$ and $\mathfrak r$ be some large (bounded) numbers to be chosen (they are chosen in \eqref{eq:choose:R:mathfrak r}).
\begin{enumerate}
	\item For each vertex $v\in W$, set $r_v=1$.
	\item\label{item:exploration:2} If $r_v<R$ for all $v\in W$, explore the neighborhood $B_{\bar G_n}(v, r_v)$ simultaneously for all $v\in W$; noting that if $u\in B_{\bar G_n}(v, r_v) \cap B_{\bar G_n}(v', r_{v'})$, we stop exploring the branch starting at $u$. If $r_v\ge R$ for some $v\in W$, the process terminates. Otherwise, go to (3).
	\item For each vertex $v\in W$, if $N(v, r_v)$ intersects at most $100 \mathfrak r$ other balls $B_{\bar G_n}(v', r_{v'})$ ($v'\in W$), set $r_v := r_v+1$. Otherwise, keep $r_v$ intact. If none of the $r_v$ is increases in this step, the process terminates. Otherwise, go back to (2).
\end{enumerate} 

We show that when the exploration process terminates, the number of vertices $v\in W$ at which the process stops before reaching radius $R$ is insignificant. For that, we choose $R$ and $\mathfrak r$ so that the expected number of vertices of $W$ that lie in a neighborhood $N(v, 2R)$ is small compared to $\mathfrak r$ (see \eqref{eq:R:mathfrak r:1}). And so, it is unlikely that the different neighborhoods $N(v, R)$ intersect frequently.
\begin{lemma}\label{lm:rv:R} Let $R$ and $\mathfrak r$ be positive numbers bounded by some constants and satisfying
	%	\begin{equation}\label{eq:R:mathfrak r:1}
	%		(1+3\ep)^{2R-1}b^{2R-1} \frac{j^{2}\mathfrak u_j}{d} <\frac{\mathfrak r}{4}.
	%	\end{equation}
	\begin{equation}\label{eq:R:mathfrak r:1}
	\frac{\bar b^{2R-1} j^{2}\mathfrak u_j}{d} \le \frac{\mathfrak r}{10}.
	\end{equation}
	The number of $v\in W$ with $r_v = R$ is at least $0.98 |W|$ \textsf{whp}.
\end{lemma}
Note that $\bar b$ is a random variable and so are $R$ and $\mathfrak r$. Nevertheless, we will choose $R$ and $\mathfrak r$ so that they are bounded.
\begin{proof} For this proof, we write $N(v, r_v)$ for $B_{\bar G_n}(v, r_v)$ for simplicity.
	We first show that for each $v\in W$, it is likely that $r_v = R$; more specifically,
	\begin{equation}\label{eq:rv<R}
	\P\left (r_v=R\right )\ge 0.99.
	\end{equation}
	Indeed, if $r_v<R$ then $N(v, R)$ intersects more than $100\mathfrak r$ other balls $N(u, R)$ ($u\in W$), which implies that $N(v, 2R)$ contains more than $100 \mathfrak r$ elements of $W$. By Markov's inequality,
	$$\P\left (r_v<R\right )\le \frac{\E \left |N(v, 2R)\cap W\right |}{100 \mathfrak r}.$$
	%	By Lemma \ref{lm:degree:reduction}, the branching rate $\bar b$ of the degree sequence $d_1, \dots, d_{\bar n}$ of $\bar G_n$ belongs to the interval $(1-\ep, 1+\ep) b$.
	By Items \ref{item:sumdi} and \ref{item:dj} of Lemma \ref{lm:degree:reduction}, we have
	$$\E \left |N(v, 2R)\cap W\right | \le 2j \bar b^{2R-1} \frac{2j (2+\ep) n \mathfrak u_j}{(1-\ep) nd}\le \mathfrak r$$
	where we used \eqref{eq:R:mathfrak r:1}. Thus, we get \eqref{eq:rv<R}.
	
	It remains to show that \textsf{whp}, at least $0.98|W|$ vertices $v$ satisfy $r_v = R$. We derive this from \eqref{eq:rv<R} and Azuma's inequality. Let $X$ be the number of vertices $v\in W$ with $r_v = R$. By \eqref{eq:rv<R} and Item \ref{item:dj} of Lemma \ref{lm:degree:reduction},
	$$\E X\ge 0.99 |W|\ge 0.99 (1-\ep) n\mathfrak u_j.$$

	Enumerate the vertices of $\bar G_n$ by $u_1, \dots, u_{\bar n}$. Let $\mathcal F_i$ be the $\sigma$-algebra generated by the matchings of the half-edges of vertices $u_1, \dots, u_i$. We will apply the Azuma's inequality to the martingale $\E\left (X-\E X\big| \mathcal F_i \right ), i=0, \dots, \bar n$. Since the maximal degree in $\bar G_n$ is $2j$, we have that for every $i$,
	\begin{equation}\label{key}
	\left |\E\left (X-\E X\big| \mathcal F_{i+1} \right ) - \E\left (X-\E X\big| \mathcal F_{i} \right )\right |\le 4j \max_{u\in \bar G_n} |N(u, 2R)|\le 4j (2j)^{2R}
	\end{equation}
	where in the first inequality, we observed that for any fixed matching of the half-edges of vertices $u_1, \dots, u_i$ and any two different matchings of the half-edges of $u_{i+1}$, there exists a bijection between the extensions of these matchings into perfect matchings of $\bar G_n$ such that the number of different matchings are at most $4j$. Using Azuma's inequality and the fact that $j$ and $R$ are constants, we obtain
	\begin{equation}\label{key}
	\P\left (\left |X-\E X\right |\ge \ep \E X \right )\le \exp\left (-\frac{\ep^{2}(\E X)^2}{2 \bar n(4j (2j)^{2R})^{2}}\right ) = \exp\left (-\Omega(n)\right ).
	\end{equation}
	Thus, \textsf{whp}, $X\ge (1-\ep) \E X \ge 0.98 |W|$ as stated.
\end{proof}

\textit{Step 3. Finding $W_0$.}  We now find the desired embedded expander $W_0$. Our strategy is roughly as follows. In step 2, we have explored the $R$-neighborhoods of the high-degree vertices in $W$. We shall show (in Lemma \ref{lm:out:halfedges}) that most of these neighborhoods have a lot of unmatched half-edges. If we think about a new graph in which each of these neighborhoods acts as a single vertex with high degree, then a high-degree core of this new graph corresponds to the desired $W_0$.

Consider a new graph $G'_n$ with vertex set $V' \cup V''$ where $V''$ are the vertices of $\bar G_n$ that have not been touched in the exploration step 2 and each element of $V'$ is a ball $N(v, r_v)$ in step 2 whose half-edges are the unmatched half-edges of $N(v, r_v)$. If there is an unmatched half-edge that belongs to at least 2 balls, we choose one such ball at random and associate this half-edge to that ball. The remaining randomness is the uniform perfect matching of the half-edges in $G'_n$. We show that many vertices in $V'$ have high degree. Note that $|V'| = |W|$.
\begin{lemma}[$V'$ has high degree]\label{lm:out:halfedges} There exist positive constants $\ep', \ep''$ and $R_0$ depending only on $\mu$ such that for all bounded positive numbers $R_1, R, \mathfrak r$ satisfying
	\begin{equation}\label{eq:out:halfedges:cond}
	\begin{split}
	&R_0\le \min\{R_1, R-R_1\},\quad 800 \mathfrak r\le \ep'^{2} (\bar b (1-\ep''))^{R_1-1} j, \\
	&\qquad \frac{\bar b^{2R_{1}-1} j^{2}\mathfrak u_j}{d} \le \frac{1}{10^{4}}, \quad \frac{\bar b^{2R-1} j^{2}\mathfrak u_j}{d} \le \frac{\mathfrak r}{10},
	\end{split}
	\end{equation}
	the number of vertices in $V'$ with degree at least $M$ is at least $\frac{\ep'}{2}|V'|$ \textsf{whp} where
	%	\begin{equation}\label{def:M}
	%	M = \frac{\ep'^{2} (1-\ep)^{R-1}b^{R-1} j}{8}.
	%	\end{equation}
	\begin{equation}\label{def:M}
	M = \frac{\ep'^{3} (\bar b (1-\ep''))^{R-1} j}{8}.
	\end{equation}
\end{lemma}
We note that in \eqref{eq:out:halfedges:cond}, the last two inequalities are for Lemma \ref{lm:rv:R} to hold. The constant $R_1$ is mainly for technical reasons. The condition that $800 \mathfrak r\le \ep' (\bar b (1-\ep''))^{R_1-1} j$ is there so that when we ignore at most $100\mathfrak r$ possible common branches, the number of remaining branches is still significant.
\begin{proof}
	First, we will show that
	\begin{claim}\label{claim:out:halfedges}
		For every $v\in W$, with probability at least $\ep'$, the number of half-edges on the boundary of $N(v, R)$ that do not belong to any of the balls $N(v', R), v'\in W\setminus\{v\}$ is at least $M$.
	\end{claim}
	%	Indeed, let $l=\deg_{\bar G_n} v\in [j/2, 2j]$ and $u_1, \dots, u_l$ be the vertices of distance $R_0$ from $v$ (\textsf{whp}, these $u_1, \dots, u_l$ are distinct). We have showed in Lemma \ref{lm:rv:R} that \textsf{whp}, $N(v, R)$ intersects at most $100 \mathfrak r$ other balls $N(v', R), v'\in W$. Consider the branches $B(u_h, R-1)$ of $N(v, R)$. Conditioned on $u_1, \dots, u_l$, it suffices to show that \textsf{whp}, for any choices of sets $A\subset [1, \dots, l]$ with $|A|\le 100 \mathfrak r$,
	Indeed, let $u_1, \dots, u_l$ be the vertices of distance $R_1$ from $v$. Since $\frac{\bar b^{2R_{1}-1} j^{2}\mathfrak u_j}{d} <\frac{1}{10^{4}}$, by Lemma \ref{lm:rv:R},  with probability at least $0.99$, $N(v, R_1)$ does not intersect any other balls $N(v', R_1)$.
	
	Since $\frac{\bar b^{2R-1} j^{2}\mathfrak u_j}{d} <\frac{\mathfrak r}{10}$, by Lemma \ref{lm:rv:R}, with probability at least $0.99$, $N(v, R)$ intersects at most $100 \mathfrak r$ other balls $N(v', R), v'\in W$. Consider the branches $B(u_h, R-R_1)$ consisting of vertices at distance $r$ from $v$ and distance $r-R_1$ from $u_h$ for $R_1\le r\le R$. Conditioned on $u_1, \dots, u_l$, it suffices to show that with probability at least $0.9$, for any choices of sets $A\subset [1, \dots, l]$ with $|A|\le 100 \mathfrak r$,
	\begin{equation}\label{eq:boundary:1}
	\left |\bigcup_{h\in [1, \dots, l]\setminus A} \partial B(u_h, R-R_1)\right |\ge M.
	\end{equation}
	Letting $Z_{R'}$ be the number of children in the $R'$-th generation of the corresponding size-biased Galton-Watson process, we have for any $\ep''\in (0, 1)$, $\frac{Z_{R'}}{(\bar b(1-\ep''))^{R'}}$ converges almost surely to some random variable $Z$ (with $Z$ not identically 0 and taking values in $[0, \infty]$) as $R'\to \infty$ (see for example, \cite{AthreyaNey1972}, pages 24--29). For some sufficiently small constant $\ep'$ (that only depends on $\mu$), we have
	$$\P(Z\ge 4\ep')\ge 4\ep'.$$
	Thus, for a sufficiently large $R_0$ and $R_0\le R_1, R_0\le R-R_1$, we have
	\begin{equation}\label{key}
	\P\left (l \ge \ep'  (\bar b(1-\ep''))^{R_1-1}j\right )\ge \ep'\nonumber
	\end{equation}
	and
	\begin{equation}\label{key}
	\P\left (|\partial B(u_h, R-R_1)|\ge \ep'  (\bar b(1-\ep''))^{R-R_1}\right )\ge \ep' .\nonumber
	\end{equation}
	We can choose $\ep''$ so small that $\bar b(1-\ep'')>1$. Under the event that $l\ge \ep'  (\bar b (1-\ep''))^{R_1-1}j$, let $X_h$ ($h=1, \dots, l$) be the indicator of the event that $|\partial B(u_h, R-R_1)|\ge \ep'  (\bar b (1-\ep''))^{R-R_1} $. When $R_0$ is sufficiently large, $l$ is also large. Thus, with probability at least $0.9$, at least $\ep'l/4$ indices $h\in [1, \dots, l]$ have $X_h=1$. %This is because whp, it is the same as generating the $B(u_h, R-1)$ independently for each $h$ and conditioning on the event that some half-edges are matched to the same half-edge, which happens with probability $o(1)$.
	Under this event, since $100 \mathfrak r\le \ep' l/8$ by \eqref{eq:out:halfedges:cond}, for any choices of sets $A\subset [1, \dots, l]$ with $|A|\le 100 \mathfrak r$, there are at least $\ep' l/8$ indices $h\notin A$ with $X_h=1$, which implies
	\begin{equation*}
	\sum_{h\in [1, \dots, l]\setminus A} \left |\partial B(u_h, R-R_1)\right |\ge \ep'^{2} (\bar b (1-\ep''))^{R-R_1} l/8.\nonumber
	\end{equation*}
	Since $j, \bar b$ and $R$ are bounded by some constant, for each $v$, the boundaries $\partial B(u_h, R-R_1)$ are disjoint with probability at least $0.9$, proving \eqref{eq:boundary:1} and Claim \ref{claim:out:halfedges}.

	Next, by using the Azuma's inequality the same way that we used in proving Lemma \ref{lm:rv:R}, we obtain that \textsf{whp}, there are at least $(1-\ep)\ep'|W|$ vertices $v$ satisfying the event in Claim \ref{claim:out:halfedges} simultaneously.
	Combining this with Lemma \ref{lm:rv:R} completes the proof of Lemma \ref{lm:out:halfedges}.
\end{proof}

Note that the (random) edges of $G'_n$ form a uniformly chosen perfect matching of its half-edges. On the half-edges of $G_n'$, consider the following coloring scheme on the half-edges:
\begin{itemize}
	\item For each vertex $v$ with at least $M$ half-edges in $G_n'$, choose exactly $M$ half-edges among them uniformly at random and color them \textsf{blue};
	
	\item Perform the uniform random matching among all half-edges in $G_n'$, and let $W_1'$ be the induced subgraph on vertices with degree at least $M$;
	
	\item Let $\mathfrak K$ be the number of \textsf{blue} edges, formed by two \textsf{blue} half-edges, and for each $v\in W_1'$, let $\deg_b(v)$ be the number of \textsf{blue} edges adjacent to $v$.

\end{itemize}

%  In terms of the cut-off line algorithm, each half-edge in $G'_n$ chooses a height uniformly in $[0, 1]$. The highest $2\mathfrak K$ \textsf{blue} half-edges in $W_1'$ are matched with one another. Let $h_{\mathfrak K}$ be the smallest height of these half-edges. And then each of the remaining half-edges of $W_1'$ is matched with the highest unmatched half-edge outside of $W_1'$ one by one. Finally, each of the remaining half-edges outside of $W_1'$ is matched with the highest unmatched half-edge (outside of $W_1'$) one by one. 
\noindent Conditioned on $\mathfrak K$,  we see that the distribution of $\{\deg_b (v) \}_{v\in W_1'}$ is given by
\begin{equation}
\{B_v\}_{v\in W_1'}, \ \ \ B_v \sim \textnormal{i.i.d.}\ \Bin(M,\theta) \ \textnormal{conditioned on } \sum_{v\in W_1'} B_v = 2\mathfrak K. \label{eq:condition:k}
\end{equation}
%Then choose a uniform matching of the remaining half-edges in $W_1'$ with half-edges outside of $W_1'$. Finally, choose a uniform perfect matching of the remaining half-edges outside of $W_1'$. 
Note that due to the conditioning on the sum of $\{B_v\}$, their distribution is well-defined regardless of the specific value of $\theta$. However, for explicitness, we let 
$$\theta:= \frac{2\mathfrak K}{M|W_1'|}.$$

\begin{lemma}\label{lm:dominate:bin}
	Let $\ep'$ be as in Lemma \ref{lm:out:halfedges}. With high probability, 
	\begin{equation}\label{def:theta}
	\theta \ge \frac{\ep' \mathfrak u_jM }{60 d}.
	\end{equation}
\end{lemma}

\begin{proof}
	By Item \ref{item:dj} of Lemma \ref{lm:degree:reduction}, Lemmas \ref{lm:rv:R} and \ref{lm:out:halfedges}, the number of vertices in $W_1'$ is at least $\frac{\ep'|W|}{2}\ge \frac{\ep'n\mathfrak u_j}{2}$.
	
	Thus, the total number of \textsf{blue} half-edges in $W_1'$ is at least $ \frac{\ep'n\mathfrak u_jM}{2}$ \textsf{whp} while the total degree in $G'_n$ is at most $2nd$ by Item \ref{item:sumdi} of Lemma \ref{lm:degree:reduction}. Let 
	$$\theta_0 =\frac{\#\{\text{number of \textsf{blue} half-edges in } W_1'\}}{\#\{\text{number of half-edges in } G'_n\}}=:\frac{d_{W_1'}}{d_{G'_n}}\ge  \frac{\ep'\mathfrak u_jM }{4d} .$$
	
	%	$$\theta_0  \ge\frac{\ep'n\mathfrak u_jM }{4 nd} = 15\theta.$$

	Since $d_{W_1'}=M|W_1'|$, it suffices to show that \textsf{whp}, $\mathfrak K\ge \theta_0d_{W_1'}/12$. Indeed, splitting the set of \textsf{blue} half-edges of $W_1'$ into two parts of equal size $A$ and $B$ (independent of their heights). We perform the cut-off line algorithm \ref{def:cut-off line} to match the half-edges of $A$ first. Since at least $d_{W_1'}/4 = \theta_0 d_{G'_n}/4$ highest half-edges in $G'_n$ have been matched during this step, the cut-off line is below $1 - \frac{\theta_0}{5}$ \textsf{whp} (otherwise, by Chernoff inequality, the number of half-edges above the cut-off line is at most $\theta_0 d_{G'_n}/4$). By Chernoff inequality again, the number of half-edges of $B$ that lie above $1 - \theta_0/5$ is at least $\theta_0d_{W_1'}/12$ \textsf{whp}. Since all edges between $A$ and $B$ are inside $W_1'$, $\mathfrak K\ge \theta_0d_{W_1'}/12$ \textsf{whp}.
	%	 By Chernoff inequality, \textsf{whp}, $h_{\mathfrak K} \le 1-\theta_0/15$. Thus, \textsf{whp} all of the half-edges of $W_1'$ above the line $1-\theta_0/15$ belong to an internal edge of $W_1'$. Finally, for each vertex $v$ in $W_1'$, the number of its half-edges above the line $1-\theta_0/15$ has distribution dominating $\Bin(M, \theta_0/15)\ge \Bin(M, \theta)$, as claimed.
\end{proof}

Lemma \ref{lm:dominate:bin} allows to find a high-degree core of $W_1'$.
\begin{lemma}\label{lm:bin:core}
	Let $M$ and $\theta$ be as in Lemmas \ref{lm:out:halfedges} and \ref{lm:dominate:bin}. Assume that $\theta M\ge 100$. Then \textsf{whp}, $W_1'$ contains a subgraph $W_0'$ with the following properties.
	\begin{itemize}
		\item The number of vertices in $W_0'$ is at least $|W_1'|/2$.
		\item Each vertex in $W_0'$ has degree at least $\frac{\theta M}{20}$ inside $W_0'$; in other words, $W_0'$ is an $\frac{\theta M}{20}$-core.
		\item Each vertex in $W_0'$ has degree at most $(2j)^{R}$.%it's important for the next lemma that this is a constant.
	\end{itemize}
\end{lemma}
\begin{proof}
	Let $s=\frac{\theta M}{20}$. The last property  follows from the fact that the maximum degree of vertices in $V'$ is $(2j)^{R}$. For the rest of this proof, we only look at the \textsf{blue} half-edges in $W_1'$ that are matched to another \textsf{blue} half-edge in $W_1'$.
	To find $W_0'$, we use the cut-off line algorithm \ref{def:cut-off line} to find a uniform perfect matching of these half-edges of $W_1'$ as follows. Each of these half-edges is re-assigned a height uniformly chosen in $[0, 1]$.
	If there is a vertex in $W_1'$ with less than $s$ unmatched half-edges (equivalently, less than $s$ half-edges below the cut-off line), match its half-edges  to the highest unmatched half-edges and move the cut-off line accordingly. Remove this vertex. Repeat this step until there are no such vertices left.

	Let $W_0'$ be the set of remaining vertices. It remains to show that $|W_0'|\ge \frac{|W_1'|}{2}$ \textsf{whp}. Note that since $\mathfrak K = \Theta(n)$, the probability that $\sum_{v\in W_1'} B_v = 2\mathfrak K$ in \eqref{eq:condition:k} happens with probability $\Omega(n^{-C})$. In the rest of this proof, the tail probabilities are exponentially small in $n$ without conditioning on the event $\sum_{v\in W_1'} B_v = 2\mathfrak K$. And so, we can forget about the conditioning and assume that the number of internal half-edges of each vertex $v\in W_1'$ has degree distribution $\Bin(M, \theta)$.

	We show that after the removal, the cut-off line is above $2/3$ \textsf{whp}. Assuming this, we have
	$$|W_0'|\ge |W_1'|- N'$$
	where $N'$ is the number of vertices in $W_1'$ having less than $s$ half-edges below the line $2/3$. By Lemma \ref{lm:dominate:bin}, the number of internal half-edges of each vertex $v\in W_1'$ has degree distribution $\Bin(M, \theta)$. Thus, the distribution of the number of its half-edges that lie below the line $2/3$ is $\Bin\left (M, \frac{2\theta}{3}\right )$. So, the probability that $v$ has less than $s$ half-edges below the line $2/3$ is at most
	$$\P\left (\Bin\left (M,\frac{2\theta}{3}\right )\le \frac{\theta M}{20}\right )\le \exp\left (-\frac{\theta M}{12}\right )\le \frac{1}{200}.$$
	where we used the Chernoff inequality and the assumption $\theta M \ge 100$.
	By Chernoff inequality, we have \textsf{whp}, $N'\le \frac{|W_1'|}{100}$. And so, $|W_0'|\ge \frac{|W_1'|}{2}$ as desired.
	
	Now, we prove that after the removal process above, the cut-off line is above $2/3$ \textsf{whp}. Let $a$ be the number of removed vertices ($0\le a\le |W_1'|$). The total number of matched half-edges is at most $2as$ because each time we remove a vertex, at most $2s$ half-edges are matched. Thus, the total number of half-edges above the cut-off line is at most $2 as \le 2s|W_1'|$.
	
	On the other hand, given a vertex $v$ with degree distribution $\Bin(M, \theta)$, the distribution of the number of its half-edges that lie above the line $2/3$ is $\Bin\left (M, \frac{\theta}{3}\right )$. By Chernoff inequality, the number of half-edges of $W_1'$ above the line $2/3$ is at least $M\theta |W_1'|/4 = 5s|W_1'|>2s|W_1'|$.
	Thus, \textsf{whp}, the cut-off line is above $2/3$, completing the proof of Lemma \ref{lm:bin:core}.
\end{proof}

Next, we show that $W_0'$ is an embedded expander.
\begin{lemma}\label{lm:w0:expander}
	Let $M$ and $\theta$ be as in Lemmas \ref{lm:out:halfedges} and \ref{lm:dominate:bin}. Assume that $\theta M\ge 100$.	There exists a constant $\alpha>0$ such that \textsf{whp}, $W_0'$ is an $(\alpha, 1)$-embedded expander.
\end{lemma}
\begin{proof}
	Let $N = |W_0'|$, $s =\frac{\theta M}{20}$. Note that $N=\Theta(n)$ and $s\ge 5$ is a constant. It suffices to show that for any subset $A$ of the vertex set of $W_0'$ of size $\alpha N$, the size of $N(A, 1)$ is at least twice that of $A$. In other words, \textsf{whp}, for every $m\le \alpha N$ and subsets of vertices $A, B$ with $|A| = m, |B| = 2m$, the neighbors of $A$ are not contained fully in $B$. Fix two sets $A$ and $B$ with $|A| = m, |B| = 2m$. By Lemma \ref{lm:bin:core}, the number of half-edges in $A$ is at least $m s$, the number of half-edges in $B$ is $b\le (2j)^{R}2m$ and the total number of half-edges in $W_0'$ is $c\ge Ns$. The probability that all the neighbors of $A$ belong to $B$ is at most
	\begin{equation*}
	\frac{b}{c-1}\frac{b-1}{c-3}\dots \frac{b-ms+1}{c - 2ms+1}\le \left (\frac{b}{Ns- 2ms+1}\right )^{ms} \le \left (\frac{4(2j)^{R}m}{Ns}\right )^{ms}.\nonumber
	\end{equation*}
	Taking the union bound over $m$ and choices of $A, B$, we get that the probability that $W_0'$ is not an $(\alpha, 1)$-embedded expander is at most
	\begin{equation}\label{eq:union:expander}
	\sum_{m=1}^{\alpha N} {N\choose 2m}^{2}  \left (\frac{4(2j)^{R}m}{Ns}\right )^{ms} \le  \sum_{m=1}^{\log n}   \frac{C\log ^{s-4}n}{N^{s-4}} + \sum_{m=\log n}^{\alpha N}  \left (C\alpha^{s-4}\right )^{m}
	\end{equation}
	for some constant $C$ depending only on $j, R$ and $s$. Choosing $\alpha\le \frac{1}{2C}$ makes the r.h.s. of \eqref{eq:union:expander} of order $o(1)$. This completes the proof of Lemma \ref{lm:w0:expander}.
\end{proof}

\begin{proof}[Proof of Lemma \ref{lm:structure:subex}] 
	Since $W_0'$ is a subset of $V'$, each of its vertices is a ball $N(v, r_v)$ in $\bar G_n$, a subgraph of $G_n$. Let $W_0$ be the collection of all such centers $v$. Clearly, $|W_0| = |W_0'| = \Theta(n)$ and $W_0$ is an $(\alpha, 2R+1)$-embedded expander of $G_n$.
	To finish the proof of Lemma \ref{lm:structure:subex} for subexponential degree distributions, it remains to show that there exists a choice of $j, R, \mathfrak r$ satisfying the assumptions of the previous Lemmas (in particular, Lemmas \ref{lm:rv:R}, \ref{lm:out:halfedges}, \ref{lm:bin:core} and \ref{lm:w0:expander}). In other words, for any given positive constants $R_0, \delta$ and $\ep'$ ($\delta$ and $\ep'$ can be arbitrarily small and $R_0$ can be arbitrarily large), we show that there exist a constant $j$ and random variables $R, R_1, \mathfrak r$ such that the following conditions holds:
	\begin{equation}\label{eq:upperboundR}
	R\le \delta j,
	\end{equation}
	\begin{equation}\label{eq:R:large}
	R_0\le \min\{R_1, R-R_1\},
	\end{equation}
	\begin{equation}\label{eq:R:mathfrak r}
	\frac{ \bar b^{2R_1-1} j^{2}\mathfrak u_j}{d} \le \frac{1}{10^{4}}, \quad 800 \mathfrak r\le \ep'^{2}(\bar b (1-\ep''))^{R_1-1}j, \quad \frac{ \bar b^{2R-1} j^{2}\mathfrak u_j}{d} \le \frac{\mathfrak r}{10},
	\end{equation}
	\begin{equation}\label{eq:lm:bin:core}
	\frac{\ep'^{7} (\bar b (1-\ep''))^{2R-2} j^{2}\mathfrak u_j}{64\cdot 60 d} \ge 100.
	\end{equation}
	%\begin{equation}\label{eq:mathfrak r:j}
	%800 \mathfrak r\le \ep' j.
	%\end{equation}
	%and
	%\begin{equation}\label{eq:R:j:delta}
	%R\le \delta j.
	%\end{equation}
	Note that \eqref{eq:upperboundR} comes merely from the statement of Lemma \ref{lm:structure}.
	For given $j$ and $\bar b$, we define $\mathfrak r$, $R$ and $R_1$ by the following equations so that \eqref{eq:R:mathfrak r} holds automatically
	\begin{equation}\label{eq:choose:R:mathfrak r}
	\frac{ \bar b^{2R_1-1} j^{2}\mathfrak u_j}{d} =\frac{1}{10^{4}}, \quad 800 \mathfrak r= \ep'^{2} (\bar b (1-\ep''))^{R_1-1}j, \quad \frac{ \bar b^{2R-1} j^{2}\mathfrak u_j}{d} =\frac{\mathfrak r}{10}.
	\end{equation}
	Since $\mu$ has finite second moment, $\lim _{j\to \infty} j^{2}\mathfrak u_j = 0$ and so \eqref{eq:R:large} holds when $j$ is sufficiently large.
	Since $\mu$ is subexponential, $\mathfrak u_j\ge e^{-\ep_0 j}$ for any constant $\ep_0>0$ and for large $j$. By choosing $\ep_0$ to be small compared to $\delta$, \eqref{eq:upperboundR} holds. The inequality \eqref{eq:lm:bin:core} holds automatically. This completes the proof of Lemma \ref{lm:structure:subex} for subexponential distributions $\mu$.
\end{proof}
\begin{remark}\label{rmk:W0:preprocess}
	Observe that the above subgraph $W_0$ is also an $(\alpha, 2R+1)$-embedded expander of the graph $\bar G_n$ with $\deg_{\bar G_n} w\ge j/2$ for all $w\in W_0$.
\end{remark}

\subsection{Proof of Lemma \ref{lm:structure}} \label{sec:structural:ex}
In Section \ref{sec:structural:subex}, we used the assumption that $\mu$ is subexponential only in the last step of choosing the parameters $R, R_1$ and $\mathfrak r$. In particular, we used this assumption to obtain that $j$ can be an arbitrarily large constant while $\mathfrak u_j:=\mu[j, 2j]$ remains positive and $\mathfrak u_j \ge e^{-\ep_0 j}$. The inequality $\mathfrak u_j\ge e^{-\ep_0 j}$ is only used to show that with the choice of parameters $R, R_1, \mathfrak r$ as in \eqref{eq:choose:R:mathfrak r}, \eqref{eq:upperboundR} holds as stated in Lemma \ref{lm:structure:subex}. Here, when $\mu$ has an exponential tail, Lemma \ref{lm:structure} does not assert that $R\le \delta j$ and so there is no need for $\mathfrak u_j\ge e^{-\ep_0 j}$.

Thus, if $\mu$ has an infinite support, one can still find an arbitrarily large constant $j$ for which $\mathfrak u_j>0$. This is therefore enough for the rest of the proof of Section \ref{sec:structural:subex} to follow, proving Lemma \ref{lm:structure} for such $\mu$.

If the support of $\mu$ is finite, let $j$ be the largest integer in the support of $\mu$. Let $\mathfrak u_j$ be an arbitrarily small constant with $\mathfrak u_j\le \mu[j, 2j] = \mu(j)$ ($\mathfrak u_j$ could be much smaller than $\mu(j)$). Since the degrees in $G_n$ are already bounded, there is no need to run the Step 1 of preprocessing the graph as in Section \ref{sec:structural:subex}. So, for this case, $\bar G_n = G_n, \bar n = n, \bar b = b$ and so on. For the exploration, Step 2, we choose $W$ by assigning each vertex in $G_n$ of degree $j$ to $W$ independently with probability $\frac{\mathfrak u_j}{\mu(j)}$. The rest of the proof follows without any changes. In \eqref{eq:choose:R:mathfrak r}, we choose $\mathfrak u_j$ to be sufficiently small so that \eqref{eq:R:large} holds and so does \eqref{eq:lm:bin:core}. \qed

\section*{Acknowledgment}
We thank the anonymous referee for fruitful comments on improving the expositions of the manuscript.

\bibliographystyle{plain}
\bibliography{aapreference1}

\newpage

\section{Appendix}

\subsection{Proof of Lemma \ref{lem:aug}}\label{app:aug}

Here we prove Lemma \ref{lem:aug}, which is based on an elementary analysis of large deviation events.

\begin{proof}[Proof of Lemma \ref{lem:aug}]  The first statement follows directly from the Cauchy-Schwarz inequality: If we have $\E_{D\sim \mu} \exp(3\ep D) <\infty$ for some $\ep>0$, then
	\begin{equation*}
	\left( \sum_k e^{\ep k} \sqrt{p_k} \right)^2 \leq \left(\sum_k e^{3\ep k} p_k \right) \left(\sum_k e^{-\ep k} \right) <\infty.
	\end{equation*}
	
	For the second statement, let $n$ be a given large enough integer, and define
	$$k_n = \min \left\{k: \sum_{j\geq k} p_k \leq \frac{1}{n\log\log n} \right\} .$$
	
	Let $D_i$ for $i=1,\ldots n$ be i.i.d samples from $\mu$. We start by studying the empirical distribution of the $D_i$. First, by a simple union bound, the definition of $k_n$ implies that
	\begin{equation*}
	\P \left(\exists i\in [n] : \, D_i \geq k_n \right) \leq \frac{1}{\log\log n} =o(1).
	\end{equation*}
	Moreover, since $\mu$ has an exponential tail, there exists a constant $C>0$ depending on $\mu$ such that $k_n \leq C\log n$. Recall the definition of $k_0 = \max\{k: \sum_{j\geq k} \sqrt{p_k} \geq 1/2 \} $. Our next goal is to show that with high probability, the number of $i$ such that $D_i = k$ is at most $\frac{1}{2} n \sqrt{p_k}$   for all $k_0\leq k\leq k_n$. We consider two possible cases of $k$ as follows:
	\begin{enumerate}
		\item [1.] For $k$ such that $p_k \leq (n \log^2 n)^{-1}$, Markov's inequality implies that
		\begin{equation}\label{eq:augbd1}
		\P (|\{i: D_i=k \}|\geq 1 )\leq (\log^2 n)^{-1}.
		\end{equation}
		
		\item [2.] For $k$ such that $p_k \geq (n\log^2 n)^{-1}$, we use the following large deviation estimate for binomials (Corollary 22.9 of \cite{Frieze15}): for $c>1,$
		\begin{equation*}
		\P (\Bin(n,p)\geq cnp ) \leq \exp \{-np(c\log c + 1-c  )\}.
		\end{equation*}
		This gives that
		\begin{equation}\label{eq:augbd2}
		\begin{split}
		&\P\left(|\{i:D_i=k \} |  \geq \frac{1}{2} n\sqrt{p_k}  \right) \leq
		\exp \left(\frac{1}{2} n\sqrt{p_k} (1-2\sqrt{p_k}+\log(2\sqrt{p_k})) \right) \leq \exp \left(-n^{1/3}\right).
		\end{split}
		\end{equation}
	\end{enumerate}
	Since $k_n \leq C\log n$, applying a union bound on (\ref{eq:augbd1},  \ref{eq:augbd2}) tells us that
	\begin{equation}\label{eq:augineq}
	\P \left(\exists\; k_0\leq k \leq k_n: |\{i: D_i =k \}| \geq \frac{1}{2}n\sqrt{p_k}  \right) = o(1).
	\end{equation}
	When $\mu$ satisfies $k_0<k_{\max}$, (\ref{eq:augineq})  implies that the empirical distribution of $\{D_i : i\in [n]\}$ is stochastically dominated by $\mu^\sharp$, since $\mu^\sharp (k) \geq \sqrt{p_k}$ by the definition of $k_0$. On the other hand, if $k_0=k_{\max}$, the stochastic domination becomes trivial because we only augment the weight of $k_{\max}$ in $\mu$.
	Since taking out any $n/3$ entries from $[n]$ can only increase each probability mass of the empirical distribution of $\{D_i : i\in[n] \}$ by a factor of $3/2$, with high probability we have for each $k_0\leq k $,
	$$ |\{i: D_i = k \}| \leq \frac{3}{4}n\sqrt{p_k},$$
	and hence we conclude the second statement of Lemma \ref{lem:aug}.
\end{proof}

\subsection{Proof of Lemma \ref{lem:1cyc}}\label{app:treeexcess}

In this section, we prove Lemma \ref{lem:1cyc}. We use Lemma \ref{lem:aug} to bound the probability of $N(v,c\log n)$ having at least two cycles.

\begin{proof}[Proof of Lemma \ref{lem:1cyc}] Let $v$ be an arbitrary vertex in $G\sim \mathcal{G}(n,\mu)$, fixed before we explore the matchings of half-edges. We again study the local neighborhood $N(v,L)$ by exploration process, particularly in terms of the breadth-first search perspective. We start exploring from the single vertex $v$, and at time $s$ we explore all the vertices of distance $s$ from $v$, based on what we explored until time $s-1$. Let $V_s$ be the collection of vertices explored at time $s$, and set $X_s = |V_s|$.
	%Notice that there are two possibilities of obtaining a cycle in $N(v,L)$:
	%\begin{enumerate}
	%	\item [1.] At time $s+1$, two (or more) vertices in $V_s$ explore the same vertex.
	
	%	\item [2.] There are two half edges adjacent to two vertices (or a vertex) in $V_s$ that are matched together in the next exploration step.
	%\end{enumerate}
	We will bound the probability of discovering at least two cycles during the exploration process until depth $L=c \log n$ ($c$ will be determined later). Let $\mu^\sharp$ be the augmented distribution (Definition \ref{def:aug}) and let $\widetilde{\mu}^\sharp:= (\mu^\sharp)'_{[1,\infty )}$ denote its size-biased distribution conditioned on being inside the interval $[1,\infty)$. Also, let $\mathcal{T}\sim \gw(\mu^\sharp, \widetilde{\mu}^\sharp)_L $ and let $Y_s$ be the number of vertices in $\mathcal{T}$ at depth $s$. Then, Lemma \ref{lem:aug} implies that there exists a coupling between $(X_s)_{s\leq L}$ and $(Y_s)_{s\leq L}$ in such a way that $X_s \leq Y_s$ for all $s\leq L$, as long as $\sum_{s=0}^L X_s \leq n/3$. Define $B$ to be the event that $\sum_{s=0}^L X_s \leq n/3$. On $B$, we clearly have
	$$\sum_{v\notin \cup_{s\leq L} V_s} \deg(v) \geq \frac{2n}{3}.$$
	
	Assume that when moving from $V_s$ to $V_{s+1}$, we pair the half-edges adjacent to $V_s$ one by one. Let $H_s$ be the number of unpaired half-edges adjacent to the vertices in $V_s$. For $s\leq L$ and $1\leq i \leq H_s$, let $\mathcal{F}_{s,i}$ be the $\sigma$-algebra generated by the exploration process until pairing the $(i-1)$-th half-edge. Set $W_{s,i}$ to be the collection of unpaired half-edges at that moment.      Further, let $A_{s,i}$ be the event that the $i$-th half-edge adjacent to $V_s$ is paired with a half-edge in $W_{s,i}$. Then clearly,
	\begin{equation*}
	\P (A_{s,i} | \mathcal{F}_{s,i},B) \leq \frac{3|W_{s,i}|}{2n}.
	\end{equation*}
	
	\noindent We bound the size of $W_{s,i}$ based on the following observations:
	\begin{enumerate}
		\item [1.] Since the exploration of half-edges in $H_{s-1}$ for $s\geq 1$ is done independently step by step, we can stochastically dominate $|H_s|$ by i.i.d $\zeta_j \sim \widetilde{\mu}^\sharp$ as
		\begin{equation*}
		|H_s| \;\leq_{st}\; \sum_{j=1}^{Y_s} \zeta_j\; \overset{d}{=}\; Y_{s+1}.
		\end{equation*}
		
		\item [2.]  Since $\zeta \sim \widetilde{\mu}^\sharp$ satisfies $\zeta\geq 1$, we can bound $|W_{s,i}|$ similarly by
		\begin{equation*}
		|W_{s,i}| \;\leq_{st}\; \sum_{j=1}^{H_s} \zeta_j\; \leq_{st}\; \sum_{j=1}^{Y_{s+1}} \zeta_j\; \overset{d}{=} \;Y_{s+2}.
		\end{equation*}
	\end{enumerate}
	Therefore, combining above argument gives that
	\begin{equation*}
	\sum_{j=1}^{H_s} \one_{A_{s,j}} \; \leq_{st}\; \Bin\left({Y_{s+1}},{\frac{3Y_{s+2}}{2n}}\right),
	\end{equation*}
	where the l.h.s is conditioned on the event $B$. Notice that the l.h.s of the above inequality stochastically dominates the number of cycles formulated during the exploration of depth from $s$ to $s+1$. Hence the number of cycles in $N(v,L)$ conditioned on $B$ is stochastically dominated by
	\begin{equation}\label{eq:bd of cyc}
	\Bin \left(\sum_{s=1}^{L+1} Y_s , \frac{3 Y_{L+2}}{2n}  \right).
	\end{equation}

	Now we bound the size of $Y_s$ to conclude our argument. Let $\zeta \sim \mu^\sharp$ and $\zeta' \sim \widetilde{\mu}^\sharp$, and let $\ep, M$ be the constants that satisfy  $\max\{ \E e^{\ep \zeta}, \E e^{\ep \zeta'} \} \leq M $. Set $K$ to be a large constant such that $M^{1/K} =e^{\ep}$. Then we observe that $Y_s / K^s$ has an exponential tail for all $s$, since
	
	\begin{equation}\label{eq:exp tail for aug-gw}
	\begin{split}
	\E \left[\exp \left(\ep \frac{Y_s}{K^s} \right) \right]
	&=
	\E \left[\E_{\zeta' } \left[ \exp \left(\ep \frac{\zeta'}{K^s} \right)  \right]^{Y_{s-1}} \right]\leq
	\E \left[M^{Y_{s-1}/K^s} \right]= \E \left[\exp \left(\ep \frac{Y_{s-1}}{K^{s-1}} \right) \right],
	\end{split}
	\end{equation}
	where the inequality is due to Jensen's inequality. Iterating this $(s-1)$-times gives that the l.h.s is bounded by $e^\ep$. Set $c>0$ to be the constant satisfying $K^{c \log n + 3} = n^{1/6}.$
	Based on the above observation, we bound the quantity (\ref{eq:bd of cyc}) as follows.
	
	\begin{equation}\label{eq:bd the num of cyc}
	\begin{split}
	\P \left(\Bin \left(\sum_{s=1}^{L+1} Y_s , \frac{3 Y_{L+2}}{2n}  \right) \geq 2 \right)\; &\leq\;
	\P\left( \Bin \left( n^{1/5}, \frac{3}{2}n^{-4/5} \right) \geq 2 \right)+
	\P \left(\sum_{s=1}^{L+2} Y_s \geq n^{1/5} \right).
	\end{split}
	\end{equation}
	It is easy to see that the first term in the r.h.s is bounded by $o(n^{-1})$. The second term can be bounded using (\ref{eq:exp tail for aug-gw}). Namely,
	
	\begin{equation*}
	\begin{split}
	&\P \left(\sum_{s=1}^{L+2} Y_s \geq n^{1/5} \right)
	\leq
	e^{-\ep n^{1/30}} \E \left[\exp\left(\ep K^{-(L+3)} \sum_{s=1}^{L+2} Y_s\right) \right]\\
	&\qquad\leq
	e^{-\ep n^{1/30}} \E \left[\exp\left(\ep K^{-(L+3)} \sum_{s=1}^{L+1} Y_s\right)  \cdot \exp \left(\ep K^{-(L+2)} Y_{L+1}  \right) \right].
	\end{split}
	\end{equation*}
	Iterating this $(L+1)$ more times, we obtain
	\begin{equation*}
		\P \left(\sum_{s=1}^{L+2} Y_s \geq n^{1/5} \right)
		\leq
		e^{-\ep n^{1/30}} \E \left[\exp\left( \ep \left(\sum_{s=1}^{L+2} K^{-s}\right) \right) \right]
		 \leq \exp {(\ep (1- n^{1/30}))} = o(n^{-1}),
	\end{equation*}
	as long as $K\geq 2$. Applying our estimates to (\ref{eq:bd the num of cyc}), we conclude the desired result.
\end{proof}

\subsection{Proof of Proposition \ref{prop:cp on egw}}\label{app:cp on egw}

The proof follows the same technique as Lemma \ref{lem:cp on gwc}.

\begin{proof}[Proof of Proposition \ref{prop:cp on egw}] Let $s\geq 2$ and $L$ be any integers, and we build up an inductive argument starting from $l=0$.
	
	Let $\mathcal{S}_0 \sim \egw (\mu^\sharp, \widetilde{\mu}^\sharp;0,s)_L$, and $\rho^+ \in \mathcal{S}_0^+$ be the parent of $\rho$ as before. Define $S_{0,s,L}$ to be the first time when $(X_t)\sim\cp^\lambda_{\rho^+} (\mathcal{S}_0^+;\one_\rho)$ reaches  state $\zero$. Similarly as in Lemmas \ref{lem:cp on gw} and \ref{lem:cp on gwc}, we consider $(\widetilde{X}_t) \sim \tcp^\lambda_{\rho^+;\rho}(\mathcal{S}_0^+;\one_\rho)$, which is coupled with $(X_t)$ in such a way that they share the same infection and recovery clocks, except that in $(\widetilde{X}_t)$, the recovery at $\rho$ is ignored if at that time there exists an infected vertex other than $\rho$ and $\rho^+$. Letting $D\sim \mu^\sharp$ be $D+3=\deg(\rho;\mathcal{S}_0^+)$, $\mathcal{T}_{u_1},\ldots,\mathcal{T}_{u_D}$ be the i.i.d $\gw(\widetilde{\mu}^\sharp)_{L-1}$ subtrees from the children of $\rho$ and $\mathcal{S}'$ be the $\gwc(\widetilde{\mu}^\sharp;s)_L$ process that also hangs at $\rho$, we obtain the following by repeating the same argument in Lemma \ref{lem:cp on gwc}.
	\begin{equation}\label{eq:recursion for egw0}
	\E[S_{0,s,L}|D] \leq (1+\lambda \E[S_{L-1}])^D (1+2\lambda \E[S_{s,L}]),
	\end{equation}
	where $S_{L-1}$, $S_{s,L}$ are as in the statements of Lemmas \ref{lem:cp on gw} and \ref{lem:cp on gwc}, respectively.
	
	For general $l\neq 0$, we first develop the same argument in terms of $\mathcal{S}_l' \sim \egw(\widetilde{\mu}^\sharp;l,s)_L$. Let $S_{l,s,L}'$ be the first time when $\cp^\lambda_{\rho^+}((\mathcal{S}_l')^+;\one_\rho)$ reaches at $\zero$. For $D'\sim \widetilde{\mu}^\sharp$ denoting $D'+1 = \deg(\rho; (\mathcal{S}_{l}')^+)$, the subgraphs of descendents from the children of $\rho$ consist of $D'$ i.i.d  $\egw(\widetilde{\mu}^\sharp;l-1,s)_L$ processes. Therefore, repeating the previous reasoning gives that
	\begin{equation}\label{eq:recursion augegw}
	\E[S_{l,s,L}'|D'] \leq  (1+\lambda \E[S_{l-1,s,L}'])^{D'}.
	\end{equation}
	
	Finally, the subgraphs of descendents (from children of $\rho$) of $\mathcal{S}_l \sim \egw(\mu^\sharp,\widetilde{\mu}^\sharp;l,s)_L$  for $l\geq 1$ with $\deg(\rho)=D$ consist of $D$ i.i.d  $\egw (\widetilde{\mu}^\sharp; l-1,s)_L$ processes. Therefore, we deduce that $S_{l,s,L}$, the first time when $\cp^\lambda_{\rho^+}(\mathcal{S}_l;\one_\rho)$ reaches $\zero$, satisfies
	\begin{equation}\label{eq:recursion egw1}
	\E [S_{l,s,L}|D] \leq (1+\lambda \E[S_{l-1,s,L}'])^D.
	\end{equation}
	Here, the law of $D$ follows the conditional distribution of $\mu^\sharp$ being inside the interval $[1,\infty)$.
	
	Combining the three equations (\ref{eq:recursion for egw0}), (\ref{eq:recursion augegw}) and (\ref{eq:recursion egw1}), we obtain that there exists a constant $\lambda_0$ such that for all $\lambda\leq \lambda_0$, $l,s$ and $L$,
	\begin{equation*}
	\E[S_{l,s,L}]\leq 2e,
	\end{equation*}
	by manipulating the constants in the same way as Lemma \ref{lem:cp on gwc}. Then, the standard coupling between contact processes tells us that $R_{l,s,L} \leq_{st} S_{l,s,L}$, which concludes the proof. 
\end{proof}

\subsection{Proof of Lemma \ref{lem:dcp on egw}}\label{app:dcp on egw}

To establish Lemma \ref{lem:dcp on egw}, we first prove the result for GWC-processes. Let $\widetilde{\mu}^\sharp$ be as in Section \ref{sec:sub rg}. Namely, $\widetilde{\mu}^\sharp$ is the augmented distribution of $\mu'_{[1,\infty)}$, where $ \mu'_{[1,\infty)}$ is the size-biased distribution of $\mu$ conditioned on being in $[1,\infty)$.

\begin{lemma}\label{lem:dcp on gwc}
	Let $\mathcal{S} \sim \gwc( \widetilde{\mu}^\sharp;s)_L$, and $\nu^\theta_{\mathcal{S}}$ be the stationary distribution of $\ddp^{\lambda,\theta}_{\rho} (\mathcal {S})$ on the space $\{0,1\}^{\mathcal{S}\setminus\{\rho\} }$, which is the delayed contact process on $\mathcal{S}$ with $\rho$ set to be infected permanently. Then there exist constants $C, \lambda_0>0$ depending only on $\mu$ such that for all $\lambda\leq \lambda_0$ and $s, L$ with $s\geq 2$, we have $\E [\nu^\theta_{\mathcal{S}}(\zero)^{-1}] \leq 2$ for $\theta = C\lambda$.
\end{lemma}

\begin{proof}
	Let $v,v'$ be the two neighbors of $\rho$ in $\mathcal{S}$. Let $S_{v}^\theta$ denote the first time when $\ddp^{\lambda,\theta}_\rho (\mathcal{S};\one_v)$ reaches $\zero$, and define $S_{v'}^\theta$ analogously. Then, we set $S_L^\theta = \frac{1}{2} ( S_{v}^\theta  + S_{v'}^\theta)$.
	
	As before, we build up an inductive argument on $s$. The case $s=2$ is essentially the same as Proposition \ref{prop:dcp on gw}, since $\mathcal{S}\sim \gwc(\widetilde{\mu}^\sharp;2)_L$ can be thought of as $\mathcal{T}_L^+$ with $\mathcal{T}_L\sim \gw(\widetilde{\mu}^\sharp)_L$, where $\rho^+$ in $\mathcal{T}_L^+$ is connected with $\rho $ by a double-edge. Thus, the same proof  of Proposition \ref{prop:dcp on gw} can be applied, and we leave the details to the reader.
	
	The general case $s\geq 3$ is also similar to the previous arguments of Propositions \ref{prop:cp on egw} and \ref{prop:dcp on gw}, but there is a subtle difference in comparing the stationary distributions, which makes the current case more technical. As before, we start with introducing a modified process as follows.
	
	Let $v$ be a neighbor of $\rho$ in $\mathcal{S}$, and $(X_t) \sim \ddp^{\lambda,\theta}_\rho (\mathcal{S};\one_v)$. Define $(\widetilde{X}_t) \sim \widetilde{\ddp}^{\lambda,\theta}_{\rho;v}(\mathcal{S};\one_v)$ as
	\begin{enumerate}
		\item [1.] $(\widetilde{X}_t)$ has the same infection and recovery clocks at $(X_t)$.
		
		\item[2.]  In $(\widetilde{X}_t)$, any recovery attempt at $v$ is ignored if there exists an infected vertex other than $\rho$ and $v$ at that moment.
	\end{enumerate}
	
	If $\{\rho, v\}$ infects a (random) neighbor $U$ before $v$ is healed, then $(\widetilde{X}_t)$ behaves as $\ddp^{\lambda,\theta}_{\rho,v}(\mathcal{S};\one_U)$ (meaning that we fix both $\rho,v$ to be infected forever), until $\widetilde{X}_t$ comes back to $\one_v$. Let $\widetilde{S}_{s,L}^\theta$ denote the first time when $\widetilde{\ddp}^{\lambda,\theta}_{\rho;v}(\mathcal{S};\one_v)$ becomes $\zero$, and $\widetilde{S}^\theta$ be the first time it takes for $\ddp^{\lambda,\theta}_{\rho,v}(\mathcal{S};\zero)$ to return to $\zero$ after infecting a (random) vertex $U$ other than $\rho$ and $v$. Setting $D\sim \widetilde{\mu}^\sharp$ to be $\deg(v)=D+2$,  the same reasoning as (\ref{eq:geometric trial cp}, \ref{eq:geometric trial dcp}) implies that
	\begin{equation}\label{eq:recursion dcp on gwc}
	\begin{split}
	\E \left[\left.\widetilde{S}_{s,L}^\theta\,\right|\,\mathcal{S} \right]
	&=
	\sum_{k=0}^\infty \left(\frac{(D+2)\lambda}{1+(D+2)\lambda} \right)^k \frac{1}{1+(D+2)\lambda}\times \left[\frac{k+1}{\theta(1+(D+2)\lambda)}
	+ k\,\E\left[\left. \widetilde{S}^\theta \,\right|\,\mathcal{S} \right] \right]\\
	&=
	\frac{1}{\theta} \left(1+(D+2)\lambda \theta \E\left[\left. \widetilde{S}^\theta \,\right|\,\mathcal{S} \right] \right).
	\end{split}
	\end{equation}
	
	Now we take account of the stationary measures to compare the running times. We first set up some notations as follows.
	\begin{itemize}
		\item $\nu_{\mathcal{S}}'$ and $\pi_{\mathcal{S}}'$ are the stationary distribution of $\ddp^{\lambda,\theta}_{\rho, v} (\mathcal{S})$ and $\cp^\lambda_{\rho,v}(\mathcal{S}) $, respectively.
		
		\item $\mathcal{T}_{u_1} , \ldots , \mathcal{T}_{u_D}$ denote the subtrees from the children $u_1,\ldots, u_D$ of $v$ outside the cycle. Note that these subtrees are i.i.d $\gw(\widetilde{\mu}^\sharp)_{L-1}$.
		
		\item Set $\widetilde{\mathcal{S}} = \mathcal{S}\setminus \cup_{i=1}^D \mathcal{T}_{u_i} $. $\ddp^{\lambda,\theta}_{\rho,v}(\widetilde{\mathcal{S}}) $ denotes the delayed contact process on $\widetilde{\mathcal{S}}$ that fixes both $\rho, v$ to be infected permanently, which has the depth $r(x;\widetilde{\mathcal{S}})$  computed with respect to $\rho$. In particular, all possible states of $\ddp^{\lambda,\theta}_{\rho,v} (\widetilde{\mathcal{S}})$ have depth $r(x)$ at least one, since $v$ is always infected.
		
		\item $\nu_{\mathcal{T}_{u_i}}$ and $\nu_{\widetilde{\mathcal{S}}}$ are the stationary distributions of $\ddp^{\lambda,\theta}_v (\mathcal{T}_{u_i}^+)$ and $\ddp^{\lambda,\theta}_{\rho,v}(\widetilde{\mathcal{S}})$, respectively. Moreover, $\pi_{\mathcal{T}_{u_i}}$ and $\pi_{\widetilde{\mathcal{S}}}$ denote the stationary distributions of $\cp^{\lambda}_v(\mathcal{T}_{u_i}^+)$ and $\cp^\lambda_{\rho,v}(\widetilde{\mathcal{S}})$, respectively. Also, set $$\nu_{\mathcal{S}}^\otimes =\left(\otimes_{i=1}^D \nu_{\mathcal{T}_{u_i}} \right) \otimes \nu_{\widetilde{\mathcal{S}}}.$$
	\end{itemize}
	Note that $\pi_{\mathcal{S}}' = \left(\otimes_{i=1}^D \pi_{\mathcal{T}_{u_i}} \right) \otimes \pi_{\widetilde{\mathcal{S}}}$. Keeping in mind that $r(\zero; \mathcal{S})=1$ in $\ddp^{\lambda,\theta}_{\rho,v}(\mathcal{S})$, we obtain by using (\ref{eq:stationary of delayed gw}) that
	\begin{equation}\label{eq:prod dcp vs total dcp}
	\nu_{\mathcal{S}}' (\zero) \geq \nu_{\mathcal{S}}^\otimes (\zero).
	\end{equation}
	Moreover, observe that if we merge $\rho$ and $v$ in $\widetilde{\mathcal{S}}$ into a single vertex $\rho'$, then the resulting graph $\widetilde{\mathcal{S}}'$ satisfies $\widetilde{\mathcal{S}}' \sim \gwc(\widetilde{\mu}^\sharp;s-1)_L$, and we can consider the natural one-to-one correspondence between the two state spaces $\{0,1\}^{\widetilde{\mathcal{S}}\setminus\{\rho,v \}}$ and $\{0,1\}^{\widetilde{\mathcal{S}}'\setminus\{\rho'\}}$. Thus, we can regard them as 
	$$\Omega =\{0,1\}^{\widetilde{\mathcal{S}}\setminus\{\rho,v \}}= \{0,1\}^{\widetilde{\mathcal{S}}'\setminus\{\rho'\}}.$$ 
	For any $x\in \Omega \setminus \{\zero \}$, note that
	\begin{equation*}
	r(x;\widetilde{\mathcal{S}}) \in \left\{r(x;\widetilde{\mathcal{S}}'),\; r(x;\widetilde{\mathcal{S}}')+1 \right\}.
	\end{equation*}
	In particular, $r(x;\widetilde{\mathcal{S}}) -1\leq r(x;\widetilde{\mathcal{S}}')$. Further, we have $r(\zero;\widetilde{\mathcal{S}})=1$ and $r(\zero;\widetilde{\mathcal{S}}')=0$. This implies that if $\nu_{\widetilde{\mathcal{S}}'}$ denotes the stationary distribution of $\ddp^{\lambda,\theta}_{\rho'}(\widetilde{\mathcal{S}}')$, then
	$$\nu_{\widetilde{\mathcal{S}}'}(\zero) \leq \nu_{\widetilde{\mathcal{S}}}(\zero) .$$
	
	Therefore, combining with (\ref{eq:prod dcp vs total dcp}), we have
	\begin{equation}\label{eq:prod dcp on gwc vs total dcp}
	\nu_{{\mathcal{S}}}'(\zero) \geq \left(\prod_{i=1}^D \nu_{{\mathcal{T}_{u_i}}}(\zero) \right) \cdot \nu_{\widetilde{\mathcal{S}}'} (\zero).
	\end{equation}
	We can relate the quantities in (\ref{eq:prod dcp on gwc vs total dcp}) with the running times of the delayed processes. Let $S_{L-1}^\theta$  be the first time when $\ddp^{\lambda,\theta}_{\rho}(\mathcal{T}_{u_i}; \one_{u_i}) $ returns to $\zero$. Similarly, let $v_1, v_2\notin \{\rho,v \}$ be the two neighbors of $\{\rho,v\}$ in $\mathcal{S}'$, let $S_{v_i}^\theta$ be the first time when $\ddp^{\lambda,\theta}_{\rho,v}(\mathcal{S}';\one_{v_i})$ reaches $\zero$, and observe that $S_{s-1,L}^\theta \overset{d}{=} \frac{1}{2}(S_{v_1}^\theta + S_{v_2}^\theta)$, where  the definition of $S_{s-1,L}^\theta$ is given in the beginning of the proof.  Continuing similarly as (\ref{eq:stationary and gwc}, \ref{eq:stationary and delayed running time}), we get that
	\begin{equation}\label{eq:stationary dcp on gwc and running time}
	\begin{split}
	\nu_{{\mathcal{S}}}'(\zero)
	&=
	\frac{1}{1+(D+2)\lambda \theta  \E [\widetilde{S}^\theta|\mathcal{S}]};\\
	\nu_{{\mathcal{T}_{u_i}}}(\zero)
	&=
	\frac{1}{1+\lambda \E [S_{L-1}^\theta|\mathcal{T}_{u_i} ]};\\
	\nu_{\widetilde{\mathcal{S}}'}(\zero)
	&=
	\frac{1}{1+2\lambda \E[S_{s-1,L}^\theta|\widetilde{\mathcal{S}}] },
	\end{split}
	\end{equation}
	where the additional factor of $\theta$ in the first identity comes from the fact that $r(\zero;\mathcal{S})=1$ in $\ddp^{\lambda,\theta}_{\rho,v}(\mathcal{S})$. Plugging  these into (\ref{eq:prod dcp on gwc vs total dcp}) and using (\ref{eq:recursion dcp on gwc}), we obtain that
	\begin{equation*}
	\E [S_{s,L}^\theta|D] \leq \E\left[\left.\widetilde{S}_{s,L}^\theta\right|D\right]
	\leq
	\frac{1}{\theta} (1+\lambda \E[S_{L-1}^\theta])^D (1+2\lambda \E[S_{s-1,L}^\theta]).
	\end{equation*}
	Arguing similarly as Lemma \ref{lem:cp on gwc} and Proposition \ref{prop:dcp on gw}, we deduce that there exist constants $C,\lambda_0>0$ depending on $\mu$ such that for all $\lambda\leq \lambda_0$ and $s,L$ with $s\geq2$, $\E[S_{s,L}^\theta]\leq 3/\theta$ for $\theta = C\lambda$. Setting $C$ to satisfy $C\geq 3$, and applying this to the right-hand side of the above equation (which is written in terms of $(s-1,L)$) gives the desired conclusion.
\end{proof}

\begin{proof}[Proof of Lemma \ref{lem:dcp on egw}] To finish the proof of Lemma \ref{lem:dcp on egw}, we argue similarly as Proposition \ref{prop:cp on egw}. Namely, we establish the result for $\egw(\widetilde{\mu}^\sharp;l,s)_L$ and then extend it to the general case $\egw(\mu^\sharp, \widetilde{\mu}^\sharp;l,s)_L$. In both steps, we appeal to the same technique as Proposition \ref{prop:dcp on gw}, which is simpler than what is done here for the GWC-processes. We omit the details due to similarity.
\end{proof}

\subsection{Proof of Lemma \ref{lm:degree:reduction}, Items \ref{item:config}-\ref{item:branchingrate}}\label{app:preprocessing}

Item \ref{item:config} follows from the definition of $G_n$ that its edges are obtained from a uniformly chosen perfect matching of the half-edges.
%	Item \ref{item:muj} follows from the fact that $\E_{D\sim \mu} e^{\ep D/2} = \infty$ and $j$ can be chosen arbitrarily large.

For Item \ref{item:sumdi}, choose $j$ large enough such that 
$$\delta:= \E_{D\sim \mu} D\one_{D\ge 2j+1}\le \ep/4.$$
For each vertex $v\in G_n$, consider the random variable 
$$X_v := \deg_{G_n}(v)\one_{\deg_{G_n}(v)\ge 2j+1}.$$ 
These random variables are independent with mean $\delta$ and variance bounded by the second moment of $\mu$. By Chebyshev's inequality, \textsf{whp}
$$\sum_{v\in G_n} X_v\le \ep n/2.$$
Thus, \textsf{whp}, the total number of removed half-edges is at most $2\sum_{v\in G_n} X_v\le \ep n$. So is the number of removed vertices. Thus, $\bar n\ge (1-\ep )n$. Applying Chernoff inequality to the random variables $\bar X_v := \deg_{G_n}(v)\one_{\deg_{G_n}(v)\le 2j}$ we obtain that \textsf{whp},
$$\sum _{v\in G_n} \bar X_{v} \in (1-\ep, 1+\ep) nd.$$
Combining this with the fact that the total number of deleted half-edges is at most $\ep nd$ \textsf{whp}, we get
$$(1-2\ep) nd\le \sum _{v\in G_n} \bar X_{v} -  \ep n\le d_1+\dots +d_{\bar n}\le \sum _{v\in G_n} \bar X_{v} \le (1+\ep) nd$$
completing the proof of Item \ref{item:sumdi}.

To prove Item \ref{item:branchingrate}, let $\ep' = \min\{\ep, \ep(b-1)\}$. Let $k_0$ be a large constant such that for all $h\ge k_0$, the branching rate of $\mu_{[0, h]}\in (1-\ep', 1+\ep')b$, namely
\begin{equation}\label{eq:b:ep}
\frac{\E_{D\sim \mu} D(D-1)\one_{0\le D\le h}}{\E_{D\sim \mu} D \one_{0\le D\le h}}\in  (1-\ep', 1+\ep')b,
\end{equation}
\begin{equation}\label{eq:b:ep:2}
\begin{split}
&\E_{D\sim \mu} D\one_{0\le D\le k_0}\in (1-\ep', 1+\ep') d,\\
\text{and }\quad &\E_{D\sim \mu} D^{2}\one_{0\le D\le k_0}\ge \frac{2}{\ep'} \E_{D\sim \mu} D^{2}\one_{k_0<D}.\end{split}	\end{equation}

Let $k\ge k_0$ be such that
\begin{equation}\label{eq:2moment}
\E _{D\sim \mu}  D^{2}\one_{k_0<D\le k} \ge \frac{2}{\ep'}\E _{D\sim \mu}  D^{2}\one_{k<D}.
\end{equation}
which $k$ exists because of the boundedness of $\E_{D\sim \mu} D^{2}$.  Note that \eqref{eq:2moment} implies that
\begin{equation}\label{eq:2moment:2}
\E _{D\sim \mu} \left ( D\one_{k_0<D\le k}\right )\ge \frac{2}{\ep'}\E _{D\sim \mu} \left ( D \one_{k<D}\right ).
\end{equation}

We now show that for all constant $j\ge k$, $\bar b\in (1-\ep', 1+\ep') b$ \textsf{whp}. Let $E_l$ and $\bar E_l$ be the number of half-edges attached to vertices of degree $l$ in $G_n$ and $\bar G_n$ respectively. We need to show that \textsf{whp},
\begin{equation}\label{eq:bar b:b}
\bar b = \frac{\sum_{l=0}^{2j} (l-1)\bar E_l}{\sum_{l=0}^{2j} \bar E_l} \in (1-\ep, 1+\ep) b.
\end{equation}

Since $j$ is a constant and $E_l = l \sum_{v\in G_n}\one_{\deg_{G_n} (v)=l}$, by Chernoff inequality, \textsf{whp} we have,
\begin{equation}\label{eq:e':e}
\begin{split}
&\sum_{l=0}^{2j} E_l  \in  (1-\ep', 1+\ep') \sum_{l=0}^{2j} l \mu(l) n \quad \text{ and } \\
&\sum_{l=0}^{2j} (l-1)E_l  \in  (1-\ep', 1+\ep')\sum_{l=0}^{2j} l(l-1)\mu(l) n.
\end{split}
\end{equation}
This together with \eqref{eq:b:ep} and \eqref{eq:b:ep:2} give
$$\frac{\sum_{l=0}^{2j} (l-1)E_l}{\sum_{l=0}^{2j} E_l} \in (1-3\ep', 1+3\ep') b \quad \text{and}\quad \sum_{l=0}^{2j} l \mu(l) \in (1-\ep', 1+\ep')d.$$
Since the total number of removed half-edges is at most $\ep' n d$ \textsf{whp},
$$ \sum_{l=0}^{2j} \bar E_l  \in  \sum_{l=0}^{2j} E_l + (-\ep', 0) nd\subset (1-2\ep', 1+2\ep') \sum_{l=0}^{2j} E_l .$$
From this and \eqref{eq:e':e}, \eqref{eq:bar b:b} reduces to proving that
$$\sum_{l=0}^{2j} l \bar E_l\in (1-\ep', 1+\ep')\sum_{l=0}^{2j} l E_l.$$
The upper bound is straightforward. To prove the lower bound, let  $N_l$ be the number of vertices of degree $l$ in $G_n$. Since the number of deleted half-edges is at most $\sum_{l=2j+1}^{\infty} l N_l $, we have by Markov's inequality, \eqref{eq:2moment:2}, and Chernoff inequality, \textsf{whp}
\begin{equation}\label{eq:e:e':2:0}
\begin{split}
\sum_{l=0}^{2j} (E_l-\bar E_l)&\le  \sum_{l=2j+1}^{\infty} l N_l \le \frac{1}{\ep'} \E \sum_{l=2j+1}^{\infty} l N_l\le \frac{1}{2} \E \sum_{l=k_0+1}^{2j} l N_l \le \sum_{l=k_0+1}^{2j} l N_l.
\end{split}
\end{equation}
Thus, we have
\begin{equation}\label{eq:e:e':2}
\sum_{l=0}^{2j} l (E_l-\bar E_l) \le \sum_{l=k_0+1}^{2j} l^{2} N_l\le \ep' \sum_{l=0}^{2j} l^{2} N_l = \ep' \sum_{l=0}^{2j} l E_l
\end{equation}
where the first inequality follows from \eqref{eq:e:e':2:0} and the fact that the left-hand side of \eqref{eq:e:e':2} is largest when the deleted half-edges counted in $\sum_{l=0}^{2j} (E_l-\bar E_l)$ are drawn from vertices of highest degrees possible and the second inequality follows from the Chernoff inequality and \eqref{eq:b:ep:2}. That completes the proof of \eqref{eq:bar b:b} and hence Item \ref{item:branchingrate}. \qed

\end{document}